\documentclass[12pt,reqno]{amsart}
\usepackage{amsmath,amssymb,latexsym, amsthm, amscd, mathrsfs, stmaryrd} 
\usepackage[linktocpage=true]{hyperref}
\hypersetup{colorlinks,linkcolor=blue,urlcolor=cyan,citecolor=blue}

\usepackage[all]{xy}
\usepackage{hyperref}
\usepackage{color}
\newcommand{\nc}{\newcommand}
\nc{\browntext}[1]{\textcolor{brown}{#1}}
\nc{\greentext}[1]{\textcolor{green}{#1}}
\nc{\redtext}[1]{\textcolor{red}{#1}}
\nc{\bluetext}[1]{\textcolor{blue}{#1}}
\nc{\brown}[1]{\browntext{ #1}}
\nc{\green}[1]{\greentext{ #1}}
\nc{\red}[1]{\redtext{ #1}}
\nc{\blue}[1]{\bluetext{ #1}}
\nc{\zb}[1]{\redtext{From zb: #1}}

\newtheorem{theorem}{Theorem}  [section]
\newtheorem{corollary}[theorem]{Corollary}
\newtheorem{lemma}[theorem]{Lemma}

\newtheorem{proposition}[theorem]{Proposition}

\theoremstyle{remark}
\newtheorem{remark}[theorem]{Remark}

\numberwithin{equation}{section}

\def \haX{\widehat{X}}

\def \bt{\mathbf t}
\def \bn{\mathbf n}
\def \bl{\mathbf l}
\def \bL{\mathbb L}

\def \bh{H}

\newcommand{\cc}{{\mathcal C}}

\def \bd{\mathbf d}

\def\M{\mathcal{M}}

\def \co{\mathcal O}

\def \ch{{\mathcal H}}
\def \cm{{\mathcal M}}
\def \ct{\widetilde{{{\mathcal T}}}}
\def \bS{\mathbf S}

\def \tMH{{\cs\cd\widetilde{\ch}}}

\def \bm{\mathbf{m}}

\renewcommand{\mod}{\operatorname{mod}\nolimits}

\numberwithin{equation}{section}

\newtheorem{innercustomthm}{{\bf Theorem}}
\newenvironment{customthm}[1]
{\renewcommand\theinnercustomthm{#1}\innercustomthm}
{\endinnercustomthm}

\renewcommand{\ker}{\operatorname{Ker}\nolimits}

\newcommand{\Hom}{\operatorname{Hom}\nolimits}

\newcommand{\Aut}{\operatorname{Aut}\nolimits}

\newcommand{\Sym}{\operatorname{Sym}\nolimits}
\newcommand{\Ext}{\operatorname{Ext}\nolimits}

\newcommand{\add}{\operatorname{add}\nolimits}

\newcommand{\mbf}{\mathbf}

\newcommand{\mrm}{\mathrm}

\newcommand{\End}{\mrm{End}}
\newcommand{\rank}{\mrm{rank}}
\newcommand{\de}{\delta}

\newcommand{\N}{\mathbb N}
\newcommand{\bbZ}{\mathbb Z}

\newcommand{\ov}{\overline}
\newcommand{\qbinom}[2]{\begin{bmatrix} #1\\#2 \end{bmatrix} }
\newcommand{\Q}{\mathbb Q}
\newcommand{\sll}{\mathfrak{sl}}

\newcommand{\U}{\mbf U}

\newcommand{\F}{\mathbb F}

\newcommand{\arxiv}[1]{\href{http://arxiv.org/abs/#1}{\tt arXiv:\nolinkurl{#1}}}

\newcommand{\Y}{\mathbb Y}
\newcommand{\Z}{\mathbb Z}

\def \X{\mathbb X}
\def \fg{\mathfrak{g}}
\def \bU{{\mathbf U}}
\newcommand{\tK}{K}
\def \I{\mathbb{I}}
\def \bv{v}

\newcommand{\tUD}{{}^{\text{Dr}}\tU}
\newcommand{\UD}{{}^{\text{Dr}}\U}
\def \nua{a}

\renewcommand{\P}{\mathbb{P}}

\newcommand{\sqq}{{\bf v}}

\newcommand{\coker}{\operatorname{Coker}\nolimits}

\newcommand{\tU}{\widetilde{\mathbf U}}

\def \cl{L}

\def \R{\mathbb{R}}
\def \SS{\mathbb{S}}

\def \cc{\mathcal C}
\def\ca{\mathcal A}
\newcommand{\Iso}{\operatorname{Iso}\nolimits}
\renewcommand{\Im}{\operatorname{Im}\nolimits}
\newcommand{\res}{\operatorname{res}\nolimits}
\newcommand{\iH}{{}^\imath\widetilde{\ch}}

\newcommand{\coh}{\operatorname{coh}\nolimits}

\newcommand{\rep}{\operatorname{rep}\nolimits}

\def \PL{\mathbb{P}^1_{\bfk}}

\def \scrf{\mathscr F}
\def \scrt{\mathscr T}
\def \P{\mathbb P}

\def \cs{{\mathcal{S}}}
\def \ch{\mathcal H}
\def \cd{\mathcal D}
\def\bfk{\mathbf{k}}

\def \bp{\mathbf p}

\def \II{\I_0}
\def \Lg{L\fg}

\def \cn{\mathcal N}

\newcommand{\gr}{\operatorname{gr}\nolimits}

\def \cR{\mathcal R}

\def \bla{\boldsymbol{\lambda}}
\def \blx{x}

\def \vx{\vec{x}}
\def \vc{\vec{c}}
\def \vw{\vec{\omega}}
\def \bX{{\mathbb{X} }} 
\def \c0X{{\coh_0(\bX)}}
\def \cX{{\coh(\bX)}}

\def \Pk{{\mathbb{P}_{\bfk}^1}}

\allowdisplaybreaks

\usepackage{tikz}
\usepackage{tabularx}
\usepackage{chngcntr} 
\counterwithin{figure}{section}
\counterwithin{table}{section}
\usetikzlibrary{
	cd,
	shapes,
	arrows,
	positioning,
	decorations.markings,
	decorations.pathmorphing,
	circuits.logic.US,
	circuits.logic.IEC,
	fit,
	calc,
	plotmarks,
	matrix
}

\begin{document}

	\title[Twisted quantum loop
	algebras via semi-derived Ringel-Hall algebras]
	{Twisted quantum loop
		algebras via semi-derived Ringel-Hall algebras}
	
	\author[Ming Lu]{Ming Lu}
	\address{Department of Mathematics, Sichuan University, Chengdu 610064, P.R.China}
	\email{luming@scu.edu.cn}

	\author[Shiquan Ruan]{Shiquan Ruan}
	\address{ School of Mathematical Sciences,
		Xiamen University, Xiamen 361005, P.R.China}
	\email{sqruan@xmu.edu.cn}

	
	\begin{abstract}
		Twisted quantum loop algebras are a generalization of twisted quantum affine algebras in Drinfeld new presentation. The Hall algebras of Geigle--Lenzing's weighted projective lines are used to realize  (untwisted) quantum loop algebras of simply-laced type associated to star-shaped graphs by Schiffmann and Dou--Jiang--Xiao. 
		In this paper, we use the semi-derived Ringel-Hall algebras of more general weighted projective lines to realize the twisted quantum loop algebras associated to the valued star-shaped graphs, including the twisted quantum affine algebras in
		Drinfeld new presentation.
	\end{abstract}
	
	\maketitle
	\setcounter{tocdepth}{1}
	\tableofcontents

	\section{Introduction}
	
	\subsection{Background}
	
	In addition to the Serre presentation given by Drinfeld--Jimbo, quantum affine algebras have a current realization (called Drinfeld new presentation) \cite{Dr88,Be94,Da15}, which play a crucial role in (algebraic and geometric) representation theory; see \cite{CP91,CP98,FH11,Her05}. 
	Quantum loop algebras are defined for any Kac-Moody algebras, as a generalization of quantum affine algebras; see \cite{GKV95,Sch04,Her05,CJKT20,CJKT23}. Similar to affine Lie algebras, quantum loop algebras have two types: twisted and untwisted types. The Drinfeld new presentation of twisted quantum loop algebras is much more complicated than the untwisted type. The interest of the twisted case resides not only in that it is a generalization of the untwisted frame, it also appears quite naturally while studying the untwisted setting, due to the fact that the
	transposition of matrices establishes a duality among the aﬃne Cartan matrices through which untwisted Cartan matrices can correspond
	to twisted ones; see e.g. \cite{CP98,FH11,Da12}. 
	
	Inspired by the Hall algebra realization of the current half of quantum affine $\sll_2$ via the projective line \cite{Ka97,BKa01}, Schiffmann \cite{Sch04} developed 
	the Hall algebra of a weighted projective line 
	to realize a half part of a (untwisted) simply-laced quantum loop algebra. This was then upgraded to the whole quantum loop algebra via Drinfeld double technique  \cite{DJX12, BS12, BS13}, in particular, Dou--Jiang--Xiao \cite{DJX12} found out a collection of generators of the Drinfeld double Hall algebra of the weighted projective line and verify  them to satisfy all Drinfeld relations.
	
	The geometric realization given in \cite{Sch04,DJX12} is restricted to (untwisted) quantum loop algebras of simply-laced type. This naturally raises the question of obtaining a Hall algebra realization of general quantum loop algebras. Furthermore, the construction of Drinfeld double Hall algebras constitutes an algebraic technique with numerous variations, which arise from the use of different Hopf pairings (see \cite{Gr95, X97, DJX12, BS12, BS13} for details). Notably, the Hall algebra of a weighted projective line is merely a topological bialgebra, since its comultiplication must take values in a completed space. This necessitates the introduction of a new class of Hall algebras.
	
	Bridgeland \cite{Br13} has realized the whole quantum group (in Serre presentation) via the Hall algebra of $2$-periodic
	complexes, building on the classic construction of Ringel-Hall algebra of a quiver which
	realizes half a quantum group \cite{Rin90, Gr95}.
	In order to generalize Bridgeland's construction to arbitrary hereditary abelian categories, especially the categories of nilpotent representations of arbitrary quivers and the categories of coherent sheaves over projective and weighted projective lines, the first author and Peng \cite{LP16} introduced the  semi-derived Ringel-Hall algebras (compared with semi-derived Hall algebras defined in \cite{Gor18}).
	
	Semi-derived Ringel-Hall algebras are isomorphic to Bridgeland's Hall algebras if the hereditary abelian categories have enough projective objects, and also  isomorphic to the Drinfeld double Hall algebras. So the Hall algebra realization of quantum groups in their Serre presentation can be achieved by using the  semi-derived Ringel-Hall algebras.
	It also implies a geometric realization of simply-laced quantum loop algebras based on the works \cite{Sch04, DJX12} by considering  semi-derived Ringel-Hall algebras of the weighted projective lines; see \cite[Remark 4.16]{LP16}. 
	
	\subsection{Goal}
	
	The (ordinary) weighted projective lines considered in \cite{Sch04,DJX12,BS13} are the ones with the degrees of weighted points being $1$, and their Hall algebras are used to realize the (untwisted) simply-laced quantum loop algebras. We also have (more general) weighted projective lines with the degrees of weighted points arbitrary \cite{Ku09}. The goal of the paper is to use the semi-derived Ringel-Hall algebras of this general setting to realize twisted quantum loop algebras, in particular, including  twisted quantum affine algebras in Drinfeld new presentation; see Table \ref{tab1}. Compared with \cite{Sch04,DJX12}, we need to do some much more difficult computations since the current categories of coherent sheaves are more complicated; see \S\ref{sec:serre} and Appendix \ref{Appendix C}.
	

	\subsection{Main results}
	Throughout this paper, $\bfk=\F_q$ is the finite field and $\sqq=\sqrt{q}$. 
	Let $\bd=(d_1,d_2,\dots,d_\bt)$ and $\bp=(p_1,p_2,\dots,p_\bt)$ be $\bt$-tuples of positive integers. Let $\X=\X(\bp,\bd,\bla)$ be the weighted projective line of type $(\bp,\bd,\bla)$, that is the projective line $\P_\bfk^1$  with
	a finite number of marked points $\bla_1,..., \bla_\bt$ of degrees $d_1,\dots,d_\bt$ with attached weights $p_1,...,p_\bt$. Let $\coh(\X)$ be the category of coherent sheaves over $\X$. 
	
	Associated to $\X$, we construct a star-shaped graph $\Gamma=\Gamma_{p_1,\dots,p_\bt}$; see 
	\eqref{star-shaped}. Let $\fg$ be the (symmetrizable) Kac-Moody Lie algebra associated to the valued graph $(\Gamma,\bd)$, 
	and $\Lg$ be its loop algebra, which is a generalization of (twisted) affine Lie algebras; see \cite{MRY90,CJKT20}. Denote the set of vertices in $\Gamma$ by $\II$. 
	The (Drinfeld double) quantum loop algebra $\tU_v(\Lg)$ is generated by $X^{\pm}_{\star,l}$,  $X^\pm_{[i,j],l}$, $H_{\star,\pm m}$, $H_{[i,j],\pm m}$, $\psi_{\star,m}$, $\psi_{[i,j],m}$, $\varphi_{\star,-m}$, $\varphi_{[i,j],-m}$ and the invertible elements $K_\star$, $K'_\star$, $K_{[i,j]}$, $K'_{[i,j]}$, $C,C'$ for $l\in\Z,m>0$,  $1\leq i\leq \bt$, $1\leq j\leq p_i-1$, subject to \eqref{DR1}--\eqref{DR8}. The quantum loop algebra $\U_v(\Lg)$ can be described as a quotient algebra of $\tU_v(\Lg)$ modulo the ideal generated by
	$$CC'-1,\qquad K_\mu K_\mu'-1, \,\,\forall \mu\in\II.$$
	The quantum loop algebra  introduced by Drinfeld \cite{Dr88} (see also \cite{Be94,Da12}) is obtained from $\U_v(\Lg)$
	by adding two central generators $C^{\pm\frac{1}{2}}$ such that $C^{\frac{1}{2}}C^{-\frac{1}{2}}=1$ and $(C^{\frac{1}{2}})^2=C$. 
	In this paper, we focus mainly on the Drinfeld double quantum loop algebra $\tU_v(\Lg)$.


	Let $\X_\bfk$ be the weighted projective line over a finite field $\bfk=\F_q$ associated to $(\Gamma,\bd)$. Let $\tMH(\X_\bfk)$ be the  semi-derived Ringel-Hall algebra of $\coh(\X_\bfk)$. Our main result (see Theorem \ref{thm:morphi}) is to construct the $\Q(\sqq)$-algebra homomorphism
	\begin{align}
		\Omega: \tU_\sqq(\Lg)\longrightarrow \tMH(\X_\bfk),
	\end{align}
	which is defined on generators.
	To show that $\Omega: \tU_\sqq(\Lg) \rightarrow \tMH(\X_\bfk)$ is a homomorphism, we must verify the Drinfeld relations \eqref{DR1}--\eqref{DR8} for the twisted quantum loop algebra $\tU_v(\Lg)$, which is more difficult than  \cite{Sch04,DJX12} for simply-laced type.
	
	The counterparts in $\tMH(\X_\bfk)$ of the relations at the $\star$ point follow from the realization of $\tU_v(\widehat{\mathfrak{sl}}_2)$ by $\tMH(\P^1_\bfk)$ (see Proposition \ref{prop:Realizationsl2}), and the counterparts in $\tMH(\X_\bfk)$ of the relations involving all the vertices $[i,j]$ follow from the above definitions and the fact any two torsion sheaves supported at distinct points has zero Hom and $\Ext^1$-spaces. In order to check the remaining Drinfeld relations in $\tMH(\X_\bfk)$, the key part is to verify the relations between $\star$ and $[i,1]$ for $1\leq i\leq \bt$, which consists of plenty of highly non-trivial computations in \S\ref{sec:relationstari1}--\S\ref{sec:serre}. In fact, some of them are known by \cite{Sch04,DJX12}. It is remarkable that the verification of  Serre relation \eqref{DR8} in $\cs\cd\widetilde{\ch}(\X_\bfk)$ can be reduced to $\X_\bfk$ of weight type  $\left(\begin{array}{ccc}
		2 \\ d 
	\end{array}\right)$, however, it 
	is very difficult, which occupies \S\ref{sec:serre} and Appendix \ref{Appendix C}; see Remark \ref{remark
		on type (2,d)}.


	\subsection{Perspectives}
	
	This work is an application of semi-derived Ringel-Hall algebras. It is interesting to describe the composition subalgebra of the semi-derived Ringel-Hall algebra of a weighted projective line, 
	and give a PBW basis for the quantum loop algebra via coherent sheaves; cf. \cite{BS13}. 
	
	We shall also show that the morphism  from the quantum loop algebra to the semi-derived Ringel-Hall algebra is injective for $\fg$ of finite or affine type (if $\fg$ is simply-laced, the morphism from the ``positive" part of $\tU_v(\Lg)$ to the Hall algebra of $\X_\bfk$ is known to be injective by \cite{Sch04}). 
	We expect the morphism $\Omega: \tU_\sqq(\Lg)\rightarrow \tMH(\X_\bfk)$ to be injective for arbitrary Kac-Moody algebra $\fg$.
	It is also interesting to study the semi-derived Ringel-Hall algebras of higher genus curves, in particular, of elliptic curves. 
	
	We can further show that the Drinfeld--Beck isomorphism of the quantum group of type  $A_n^{(1)}$ or $D_{n+1}^{(2)}$ in two (Serre vs Drinfeld) presentations is induced from derived equivalence of the categories underlying the two (quivers vs weighted projective lines) Hall algebra realizations. This generalizes the result in \cite{BS12} for affine $\mathfrak{sl}_2$ by using Drinfeld double Hall algebras.
	
	In order to realize the Drinfeld new presentation of quantum affine algebras of untwisted BCFG type, we shall develop the theory of weighted projective lines further, by introducing exceptional curves over finite fields; see \cite{Ku09} for exceptional curves over real number field. Unlike weighted projective lines over finite field in this paper, the coordinate rings of exceptional curves are non-commutative, which makes this generalization highly non-trivial.

	\subsection{Organization}
	
	This paper is organized as follows. Section \ref{sec:Excep-curves} is devoted to reviewing the weighted projective lines and their coherent sheaves, and  \S\ref{sec:Quanloop} is for preliminaries on quantum loop algebras and their Drinfeld new presentations. 
	In \S\ref{sec:Semi-derived}, we review the semi-derived Ringel-Hall algebras for the categories of $2$-periodic complexes.

	We realize the quantum affine algebra of $\widehat{\mathfrak{sl}}_2$ via the semi-derived Ringel-Hall algebra of the projective line in \S\ref{sec:P1}, and  realize the quantum affine algebra of type A via the semi-derived Ringel-Hall algebra of the cyclic quiver in \S\ref{sec:cyclic}. We formulate the morphism $\Omega: \tU_\sqq(\Lg)\rightarrow \tMH(\X_\bfk)$ and state the main result in \S\ref{sec:hom}. The proof of $\Omega$  being an algebra homomorphism consists of Sections~\ref{sec:relationstari1}--\ref{sec:serre} and Appendix \ref{Appendix C}. 


	\subsection{Acknowledgments}
	We deeply thanks Fulin Chen for  helpful discussions on quantum loop algebras and (twisted) quantum affine algebras. ML is partially supported by the National Natural Science Foundation of China (No. 12171333, 12261131498). 
	SR is partially supported by Fujian Provincial Natural Science Foundation of China (No. 2024J010006) and
	the National Natural Science Foundation of China (No. 12271448).
	
	\section{Weighted projective lines and coherent sheaves}
	\label{sec:Excep-curves}
	
	Throughout $\bfk=\F_q$, a finite field of $q$ elements. In this section, we shall review the weighted projective lines over $\bfk$ and their coherent sheaves given in \cite{GL87,Ku09}. 
	
	\subsection{Coordinate ring}
	\label{subsec:coord-ring}
	
	Let $\bS=\bfk[T_0,T_1]$ be the coordinate ring of the classical projective line $\Pk$ over the field $\bfk$. 
	Let $\bla=(\bla_1, \bla_2,\cdots, \bla_{\bt})$ be a sequence of pairwise distinct homogeneous prime ideals $\bla_i=(f_i)$, where $f_i\in\bS\; 
	(1\leq i\leq \bt)$ are irreducible polynomials with $\deg(f_i)=d_i$. Denote by $\bd=(d_1,d_2,\cdots, d_{\bt})$.
	
	Let $\bp=(p_1,p_2,\cdots,p_{\bt})$ be a $\bt$-tuple of integers with $p_i\geq 1$, called the weight sequence. 
	Let $\bL(\bp,\bd)$ denote the rank one abelian group on generators $\vx_1,\vx_2,\cdots,\vx_{\bt}, \vc$ with relations $$p_i\vx_i=d_i\vc;\quad 1\leq i\leq \bt.$$ 
	Hence any element $\vx\in\bL(\bp,\bd)$ can be uniquely written in normal form as follows:
	$$\vx=\sum\limits_{1\leq i\leq \bt}l_i\vx_i+l\vc, \quad 0\leq l_i\leq p_i-1,\; l\in\bbZ.$$
	The element $\vc$ is called the canonical element, and $$\vw:=\sum\limits_{1\leq i\leq \bt}(p_i-1)\vx_i-2\vc
	=(\sum\limits_{1\leq i\leq \bt}d_i-2)\vc-\sum\limits_{1\leq i\leq \bt}\vx_i$$ is called the dualizing element of $\bL(\bp,\bd)$. 
	
	Denote by $p={\text{l.c.m}(p_1,p_2,\cdots,p_{\bt})}$, the least common multiple of $p_i$'s.
	Define a linear map on $\bL(\bp,\bd)$ via 
	$$\mathfrak{d}: \bL(\bp,\bd)\rightarrow \bbZ; \quad \vx_i\mapsto \frac{p d_i}{p_i},\;\; 1\leq i\leq \bt.$$
	Then we have $\mathfrak{d}(\vc)=p$. 
	
	Let $$\bS=\bS(\bp,\bla):=\bfk[T_0,T_1, X_1,X_2,\cdots,X_{\bt}]/(X_i^{p_i}-f_i:\;1\leq i\leq \bt).$$ Then $\bS(\bp,\bla)$ is $\bL(\bp,\bd)$-graded in the sense that $\deg(X_i)=\vx_i$ and $\deg(T_0)=\deg(T_1)=\vc$.
	
	The $\bL(\bp,\bd)$-graded nonzero principle prime ideals of $\bS(\bp,\bla)$ consist of 
	\begin{itemize}
		\item[-] Exceptional prime ideals: $(X_i)$ for $1\leq i\leq \bt$;
		\item[-] Ordinary prime ideals: all the other prime ideals of $\bfk[T_0,T_1]$ except  $\bla_i$ for $1\leq i\leq \bt$.
	\end{itemize}
	

	\subsection{Category of coherent sheaves}
	
	The weighted projective line $\bX:=\bX(\bp,\bd, \bla)$ of type $(\bp,\bd, \bla)$ is defined as the stack $\text{Proj}^{\bL(\bp,\bd)}\bS(\bp,\bla)$.
	The classification of the closed points in $\X$ is as follows: the \textit{exceptional points}, corresponding to the exceptional prime ideals $(X_i)$'s; and the \textit{ordinary points}, corresponding to the ordinary prime ideals. 
	
	The category of coherent sheaves over $\bX$ is defined by Serre's construction $$\cX:=\frac{\mod^{\bL(\bp,\bd)}\bS(\bp,\bla)}{\mod^{\bL(\bp,\bd)}_0\bS(\bp,\bla)}.$$
	All the line bundles in $\cX$ has the form $\co(\vx)$ for $\vx\in\bL(\bp,\bd)$.
	There is a natural quotient functor from $\mod^{\bL(\bp,\bd)}\bS(\bp,\bla)$ to $\cX$, such that the image of the free module $\bS:=\bS(\bp,\bla)$ serves as the structure sheaf $\co$ of $\cX$, and the degree shift $\bS(\vx)$ for any $\vx\in\bL(\bp,\bd)$ serves as the line bundle $\co(\vx)$. Moreover, we have $$\Hom(\co(\vx),\co(\vec{y}))\cong \Hom(\bS(\vx), \bS(\vec{y}))\cong \bS_{\vec{y}-\vx}.$$
	
	There is an autoequivalence $\tau=(\vw)$
	on $\cX$ such that the Serre duality 
	$$\Ext^1(X, Y)\cong D\Hom(Y, \tau X)$$
	holds functorially in $X, Y \in\cX$, where $D$ is the duality $\Hom_\bfk(-, \bfk)$.
	Moreover, $\cX$ has almost
	split sequences and the functor $\tau:\cX\to\cX$ serves as the Auslander-Reiten translation.
	
	Let $\scrf$ be the full subcategory of $\coh(\X)$ consisting of vector bundles, and $\scrt$ be the full subcategory consisting of all torsion sheaves. Then $\scrf$ and $\scrt$ are extension-closed, and any indecomposable sheaf is either a vector bundle or a torsion sheaf. Moreover, there are no non-zero homomorphisms from torsion sheaves to vector bundles.

	\subsection{Torsion sheaves}
	
	In order to describe the subcategory $\scrt$ of torsion sheaves, we shall introduce the representations of cyclic quivers.
	We consider the oriented cyclic quiver $C_n$ for $n\geq2$ with the vertex set $\Z_n=\{0,1,2,\dots,n-2,n-1\}$:
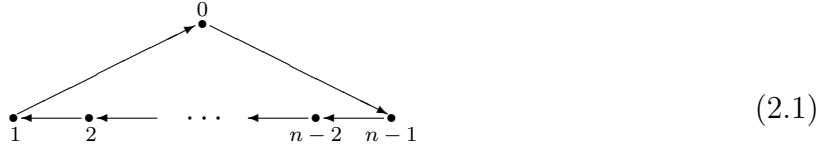
	\begin{center}\setlength{\unitlength}{0.5mm}
		\begin{equation}
			\label{fig:Cn}
			\begin{picture}(100,30)
				\put(2,0){\circle*{2}}
				\put(22,0){\circle*{2}}
				
				\put(82,0){\circle*{2}}
				\put(102,0){\circle*{2}}
				\put(52,25){\circle*{2}}
				
				\put(20,0){\vector(-1,0){16}}
				\put(40,0){\vector(-1,0){16}}
				\put(47.5,-2){$\cdots$}
				\put(80,0){\vector(-1,0){16}}
				\put(100,0){\vector(-1,0){16}}
				
				\put(54,24.5){\vector(2,-1){47}}
				\put(3,1){\vector(2,1){47}}
				
				\put(50.5,27){\tiny $0$}
				\put(1,-6){\tiny $1$}
				\put(21,-6){\tiny $2$}
				\put(75,-6){\tiny $n-2$}
				\put(95,-6){\tiny $n-1$}
			\end{picture}
		\end{equation}
		\vspace{-0.2cm}
	\end{center}
	Let $C_1$ be the Jordan quiver, i.e., the quiver with only one vertex and one loop arrow.
	
	Denote by $\rep^{\rm nil}_\bfk(C_n)$ the category of finite-dimensional nilpotent representations of $C_n$ over the field $\bfk$. Then there are $n$ simple representations $S_j$ ($0\leq j\leq n-1$) in $\rep^{\rm nil}_\bfk(C_n)$. Denote by $S_j^{(a)}$ the unique indecomposable representation with top $S_j$ and length $a$. Then $\{S_j^{(a)}\mid 0\le j\le n-1, a\in\N\}$ is  the set of isoclasses of indecomposable objects of $\rep^{\rm nil}_\bfk(C_n)$. For any $0\le j\le n-1$, define $S_j^{(\mu)}=\bigoplus_{i=1} ^{l}S_j^{(\mu_i)}$ for any partition $\mu=(\mu_1\geq \cdots\geq \mu_l)$. 
	

	Then the structure of $\scrt$ is described in the following.
	\begin{lemma}[\cite{GL87}; also see \cite{Sch04,Ku09}]
		\label{lem:isoclasses Tor}
		(1) The category $\scrt$ decomposes as a coproduct $\scrt=\coprod_{x\in\X} \scrt_{x}$, where $\scrt_{x}$ is the subcategory of torsion sheaves with support at $x$.
		
		(2) For any ordinary point $x$ of degree $d$, let $\bfk_{x}$ denote the residue field at $x$, i.e., $[\bfk_x:\bfk]=d$. Then $\scrt_{x}$ is equivalent to the category $\rep^{\rm nil}_{\bfk_{x}}(C_1)$.
		
		(3) For any exceptional point $\bla_i$ ($1\leq i\leq \bt$), let $\bfk_{i}$ denote the residue field at $\bla_i$. Then the category $\scrt_{\bla_i}$ is equivalent to $\rep^{\rm nil}_{\bfk_i}(C_{p_i})$.
	\end{lemma}

	For any ordinary point $\blx$ of degree $d$, let $\pi_{\blx}$ be the prime homogeneous polynomial corresponding to $\blx$. The multiplication by $\pi_{\blx}$ gives the following exact sequence
	$$0\longrightarrow \co\stackrel{\pi_{\blx}}{\longrightarrow} \co(d\vec{c})\longrightarrow S_{\blx}\longrightarrow0,$$
	where $S_{\blx}$ is the unique (up to isomorphism) simple sheaf in the category $\scrt_{\blx}$. Note that $\End(S_{\blx})\cong \bfk_{\blx}$ and $S_{\blx}(\vec{l})=S_{\blx}$ for any $\vec{l}\in \bL(\bp,\bd)$. 
	
	For any exceptional point $\bla_i$, multiplication by $X_i$ yields the short exact sequence
	$$0\longrightarrow \co((j-1)\vec{x}_i)\stackrel{X_i}{\longrightarrow} \co(j\vec{x}_i)\longrightarrow S_{ij}\longrightarrow0,\text{ for } 1\le j\le p_i;$$
	where $\{S_{ij}\mid  j\in\Z_{p_i}\}$ is a complete set of pairwise non-isomorphic simple sheaves in the category $\scrt_{\bla_i}$ for any $1\leq i\leq \bt$. Moreover, $\End(S_{ij})\cong \bfk_{i}$ and $S_{ij}(\vec{l})=S_{i,j+l_i}$ for any $\vec{l}=\sum_{1\leq i\leq \bt} l_i\vec{x}_i$.

	\subsection{Grothendieck groups and Euler forms}
	
	Denote by  $K_0(\cX)$ the Grothendieck group of $\cX$,
	and denote by $\widehat{X}$ the class
	in $K_0(\cX)$ of an object $X\in\cX$.
	
	Let $\delta=\widehat{S_{x}}$ for a rational ordinary point $x$, i.e., degree of $x$ equals 1. Then we have $$\delta=\widehat{\co(\vc)}-\widehat{\co} \quad\text{and} \quad\sum\limits_{j=0}^{p_i-1}\widehat{S_{ij}}=d_i\delta.$$
	Moreover, the Grothendieck group $K_0(\coh(\X))$ of $\coh(\X)$ satisfies
	\begin{align}
		K_0(\coh(\X))\cong \Big(\Z\widehat{\co}\oplus \Z\widehat{\co(\vec{c})}\oplus \bigoplus_{i,j}\Z\widehat{S_{ij}}\Big)/I,
	\end{align}
	where $I$ is the subgroup generated by $\{\sum_{j=1}^{p_i} \widehat{S_{ij}} +d_i\widehat{\co} -d_i\widehat{\co(\vec{c})}\mid i=1,\dots,\bt\}$; see \cite{GL87,Ku09}.

	The $K_0(\cX)$ is equipped with the Euler form $\langle-,-\rangle$. This bilinear form
	is defined on classes of objects $X, Y$ in $\cX$ by
	$$\langle \widehat{X}, \widehat{Y}\rangle = \dim_{\bfk} \Hom (X, Y ) -\dim_{\bfk} \Ext^1 (X, Y).$$
	
	Recall that 
	$$\Hom(\co, S_{i0})\cong \bfk_{i}, \text{ and } \Ext^1(S_{i,j}, S_{i,j'})\cong \begin{cases}
		\bfk_{i} & \text{ if }j+1\equiv j'(\mod p_i)
		\\
		0&\text{ otherwise}.
	\end{cases}$$ 
	Then by definition, we have the following. 
	\begin{lemma}
		\label{lem:Euler-form}
		The Euler form $\langle-,-\rangle$ on $K_0(\coh(\X))$ is described as follows:
		\begin{align}
			&\langle \widehat{\co},\widehat{\co}\rangle=1,\quad \langle \widehat{\co},\delta\rangle=1,\quad  \langle\delta,\widehat{\co}\rangle=-1,
			\\
			&\langle \delta,\delta\rangle=0,\quad \langle \widehat{S_{ij}},\delta\rangle=0,\quad  \langle \delta, \widehat{S_{ij}}\rangle=0,
			\\
			&\langle \widehat{\co}, \widehat{S_{ij}}\rangle =\begin{cases} d_i& \text{ if } j=p_i, \\ 0& \text{ if }j\neq p_i , \end{cases}
			\qquad \langle \widehat{S_{ij}}, \widehat{\co}\rangle =\begin{cases} -d_i& \text{ if } j=1, \\ 0& \text{ if }j\neq 1 ,\end{cases}
			\\
			&\langle\widehat{ S_{ij}}, \widehat{S_{i',j'}}\rangle =\begin{cases} d_i& \text{ if } i=i', j=j',
				\\ -d_i& \text{ if }i=i', j\equiv j'+1 (\mod p_i),
				\\
				0& \text{ else. }\end{cases}
		\end{align}
	\end{lemma}
	
	
	
	

	The degree function on $K_0(\cX)$ (or $\cX$) is given by 
	$$\deg(\co(\vx))=\mathfrak{d}(\vx),$$
	In particular, $$\deg(\co)=0; \quad \deg(\co(\vc))=p; \quad \deg(S_{ij})=\frac{pd_i}{p_i}.
	$$
	
	The rank function on $K_0(\cX)$ is given by 
	$$\rank(\widehat{X})=\langle \widehat{X}, \delta\rangle,$$
	and $\rank (X) = \rank(\widehat{X})$ for each $X \in\cX$. For an indecomposable sheaf $X$ in $ \cX$, $\rank (X)=0$ if and only if $X\in\mathscr{T}$, and $\rank(X)>0$ if and only if $X\in\mathcal{F}$.

	If $\bd$ is trivial, i.e., $d_i = 1$\ for $1\leq i\leq \bt$, we write $\bL(\bp)$ and $\bS(\bp)$ instead. In this case, $\bS(\bp)$ are just the projective coordinate algebras of the ordinary weighted projective lines
	described in \cite{GL87,Sch04}.

	
	
	\subsection{Virtual genus and Trichotomy}
	
	The virtual genus of $\bX$ is defined as $$g_{\bX}=1+\frac{p}{2}\big(\sum\limits_{i=1}^{\bt}d_i(1-\frac{1}{p_i})-2\big).$$
	The weighted projective line $\bX$ is called of domestic, tubular or wild type if the virtual genus $g_{\bX}<1$, $g_{\bX}=1$ or $g_{\bX}>1$, respectively.
	By direct calculations we obtain (see \cite{Ku09})	
	\begin{itemize}
		\item[(1)]
		$g_{\bX}<1$ if and only if $\bp=(p_1,p_2),(2,2,n),(2,3,3),(2,3,4), (2,3,5)$ and $\bd$ is trivial,
		or
		$$\left(\begin{array}{ccc}
			\bp  \\
			\bd
		\end{array}\right)=
		\left(\begin{array}{ccc}
			n  \\
			2 
		\end{array}\right),
		\left(\begin{array}{ccc}
			2&n  \\
			2&1 
		\end{array}\right),
		\left(\begin{array}{ccc}
			2  \\
			3 
		\end{array}\right), 
		\left(\begin{array}{ccc}
			3&2  \\
			2&1 
		\end{array}\right);   \quad(\text{for } n\geq 2)
		$$\\
		\item[(2)]
		$g_{\bX}=1$ if and only if $\bp=(2,2,2,2),(3,3,3),(4,4,2), (6,3,2)$ and $\bd$ is trivial,
		or
		$$\left(\begin{array}{ccc}
			\bp  \\
			\bd
		\end{array}\right)=
		\left(\begin{array}{ccc}
			2&2&2  \\
			2&1&1 
		\end{array}\right),
		\left(\begin{array}{ccc}
			2&2  \\
			3&1 
		\end{array}\right),
		\left(\begin{array}{ccc}
			2&2  \\
			2&2 
		\end{array}\right),
		\left(\begin{array}{ccc}
			2 \\
			4 
		\end{array}\right),
		\left(\begin{array}{ccc}
			3&3  \\
			2&1 
		\end{array}\right),
		\left(\begin{array}{ccc}
			3 \\
			3 
		\end{array}\right),
		\left(\begin{array}{ccc}
			4&2  \\
			2&1 
		\end{array}\right);$$\\
		\item[(3)]
		$g_{\bX}>1$ for all the other cases.
	\end{itemize}
	
	\begin{remark}
		Our notation for the type 
		of a weighted projective line
		$\left(\begin{array}{ccc}
			\bp  \\
			\bd
		\end{array}\right)$ 
		corresponds to the notation $\left(\begin{array}{ccc}
			\bp  \\
			\bd\\
			\bd
		\end{array}\right)$ 
		in the paper \cite{Ku09}.
	\end{remark}
	
	\begin{remark}\label{remark
			on type (2,d)}
		For $\X$ of weight type $\left(\begin{array}{ccc}
			2 \\ d 
		\end{array}\right)$, observe that it is of domestic type if $1\leq d\leq 3$, it is of tubular type if $d=4$, and it is of wild type for any $d\geq 5$. When dealing with the realization of simply-laced quantum loop algebras via weighted projective lines, as shown in \cite{Sch04,DJX12,BS12,BS13}, it suffices to deal with $d=1$, in which case every indecomposable bundle is a line bundle. 
		For general $d$, the complexity increases significantly, and the explicit structure of $\coh(\X)$ remains unknown to date.
		
		In order to check the Serre relation \eqref{Serre relation in specail case}, or equivalently \eqref{eq:Serre2-reduced}, we need to investigate the extensions between two line bundles. 
		The strategy is highly non-trivial and quite novel, and is of particular interest in itself (see Section \ref{Mk and Nk} and Appendix \ref{Appendix C}). Furthermore, we believe the result will be useful for further investigations of the category $\coh(\X)$.
	\end{remark}

	\section{Quantum loop algebras}
	\label{sec:Quanloop}

	In this section, we review quantum groups and quantum loop algebras (which are a generalization of twisted quantum affine algebras in the Drinfeld new presentation).

	\subsection{Quantum groups}
	Let $C=(c_{ij})_{\I\times \I}$ be a symmetrizable generalized Cartan matrix (GCM) with $D={\rm diag}(d_i\mid i\in\I)$ being its symmetrizer. 
	Let $v$ be the quantum parameter, and $v_i=v^{d_i}$ for $i\in\I$. 
	Denote, for $r,m \in {\Bbb N}$,
	\[
	[r]_{v_i}=\frac{\bv_i^r-\bv_i^{-r}}{\bv_i-\bv_i^{-1}},
	\quad
	[r]^!_{v_i}:=\prod_{k=1}^r [k]_{v_i}, \quad \qbinom{m}{r}_{v_i} =\frac{[m]_{v_i}[m-1]_{v_i}\ldots [m-r+1]_{v_i}}{[r]_{v_i}^!}.
	\]
	For $A, B$ in a $\Q(v)$-algebra, we  denote $[A,B] =AB - BA$ and $[A,B]_{a} =AB -aBA$ for $a\in\Q(v)$. 
	
	Recall that the Kronecker delta symbol $\delta_{ij}$ and the Iverson bracket $\delta\{P\}$ as follow:
	\[
	\delta_{ij}= 
	\begin{cases}
		1, & \text{if } i=j,\\[2pt]
		0, & \text{if } i\ne j;
	\end{cases}\quad\qquad
	\delta\{P\} = 
	\begin{cases}
		1, & \text{if the statement } P \text{ is true}, \\[2pt]
		0, & \text{if the statement } P \text{ is false}.
	\end{cases}
	\]


	The Drinfeld double quantum group $\tU := \tU_\bv(\fg)$ is defined to be the $\Q(\bv)$-algebra generated by $E_i,F_i, \tK_i,\tK_i'$, $i\in \I$, where $\tK_i, \tK_i'$ are invertible, subject to the following relations:
	\begin{align}
		[E_i,F_j]= \delta_{ij} \frac{\tK_i-\tK_i'}{\bv_i-\bv_i^{-1}},  &\qquad [\tK_i,\tK_j]=[\tK_i,\tK_j']  =[\tK_i',\tK_j']=0,
		\label{eq:KK}
		\\
		\tK_i E_j=\bv_i^{c_{ij}} E_j \tK_i, & \qquad \tK_i F_j=\bv_i^{-c_{ij}} F_j \tK_i,
		\label{eq:EK}
		\\
		\tK_i' E_j=\bv_i^{-c_{ij}} E_j \tK_i', & \qquad \tK_i' F_j=\bv_i^{c_{ij}} F_j \tK_i',
		\label{eq:K2}
	\end{align}
	and the following quantum Serre relations (for $i\neq j \in \I$),
	\begin{align}
		& \sum_{r=0}^{1-c_{ij}} (-1)^r \left[ \begin{array}{c} 1-c_{ij} \\r \end{array} \right]_{v_i}  E_i^r E_j  E_i^{1-c_{ij}-r}=0,
		\label{eq:serre1} \\
		& \sum_{r=0}^{1-c_{ij}} (-1)^r \left[ \begin{array}{c} 1-c_{ij} \\r \end{array} \right]_{v_i}  F_i^r F_j  F_i^{1-c_{ij}-r}=0.
		\label{eq:serre2}
	\end{align}
	Note that $\tK_i \tK_i'$ are central in $\tU$ for all $i$.

	
	Analogously as for $\tU$, the quantum group $\bU$ is defined to be the $\Q(v)$-algebra generated by $E_i,F_i, K_i, K_i^{-1}$, $i\in \I$, subject to the relations \eqref{eq:KK}--\eqref{eq:serre2} with $\tK_i'$ replaced by $K_i^{-1}$. 
	In fact, $\bU$ is the quotient algebra $\tU$ modulo the two-sided ideal generated by $K_iK_i'-1$ for all $i\in\I$.
	For any $\alpha\in\Z\I$, we can define $K_\alpha,K'_{\alpha}$ naturally.

	\subsection{Star-shaped graphs and loop algebras}

	For $\bp=(p_1,\dots,p_\bt)\in\Z_{>0}^\bt$, let us consider the following star-shaped graph $\Gamma=\Gamma_{p_1,\dots,p_\bt}$:
	
	\begin{center}\setlength{\unitlength}{0.8mm}
		\begin{equation}
			\label{star-shaped}
			\begin{picture}(110,30)(0,35)
				\put(0,40){\circle*{1.4}}
				\put(2,42){\line(1,1){16}}
				\put(20,60){\circle*{1.4}}
				\put(23,60){\line(1,0){13}}
				\put(40,60){\circle*{1.4}}
				\put(43,60){\line(1,0){13}}
				\put(60,58.5){\large$\cdots$}
				\put(70,60){\line(1,0){13}}
				\put(88,60){\circle*{1.4}}

				\put(3,41){\line(4,1){13}}
				\put(20,45){\circle*{1.4}}
				\put(23,45){\line(1,0){13}}
				\put(40,45){\circle*{1.4}}
				\put(43,45){\line(1,0){13}}
				\put(60,43.5){\large$\cdots$}
				\put(70,45){\line(1,0){13}}
				\put(88,45){\circle*{1.4}}
				
				\put(19,30){\Large$\vdots$}
				
				\put(39,30){\Large$\vdots$}
				
				\put(87,30){\Large$\vdots$}
				
				\put(2,38){\line(1,-1){16}}
				\put(20,20){\circle*{1.4}}
				\put(23,20){\line(1,0){13}}
				\put(40,20){\circle*{1.4}}
				\put(43,20){\line(1,0){13}}
				\put(60,18.5){\large$\cdots$}
				\put(70,20){\line(1,0){13}}
				\put(88,20){\circle*{1.4}}
				
				\put(-4.5,39){$\star$}
				
				\put(16,62){\tiny$[1,1]$}
				
				\put(36,62){\tiny$[1,2]$}
				\put(81,62){\tiny$[1,p_1-1]$}

				\put(16.5,47){\tiny$[2,1]$}
				\put(36.5,47){\tiny$[2,2]$}
				\put(81,47){\tiny$[2,p_2-1]$}

				\put(16.5,16){\tiny$[\bt,1]$}
				\put(36.5,16){\tiny$[\bt,2]$}
				\put(81,16){\tiny$[\bt,p_\bt-1]$}

				
			\end{picture}
		\end{equation}
		\vspace{1cm}
	\end{center}
	As marked in the graph, the central vertex is denoted by $\star$. Let $J_1,\dots,J_n$ be the branches, which are subdiagrams of type $A_{p_1-1},\dots,A_{p_\bt-1}$ respectively. Denote by $[i,j]$ the $j$-th vertex in the $i$-th branch. These examples includes all finite-type Dynkin diagrams as well as the affine Dynkin diagrams of types 
	$A_2^{(2)}$, $A_5^{(2)}$,
	$C_2^{(1)}$, $D_4^{(1)}$, $D_4^{(3)}$,  $E_6^{(1)}$,
	$E_6^{(2)}$,
	$E_7^{(1)}$, $E_8^{(1)}$, $F_4^{(1)}$, $G_2^{(1)}$. 
	
	The set of vertices is denoted by $\II$. 
	Recall $\bd=(d_1,d_2,\cdots, d_{\bt})$ in
	\S\ref{subsec:coord-ring}. 
	We endow a valuation $\bd$ to $\Gamma$ by setting $d_{\star}=1$, and $d_{[i,j]}=d_i$ for each vertex $[i,j]\in\II$.
	The GCM $C=(c_{\mu\nu})_{\II\times \II}$ of the valued graph $(\Gamma,\bd)$ consists of the following nonzero entries  $$c_{\mu\mu}=2, \;
	c_{\star,[i,1]}=-d_i, \; c_{[i,1],\star}=-1, \;\text{and}\; c_{[i,j],[i,j+ 1]}=-1=c_{[i,j+1],[i,j]}$$
	for $1\leq i\leq \bt, 1\leq j\leq p_i-2$.

	Let $\fg$ be the Kac-Moody Lie algebra corresponding to the valued graph $(\Gamma,\bd)$ or equivalently the GCM $C$. Let $\cR_0$ be the root system of $\fg$ with simple roots $\{\alpha_\star,\alpha_{ij}\mid 1\leq i\leq\bt, 1\leq j< p_i\}$, and let $Q=\Z\alpha_\star\oplus\bigoplus_{[i,j]\in\II}\Z\alpha_{ij}$ be the root lattice. 
	Denote by $\cR=\Z^\times\delta\cup\{\cR_0+\Z\delta\}$  the root system of the loop algebra $L\fg$, and $\widehat{Q}=Q\oplus\Z\delta$ its root lattice.
	
	Recall the Grothendieck group $K_0(\coh(\X))$ of the corresponding weighted projective line. Then there is a natural isomorphism of $\Z$-modules $K_0(\coh(\X))\cong \widehat{Q}$ given as below for $1\leq i\leq \bt,\;1\le j\le p_i-1, $ $r\in\Z_{>0}$ and $l\in\Z$:
	\begin{align*}
		&\widehat{S_{ij}} \mapsto \alpha_{ij}, 
		&\widehat{S_{i,0}}\mapsto d_i\de-\sum_{j=1}^{p_i-1}\alpha_{ij},\quad
		&\widehat{S_{i,0}^{(p_i-1)}}\mapsto d_i\de-\alpha_{i1}, 
		&\widehat{S_{i,0}^{(rp_i)}}\mapsto rd_i\de,\quad
		&\widehat{\co(l\vec{c})}\mapsto \alpha_\star+l\de.
	\end{align*}
	
	Following \cite{MRY90,CJKT20}, the (twisted) loop algebra $\cl \fg$ is generated by
	the set 
	$$\{h_{\mu,m},x_{\mu,m}^{\pm},c\mid \mu\in\II, m\in d_{\mu}\Z\},$$
	and subject to the following relations $(\mu,\nu\in \II, m,m_1,m_2\in d_\mu\Z,n\in d_\nu\Z)$
	\begin{align}
		& c\text{ is central, }[h_{\mu,m},h_{\nu,n}]=\frac{c_{\mu,\nu}}{d_\nu}\delta_{m+n,0}mc,
		\\
		&[h_{\mu,m},x_{\nu,n}^{\pm}]=\pm c_{\mu\nu}x_{\nu,m+n}^{\pm}, \quad (x^\pm_{\nu,m+n}=0 \text{ if }m\notin d_\nu \Z),
		\\
		&[x_{\mu,m}^+,x_{\nu,n}^-]=\delta_{\mu,\nu}(h_{\nu,m+n}+\frac{m}{d_\mu}\delta_{m+n,0}c),
		\\
		&[x_{\mu,m+d_{\mu\nu}}^\pm,x_{\nu,n}^\pm]=[x_{\mu,m}^\pm,x_{\nu,n+d_{\mu\nu}}^\pm],
		\\
		&[x_{\mu,m}^\pm,x_{\nu,n}^\pm]=0,\text{ if }c_{\mu\nu}=0,
		\\
		\label{lie-serre1}
		&\big[x_{\mu,m_1}^\pm,[x_{\mu,m_2}^\pm,x_{\nu,n}^\pm]\big]=0,\text{ if }c_{\mu\nu}=-1
		,
		\\
		\label{lie-serre2}
		&\sum_{k=0}^{d_\nu-1}\big[x_{\mu,m_1+k}^\pm,[x^\pm_{\mu,m_2+d_\nu-1-k},x^\pm_{\nu,n}]\big]=0 \quad \text{if }\ c_{\mu,\nu}<-1, \text{(i.e., }\mu=\star,\nu=[i,1]).
	\end{align}
	Here 
	\begin{align}
		d_{\mu\nu}:=\max\{d_\mu,d_\nu\}.
	\end{align}
	

	%
	%
	
	\subsection{Quantum loop algebras of star-shaped graphs}
	\label{sec:QLA}

	The Drinfeld double (twisted) quantum loop algebra $\tU_v(\Lg)$ associated to the valued star-shaped graph $(\Gamma,\bd)$ is the $\Q(v)$-algebra generated by $X_{\mu,m}^{\pm}$, $H_{\mu,n}$ and the invertible elements $K_\mu$, $K_\mu'$, $C$, $C'$ for $\mu\in\II$, $m\in d_\mu\Z$, $n\in d_\nu\Z$ subject to the following relations: 
	\begin{align}
		\label{DR1}
		&		C, C' \text{ are  central},\quad [K_\mu,K_\nu]  =[K_\mu,K_\nu']=[K_\mu',K_\nu']=0,
		\\ &[K_\mu,H_{\nu ,n}] =0=[K_\mu',H_{\nu,n}], 
		\\
		\label{DR3}
		&K_\mu X_{\nu,n}^{\pm} =v_\mu^{\pm c_{\mu\nu}} X_{\nu,n}^{\pm}K_\mu,\qquad K_\mu'X_{\nu,n}^{\pm} =v_\mu^{\pm c_{\mu\nu}} X_{\nu,n}^{\pm}K_\mu',
		\\
		\label{DR2}
		&[H_{\mu,\pm m},H_{\nu,\pm n}]=0,\quad 
		[H_{\mu,m},H_{\nu,-n}] = 
		\delta_{m, n} b_{\mu\nu m} \frac{C^m -C'^{m}}{v_\nu -v_\nu^{-1}}, \text{ for }m,n>0,
		\\
		\label{DR4}
		\hspace{-1cm}
		&[H_{\mu,m},X_{\nu, n}^{\pm}] =\begin{cases}\pm b_{\mu\nu m} C^{\frac{m\mp m}{2}} X_{\nu,m+n}^{\pm},&\text{ if }m>0,
			\\
			\pm b_{\mu\nu m} C'^{-\frac{m\pm m}{2}} X_{\nu,m+n}^{\pm},&\text{ if }m<0,
		\end{cases}
		\quad (X^\pm_{\nu,m+n}=0\text{ if }m\notin d_\nu \Z),
		\\
		\label{DR5}
		&[X_{\mu ,m}^+,X_{\nu,n }^-] =\delta_{\mu,\nu}\frac{1}{v_\mu-v_\mu^{-1}} {(C^{-n} K_\mu\psi_{\mu,m+n} - C'^{m} K_\mu' \varphi_{\mu,m+n})}, 
		\\
		\label{DR6}
		&X_{\mu,m+d_{\mu\nu}}^{\pm} X_{\nu,n}^{\pm}-v_\mu^{\pm c_{\mu\nu}} X_{\nu,n}^{\pm} X_{\mu,m+d_{\mu\nu}}^{\pm} =v_\mu^{\pm c_{\mu\nu}} x_{\mu,m}^{\pm} X_{\nu,n+d_{\mu\nu}}^{\pm}- X_{\nu,n+d_{\mu\nu}}^{\pm} X_{\mu,m}^{\pm},
		\\
		&[X^\pm_{\mu,m},X^\pm_{\nu,n}]=0,\text{ if }c_{\mu\nu}=0,
		\\
		\label{DR*}
		&\Sym_{m_1,m_2}\big(X^{\pm}_{\nu,n}X^{\pm}_{\mu,m_1}X^{\pm}_{\mu,m_2}-[2]_{v_\mu}X^{\pm}_{\mu,m_1}X^{\pm}_{\nu,n}X^{\pm}_{\mu,m_2}+X^{\pm}_{\mu,m_1}X^{\pm}_{\mu,m_2}X^{\pm}_{\nu,n}\big)=0,\notag\\
		&\hspace{7cm}\text{ for } \mu,\nu\in\II\setminus\{\star\}, \text{ with }c_{\mu\nu}=-1,
		\\
		\notag
		&\Sym_{m_1,m_2}\sum_{t=0}^{d_\nu-1}v^{d_\nu-1-2t}\big( X_{\nu,n}^{\pm}X_{\mu,m_1\pm (d_\nu-1-t)}^{\pm}X_{\mu,m_2\pm t}^{\pm}-[2]_{v_\nu}X_{\mu,m_1\pm (d_\nu-1-t)}^{\pm}X_{\nu,n}^{\pm}X_{\mu,m_2\pm t}^{\pm}  \\
		\label{DR7}
		&\hspace{3cm}+X_{\mu,m_1\pm (d_\nu-1-t)}^{\pm}X_{\mu,m_2\pm t}^{\pm} X_{\nu,n}^{\pm} \big)=0, \text{ for }\nu=[i,1],\mu=\star,
		\\\notag
		&\Sym_{k_1,\dots,k_r}\sum_{t=0}^{r} (-1)^t \qbinom{r}{t}_{d_\mu}
		X_{\mu,k_1}^{\pm}\cdots
		X_{\mu,k_t}^{\pm} X_{\nu,l}^{\pm}  X_{\mu,k_{t+1}}^{\pm} \cdots X_{\mu,k_r}^{\pm} =0,
		\\
		&\hspace{7cm}\text{ for } r= 1-c_{\mu,\nu}, \nu=\star,\mu=[i,1]. \label{DR8}
	\end{align}
	Here
	$b_{\mu\nu m}$ is defined by
	\begin{align}
		b_{\mu\nu m}=\begin{cases}
			0&\text{ if }d_{\mu\nu}\nmid m,\\
			\frac{[\widetilde{m}c_{\mu\nu}]_{v_\mu}}{\widetilde{m}} & \text{ otherwise, with }\widetilde{m}=\frac{m}{d_{\mu\nu}};
		\end{cases}
	\end{align}
	$\Sym_{m_1,m_2}$ denotes the symmetrization with respect to the indices $m_1,m_2$; $\psi_{\mu,d_\mu m}$ and $\varphi_{\mu,-d_\mu m}$ for $m>0$ are defined by the following functional equations:
	\begin{align}
		\label{exp h+}
		1+ \sum_{m\geq 1} \psi_{\mu,d_\mu m}u^m &=  \exp\Big((v_\mu -v_\mu^{-1}) \sum_{m\ge 1}  H_{\mu,d_\mu m}u^m\Big),
		\\
		\label{exp h-}
		1+ \sum_{m\geq1 }\varphi_{\mu, -d_\mu m}u^{-m} &= \exp \Big((v_\mu^{-1} -v_\mu) \sum_{m\ge 1} H_{\mu,-d_\mu m}u^{-m}\Big);
	\end{align}
	and $\psi_{\mu,r}=0=\varphi_{\mu,-r}$ if $r<0$ or $d_\mu\nmid r$. For convenience, we set $H_{\mu,0}=0$, and 
	$\psi_{\mu,0}=1=\varphi_{\mu,0}$ in the following.
	
	We call \eqref{DR*}--\eqref{DR8} the quantum Serre relations of $\tU_v(\Lg)$.
	
	\begin{remark}[cf. \cite{Be94}]
		By restricting to the $i$-th branch, the subalgebra of $\tU(L\fg)$
		generated by 
		$X_{[i,j],m}^{\pm1}$, $H_{[i,j],n}$ for $1\leq j\leq p_i-1$, $K_{[i,j]}^{\pm1}$, $(K_{[i,j]}')^{\pm1}$, $C$, $C'$ for $m\in d_i\Z$, $n\in d_i\Z^\times$ is 	the Drinfeld new presentation $\tUD_{v_i}(\widehat{\mathfrak{sl}}_{p_i})$ of $\tU_{v_i}(\widehat{\mathfrak{sl}}_{p_i})$.
	\end{remark}
	
	\begin{remark}
		\label{rem:gen-reduce}
		For any (fixed) $\mu\in\II$, the $\tU(L\fg)$ is generated by $H_{\mu,\pm1}$,  $X_{\nu,0}^{\pm1}$, and the invertible elements $C$, $C'$, $K_\nu$, $K_\nu'$ for all $\nu\in\II$.
	\end{remark}
	
	\begin{proof}
		The proof is similar to \cite[Lemma 2.7]{LR24b}, hence omitted here.
	\end{proof}
	
	Introduced by Drinfeld \cite{Dr88,Da12,CJKT23}, the quantum loop algebra $\U_v(\Lg)$ can be described as a quotient algebra of $\tU_v(\Lg)$ modulo the ideal generated by
	$$CC'-1,\qquad K_\mu K_\mu'-1, \,\,\forall \mu\in\II.$$

	For $(\Gamma,\bd)$ of Dynkin type, 
	from \cite{Dr88,Be94,Da12}, we know that $\tU_v(L\fg)$ is isomorphic to the (twisted) affine quantum group,  we give their precise type in Table \ref{tab1}. In this case, we also write $\tU_v(L\fg)$ to be $\tUD_v(\widehat{\fg})$ in the following.

	%
	%
	
	
	\begin{table}[h]  
		\begin{center}
			\centering
			\begin{tabular}{|m{7cm}<{\centering}|m{8cm}<{\centering}|}
				\hline
				Star-shaped graphs ($\Gamma,\bd$)& Types of quantum affine algebras\\
				\hline
				\begin{tikzpicture}[baseline=0.0, scale=1.5]
					\node at (-1,0.4){\tiny $\star$};
					\node at (-0.5,0.1) {\tiny$\bullet$};
					\node at (-0.5,0.7) {\tiny$\bullet$};
					\draw[-] (-0.95,0.38) to (-0.5,0.1);
					\draw[-] (-.95,0.42) to (-0.5,0.7);
					\node at (0,0.1) {\tiny$\bullet$};
					\node at (0,0.7) {\tiny$\bullet$};
					\draw[-] (-0.5,0.1) to (0,0.1);
					\draw[-] (-0.5,0.7) to (0,0.7);
					
					\node at (0.7,0.1) {\tiny$\bullet$};
					\node at (0.7,0.7) {\tiny$\bullet$};
					
					\draw[dashed] (0.1,0.1) to (0.6,0.1);
					\draw[dashed] (0.1,0.7) to (0.6,0.7);
					
					\node at (1.2,0.1) {\tiny$\bullet$};
					\node at (1.2,0.7) {\tiny$\bullet$};
					
					\draw[-] (0.7,0.1) to (1.2,0.1);
					\draw[-] (0.7,0.7) to (1.2,0.7);
					
					\node at (-0.5,0.85) {\tiny$[1,1]$};
					\node at (0,0.85) {\tiny$[1,2]$};
					\node at (1.2,0.85) {\tiny$[1,p_1-1]$};
					
					\node at (-0.5,-0.05) {\tiny$[2,1]$};
					\node at (0,-0.05) {\tiny$[2,2]$};
					\node at (1.2,-0.05) {\tiny$[2,p_2-1]$};
					
					\node at (-1.8,0.4) {\tiny $A_n$ };
					\node at (-1.8,0.1) {\tiny ($n=p_1+p_2-1$)};
				\end{tikzpicture}
				&
				\begin{tikzpicture}[baseline=0.0, scale=1.5]
					\node at (-1,0.4){\tiny $\star$};
					\node at (-0.5,0.1) {\tiny$\bullet$};
					\node at (-0.5,0.7) {\tiny$\bullet$};
					\draw[-] (-0.95,0.38) to (-0.5,0.1);
					\draw[-] (-.95,0.42) to (-0.5,0.7);
					\node at (0,0.1) {\tiny$\bullet$};
					\node at (0,0.7) {\tiny$\bullet$};
					\draw[-] (-0.5,0.1) to (0,0.1);
					\draw[-] (-0.5,0.7) to (0,0.7);
					
					\node at (0.7,0.1) {\tiny$\bullet$};
					\node at (0.7,0.7) {\tiny$\bullet$};
					
					\draw[dashed] (0.1,0.1) to (0.6,0.1);
					\draw[dashed] (0.1,0.7) to (0.6,0.7);
					
					\node at (1.2,0.1) {\tiny$\bullet$};
					\node at (1.2,0.7) {\tiny$\bullet$};
					
					\draw[-] (0.7,0.1) to (1.2,0.1);
					\draw[-] (0.7,0.7) to (1.2,0.7);
					
					\node at (-0.5,0.85) {\tiny$[1,1]$};
					\node at (0,0.85) {\tiny$[1,2]$};
					\node at (1.2,0.85) {\tiny$[1,p_1-1]$};
					
					\node at (-0.5,-0.05) {\tiny$[2,1]$};
					\node at (0,-0.05) {\tiny$[2,2]$};
					\node at (1.2,-0.05) {\tiny$[2,p_2-1]$};
					
					\node at (1.7,0.4) {\tiny$\bullet$};
					\node at (1.7,0.25){\tiny$0$};
					\draw[-] (1.2,0.7) to (1.7,0.4);
					\draw[-] (1.2,0.1) to (1.7,0.4);

					\node at (-1.8,0.4) {\tiny $A^{(1)}_n$ };
					\node at (-1.8,0.1) {\tiny ($n=p_1+p_2-1$)};
				\end{tikzpicture}
				\\
				\hline
				\begin{tikzpicture}[baseline=0.0, scale=1.5]
					\node at (-0.5,0.1) {\tiny $\bullet$};
					\node at (-0.5,0.7) {\tiny $\bullet$};
					\node at (0,0.4) {\tiny $\star$};
					\node at (0.5,0.4) {\tiny $\bullet$};
					\node at (1.5,0.4) {\tiny $\bullet$};
					\node at (2.0,0.4) {\tiny $\bullet$};
					\draw[-] (0.05,0.4) to (0.45,0.4);
					\draw[-] (-0.45,0.67) to (-0.05,0.44);
					\draw[-] (-0.45,0.14) to (-0.05,0.37);
					\draw[-] (1.55,0.4) to (1.95,0.4);
					
					\draw[dashed] (0.55,0.4) to (1.45,0.4);
					
					\node at (2,0.2) {\tiny $[1,p_1-1]$};
					\node at (0.5,0.2) {\tiny $[1,1]$};
					\node at (-0.5,-.1) {\tiny $[3,1]$};
					\node at (-0.5,0.9) {\tiny $[2,1]$};

					\node at (-1.3,0.4) {\tiny $D_n$ };
					\node at (-1.3,0.1) {\tiny ($n=p_1+2$)};
				\end{tikzpicture}
				&  \begin{tikzpicture}[baseline=0.0, scale=1.5]
					\node at (-0.5,0.1) {\tiny $\bullet$};
					\node at (-0.5,0.7) {\tiny $\bullet$};
					\node at (0,0.4) {\tiny $\star$};
					\node at (0.5,0.4) {\tiny $\bullet$};
					\node at (1.5,0.4) {\tiny $\bullet$};
					\node at (2.0,0.4) {\tiny $\bullet$};
					\node at (2.5,0.1) {\tiny $\bullet$};
					\node at (2.5,0.7) {\tiny $\bullet$};
					\draw[-] (0.05,0.4) to (0.45,0.4);
					\draw[-] (-0.45,0.67) to (-0.05,0.44);
					\draw[-] (-0.45,0.14) to (-0.05,0.37);
					\draw[-] (1.55,0.4) to (1.95,0.4);
					\draw[-] (2.0,0.4) to (2.5,0.7);
					\draw[-] (2.0,0.4) to (2.5,0.1);
					
					\draw[dashed] (0.55,0.4) to (1.45,0.4);
					
					\node at (1.8,0.2) {\tiny $[1,p_1-2]$};
					\node at (2.5,-0.1) {\tiny $0$};
					\node at (2.5,0.9) {\tiny $[1,p_1-1]$};
					\node at (0.5,0.2) {\tiny $[1,1]$};
					\node at (-0.5,-.1) {\tiny $[3,1]$};
					\node at (-0.5,0.9) {\tiny $[2,1]$};
					
					\node at (-1.5,0.4) {\tiny $D^{(1)}_n$ };
					\node at (-1.5,0.1) {\tiny ($n=p_1+2$)};
				\end{tikzpicture}
				\\
				\hline
				\begin{tikzpicture}[baseline=0, scale=1.5]
					\node at (-1.5,0.4) {\tiny $E_6$ };
					\node	at (-0.5,0.4){$\quad$};
					\node at (-0.5,0.2) {\tiny$\bullet$};
					\node at (0,0.2) {\tiny$\bullet$};
					\node at (0.5,0.2) {\tiny$\star$};
					\node at (1.0,0.2) {\tiny$\bullet$};
					\node at (1.5,0.2) {\tiny$\bullet$};
					
					\node at (0.5,0.7) {\tiny$\bullet$};
					\draw[-] (0.5,0.27) to (0.5,0.7);
					\draw[-] (-0.45,0.2) to (-0.05,0.2);
					\draw[-] (0.05,0.2) to (0.45,0.2);
					\draw[-] (1.05,0.2) to (1.45,0.2);
					\draw[-]  (0.55, 0.2) to (0.95, 0.2);
					\node at (1.5,0) {\tiny $[1,2]$};
					\node at (0,0) {\tiny $[3,1]$};
					\node at (1,0) {\tiny $[1,1]$};
					\node at (-0.5,0) {\tiny $[3,2]$};
					\node	at (-0.5,-0.1){$\quad$};
					\node at (0.8,0.7) {\tiny$[2,1]$};
				\end{tikzpicture}
				
				&\begin{tikzpicture}[baseline=0, scale=1.5]
					\node at (-1.5,0.4) {\tiny $E^{(1)}_6$ };
					\node	at (-0.5,0.4){$\quad$};
					\node at (-0.5,0.2) {\tiny$\bullet$};
					\node at (0,0.2) {\tiny$\bullet$};
					\node at (0.5,0.2) {\tiny$\star$};
					\node at (1.0,0.2) {\tiny$\bullet$};
					\node at (1.5,0.2) {\tiny$\bullet$};
					
					\node at (0.5,0.7) {\tiny$\bullet$};
					\node at (0.5,1.2) {\tiny$\bullet$};
					\draw[-] (0.5,0.27) to (0.5,0.7);
					\draw[-] (0.5,0.7) to (0.5,1.2);
					\draw[-] (-0.45,0.2) to (-0.05,0.2);
					\draw[-] (0.05,0.2) to (0.45,0.2);
					\draw[-] (1.05,0.2) to (1.45,0.2);
					\draw[-]  (0.55, 0.2) to (0.95, 0.2);
					\node at (1.5,0) {\tiny $[1,2]$};
					\node at (0,0) {\tiny $[3,1]$};
					\node at (1,0) {\tiny $[1,1]$};
					\node at (-0.5,0) {\tiny $[3,2]$};
					\node	at (-0.5,-0.1){$\quad$};
					\node at (0.8,0.7) {\tiny$[2,1]$};
					\node at (0.7,1.2) {\tiny$0$};
				\end{tikzpicture}
				\\\hline
				\begin{tikzpicture}[baseline=0, scale=1.5]
					\node at (-1.5,0.4) {\tiny $E_7$ };
					\node	at (-0.5,0.4){$\quad$};
					\node at (-0.5,0.2) {\tiny$\bullet$};
					\node at (0,0.2) {\tiny$\bullet$};
					\node at (0.5,0.2) {\tiny$\star$};
					\node at (1.0,0.2) {\tiny$\bullet$};
					\node at (1.5,0.2) {\tiny$\bullet$};
					
					\node at (0.5,0.7) {\tiny$\bullet$};
					\node at (2,0.2) {\tiny$\bullet$};
					\draw[-] (0.5,0.27) to (0.5,0.7);
					\draw[-] (1.5,0.2) to (2,0.2);
					
					\draw[-] (-0.45,0.2) to (-0.05,0.2);
					\draw[-] (0.05,0.2) to (0.45,0.2);
					\draw[-] (1.05,0.2) to (1.45,0.2);
					\draw[-]  (0.55, 0.2) to (0.95, 0.2);
					\node at (1.5,0) {\tiny $[1,2]$};
					\node at (0,0) {\tiny $[3,1]$};
					\node at (2,0) {\tiny $[1,3]$};
					\node at (1,0) {\tiny $[1,1]$};
					\node at (-0.5,0) {\tiny $[3,2]$};
					\node	at (-0.5,-0.1){$\quad$};
					\node at (0.8,0.7) {\tiny$[2,1]$};
				\end{tikzpicture}
				
				&\begin{tikzpicture}[baseline=0, scale=1.5]
					\node at (-1.5,0.4) {\tiny $E^{(1)}_7$ };
					\node	at (-0.5,0.4){$\quad$};
					\node	at (-1,0.2){\tiny$\bullet$};
					\node	at (-1,0){\tiny$0$};
					\draw[-] (-1,0.2) to (-0.5,0.2);
					\node at (-0.5,0.2) {\tiny$\bullet$};
					\node at (0,0.2) {\tiny$\bullet$};
					\node at (0.5,0.2) {\tiny$\star$};
					\node at (1.0,0.2) {\tiny$\bullet$};
					\node at (1.5,0.2) {\tiny$\bullet$};
					
					\node at (0.5,0.7) {\tiny$\bullet$};
					\node at (2,0.2) {\tiny$\bullet$};
					\draw[-] (0.5,0.27) to (0.5,0.7);
					\draw[-] (1.5,0.2) to (2,0.2);
					
					\draw[-] (-0.45,0.2) to (-0.05,0.2);
					\draw[-] (0.05,0.2) to (0.45,0.2);
					\draw[-] (1.05,0.2) to (1.45,0.2);
					\draw[-]  (0.55, 0.2) to (0.95, 0.2);
					\node at (1.5,0) {\tiny $[1,2]$};
					\node at (0,0) {\tiny $[3,1]$};
					\node at (2,0) {\tiny $[1,3]$};
					\node at (1,0) {\tiny $[1,1]$};
					\node at (-0.5,0) {\tiny $[3,2]$};
					\node	at (-0.5,-0.1){$\quad$};
					\node at (0.8,0.7) {\tiny$[2,1]$};
				\end{tikzpicture}
				\\\hline
				\begin{tikzpicture}[baseline=0, scale=1.5]
					\node at (-1.5,0.4) {\tiny $E_8$ };
					\node	at (-0.5,0.4){$\quad$};
					\node at (-0.5,0.2) {\tiny$\bullet$};
					\node at (0,0.2) {\tiny$\bullet$};
					\node at (0.5,0.2) {\tiny$\star$};
					\node at (1.0,0.2) {\tiny$\bullet$};
					\node at (1.5,0.2) {\tiny$\bullet$};
					
					\node at (0.5,0.7) {\tiny$\bullet$};
					\node at (2,0.2) {\tiny$\bullet$};
					\node at (2.5,0.2) {\tiny$\bullet$};
					\draw[-] (2.5,0.2) to (2,0.2);
					\draw[-] (0.5,0.27) to (0.5,0.7);
					\draw[-] (1.5,0.2) to (2,0.2);
					
					\draw[-] (-0.45,0.2) to (-0.05,0.2);
					\draw[-] (0.05,0.2) to (0.45,0.2);
					\draw[-] (1.05,0.2) to (1.45,0.2);
					\draw[-]  (0.55, 0.2) to (0.95, 0.2);
					\node at (1.5,0) {\tiny $[1,2]$};
					\node at (0,0) {\tiny $[3,1]$};
					\node at (2,0) {\tiny $[1,3]$};
					\node at (2.5,0) {\tiny $[1,4]$};
					\node at (1,0) {\tiny $[1,1]$};
					\node at (-0.5,0) {\tiny $[3,2]$};
					\node	at (-0.5,-0.1){$\quad$};
					\node at (0.8,0.7) {\tiny$[2,1]$};
				\end{tikzpicture}
				
				&\begin{tikzpicture}[baseline=0, scale=1.5]
					\node at (-1.2,0.4) {\tiny $E_8^{(1)}$ };
					\node	at (-0.5,0.4){$\quad$};
					\node at (-0.5,0.2) {\tiny$\bullet$};
					\node at (0,0.2) {\tiny$\bullet$};
					\node at (0.5,0.2) {\tiny$\star$};
					\node at (1.0,0.2) {\tiny$\bullet$};
					\node at (1.5,0.2) {\tiny$\bullet$};
					
					\node at (0.5,0.7) {\tiny$\bullet$};
					\node at (2,0.2) {\tiny$\bullet$};
					\node at (2.5,0.2) {\tiny$\bullet$};
					\node at (3,0.2) {\tiny$\bullet$};
					\draw[-] (2.5,0.2) to (2,0.2);
					\draw[-] (2.5,0.2) to (3,0.2);
					\draw[-] (0.5,0.27) to (0.5,0.7);
					\draw[-] (1.5,0.2) to (2,0.2);
					
					\draw[-] (-0.45,0.2) to (-0.05,0.2);
					\draw[-] (0.05,0.2) to (0.45,0.2);
					\draw[-] (1.05,0.2) to (1.45,0.2);
					\draw[-]  (0.55, 0.2) to (0.95, 0.2);
					\node at (1.5,0) {\tiny $[1,2]$};
					\node at (0,0) {\tiny $[3,1]$};
					\node at (2,0) {\tiny $[1,3]$};
					\node at (2.5,0) {\tiny $[1,4]$};
					\node at (3,0) {\tiny $0$};
					\node at (1,0) {\tiny $[1,1]$};
					\node at (-0.5,0) {\tiny $[3,2]$};
					\node	at (-0.5,-0.1){$\quad$};
					\node at (0.8,0.7) {\tiny$[2,1]$};
				\end{tikzpicture}
				\\\hline
				
				\begin{tikzpicture}[baseline=0, scale=1.5]
					\node at (-1,0.4) {\tiny $B_n$ };
					\node	at (-1,0){\tiny$(n=p_1+1)$};
					\node at (0,0.2) {\tiny$\bullet$};
					\node at (0.5,0.2) {\tiny$\star$};
					\node at (1.0,0.2) {\tiny$\bullet$};
					\node at (1.5,0.2) {\tiny$\bullet$};
					
					\node at (2,0.2) {\tiny$\bullet$};
					\node at (2.5,0.2) {\tiny$\bullet$};
					\draw[-] (2.5,0.2) to (2,0.2);
					
					\draw[dashed] (1.5,0.2) to (2,0.2);
					
					\draw[-implies, double equal sign distance] (0.05,0.2) to (0.45,0.2);
					
					\draw[-] (1.05,0.2) to (1.45,0.2);
					\draw[-]  (0.55, 0.2) to (0.95, 0.2);
					\node at (1.5,0) {\tiny $[1,2]$};
					\node at (0,0) {\tiny $[2,1]$};
					\node at (2.5,0) {\tiny $[1,p_1-1]$};
					\node at (1,0) {\tiny $[1,1]$};
					\node	at (-0.5,0.6){$\quad$};
				\end{tikzpicture}
				
				&
				\begin{tikzpicture}[baseline=0, scale=1.5]
					\node at (-1,0.4) {\tiny $A_{2n-1}^{(2)}$ };
					\node	at (-1,0){\tiny$(n=p_1+1)$};
					\node at (0,0.2) {\tiny$\bullet$};
					\node at (0.5,0.2) {\tiny$\star$};
					\node at (1.0,0.2) {\tiny$\bullet$};
					\node at (1.5,0.2) {\tiny$\bullet$};
					
					\node at (2,0.2) {\tiny$\bullet$};
					\node at (2.5,0.2) {\tiny$\bullet$};
					\node at (3,0.5) {\tiny$\bullet$};
					\node at (3,-0.1) {\tiny$\bullet$};
					\draw[-] (2.5,0.2) to (2,0.2);
					\draw[-] (2.5,0.2) to (3,-0.1);
					\draw[-] (2.5,0.2) to (3,0.5);
					
					\draw[dashed] (1.5,0.2) to (2,0.2);
					
					\draw[-implies, double equal sign distance] (0.05,0.2) to (0.45,0.2);
					
					\draw[-] (1.05,0.2) to (1.45,0.2);
					\draw[-]  (0.55, 0.2) to (0.95, 0.2);
					\node at (1.5,0) {\tiny $[1,2]$};
					\node at (0,0) {\tiny $[2,1]$};
					\node at (3,0.7) {\tiny $[1,p_1-1]$};
					\node at (3,-0.3) {\tiny $0$};
					\node at (1,0) {\tiny $[1,1]$};
					\node	at (-0.5,-0.1){$\quad$};
				\end{tikzpicture}
				\\\hline
				\begin{tikzpicture}[baseline=0, scale=1.5]
					\node at (-0.5,0.4) {\tiny $C_n$ };
					\node	at (-0.5,0){\tiny$(n=p_1)$};
					\node at (0.5,0.2) {\tiny$\star$};
					\node at (1.0,0.2) {\tiny$\bullet$};
					\node at (1.5,0.2) {\tiny$\bullet$};
					
					\node at (2,0.2) {\tiny$\bullet$};
					\node at (2.5,0.2) {\tiny$\bullet$};
					\draw[-] (2.5,0.2) to (2,0.2);
					
					\draw[dashed] (1.5,0.2) to (2,0.2);
					
					\draw[-implies, double equal sign distance] (0.95,0.2) to (0.55,0.2);
					
					\draw[-] (1.05,0.2) to (1.45,0.2);
					\node at (1.5,0) {\tiny $[1,2]$};
					\node at (2.5,0) {\tiny $[1,p_1-1]$};
					\node at (1,0) {\tiny $[1,1]$};
					\node	at (-0.5,-0.1){$\quad$};
				\end{tikzpicture}
				
				&
				\begin{tikzpicture}[baseline=0, scale=1.5]
					\node at (-0.5,0.4) {\tiny $D^{(2)}_{n+1}$ };
					\node	at (-0.5,0){\tiny$(n=p_1)$};
					\node at (0.5,0.2) {\tiny$\star$};
					\node at (1.0,0.2) {\tiny$\bullet$};
					\node at (1.5,0.2) {\tiny$\bullet$};
					\node at (3,0.2) {\tiny$\bullet$};
					\node at (3,0) {\tiny$0$};
					\node at (2,0.2) {\tiny$\bullet$};
					\node at (2.5,0.2) {\tiny$\bullet$};
					\draw[-] (2.5,0.2) to (2,0.2);
					
					\draw[dashed] (1.5,0.2) to (2,0.2);
					
					\draw[-implies, double equal sign distance] (0.95,0.2) to (0.55,0.2);
					
					\draw[-implies, double equal sign distance] (2.55,0.2) to (2.95,0.2);
					\draw[-] (1.05,0.2) to (1.45,0.2);
					\node at (1.5,0) {\tiny $[1,2]$};
					\node at (2.5,0) {\tiny $[1,p_1-1]$};
					\node at (1,0) {\tiny $[1,1]$};
					\node	at (-0.5,-0.1){$\quad$};
				\end{tikzpicture}
				\\\hline
				\begin{tikzpicture}[baseline=0, scale=1.5]
					\node at (-1.2,0.2) {\tiny $F_4$ };
					\node at (-0.5,0.2) {\tiny$\bullet$};
					\node at (0,0.2) {\tiny$\bullet$};
					\node at (0.5,0.2) {\tiny$\star$};
					\node at (1.0,0.2) {\tiny$\bullet$};
					
					
					
					\draw[-] (-0.45,0.2) to (-0.05,0.2);
					\draw[-implies, double equal sign distance] (0.05,0.2) to (0.45,0.2);
					
					\draw[-] (1,0.2) to (0.55,0.2);
					\node at (-0.5,0) {\tiny $[2,2]$};
					\node at (0,0) {\tiny $[2,1]$};
					\node at (1,0) {\tiny $[1,1]$};
					\node	at (-2,0.4){$\quad$};
				\end{tikzpicture}
				
				&
				\begin{tikzpicture}[baseline=0, scale=1.5]
					\node at (-1.2,0.2) {\tiny $E_6^{(2)}$ };
					\node at (-0.5,0.2) {\tiny$\bullet$};
					\node at (0,0.2) {\tiny$\bullet$};
					\node at (0.5,0.2) {\tiny$\star$};
					\node at (1.0,0.2) {\tiny$\bullet$};
					\node at (1.5,0.2) {\tiny$\bullet$};
					\node at (1.5,0) {\tiny$0$};
					\node at (-0.5,0) {\tiny $[2,2]$};
					
					
					\draw[-] (-0.45,0.2) to (-0.05,0.2);
					\draw[-implies, double equal sign distance] (0.05,0.2) to (0.45,0.2);
					
					\draw[-] (1,0.2) to (0.55,0.2);
					\draw[-] (1,0.2) to (1.5,0.2);
					\node at (0,0) {\tiny $[2,1]$};
					\node at (1,0) {\tiny $[1,1]$};
					\node	at (-1.5,-0.1){$\quad$};
				\end{tikzpicture}
				\\
				\hline
				\begin{tikzpicture}[baseline=0, scale=2.2]
					\node at (-0.7,0.2) {\tiny $G_2$ };
					\node at (0,0.2) {\tiny$\bullet$};
					\node at (0.5,0.2) {\tiny$\star$};
					
					
					
					\draw[-implies, double equal sign distance] (0.05,0.2) to (0.45,0.2);
					\draw[-] (0.05,0.2) to (0.45,0.2);
					\node at (0,0) {\tiny $[1,1]$};
					\node	at (-1,0.4){$\quad$};
				\end{tikzpicture}
				
				&
				\begin{tikzpicture}[baseline=0, scale=2.2]
					\node at (-0.7,0.2) {\tiny $D_4^{(3)}$ };
					\node at (0,0.2) {\tiny$\bullet$};
					\node at (0.5,0.2) {\tiny$\star$};
					\node at (1.0,0.2) {\tiny$\bullet$};
					\node at (1.0,0) {\tiny$0$};
					
					
					
					\draw[-implies, double equal sign distance] (0.05,0.2) to (0.45,0.2);
					\draw[-] (0.05,0.2) to (0.45,0.2);
					\draw[-]  (0.55, 0.2) to (0.95, 0.2);
					\node at (0,0) {\tiny $[1,1]$};
					\node	at (-0.5,-0.1){$\quad$};
				\end{tikzpicture}
				\\
				\hline
			\end{tabular} 
		\end{center}
		\caption{Types of star-shaped graphs and quantum affine algebras }
		\label{tab1}
	\end{table}

	
	\section{Semi-derived Ringel-Hall algebras}
	
	\label{sec:Semi-derived}
	
	Let $\Z_n=\Z/n\Z$ for $n\geq1$.
	In this section, we shall review the categories of periodic complexes, Ringel-Hall algebras, semi-derived Ringel-Hall algebras and $\imath$Hall algebras for arbitrary hereditary abelian categories. 
	
	\subsection{Categories of $\Z_2$-graded complexes}
	We assume that $\ca$  is an abelian category. 
	Let $\cc_{\Z_2}(\ca)$ be the  category of $\Z_2$-graded complexes over $\ca$. Namely, an object $M\in\cc_{\Z_2}(\ca)$ is a diagram with objects in $\ca$:
	$$\xymatrix{ M^0 \ar@<0.5ex>[r]^{d^0}& M^1 \ar@<0.5ex>[l]^{d^1}  },\quad d^1d^0=d^0d^1=0.$$
	There is a forgetful functor $\res: \cc_{\Z_2}(\ca)\rightarrow \ca\coprod\ca$, which maps $M=(\xymatrix{ M^0 \ar@<0.5ex>[r]^{d^0}& M^1 \ar@<0.5ex>[l]^{d^1}  })$ to $(M^0, M^{1})$.
	
	We can define the $i$-th cohomology group for $M$, denoted by $H^i(M)$, for any $i\in\Z_2$. A complex is called acyclic if its cohomology group is zero. The subcategory formed by all acyclic complexes are denoted by $\cc_{\Z_2,ac}(\ca)$.

	For any object $X\in\ca$, we define
	\begin{align}
		\label{stalks}
		\begin{split}
			K_X:=&(\xymatrix{ X \ar@<0.5ex>[r]^{1}& X \ar@<0.5ex>[l]^{0}  }),\qquad \,\, K_X^*:=(\xymatrix{ X \ar@<0.5ex>[r]^{0}& X \ar@<0.5ex>[l]^{1}  }),
			\\
			C_X:=&(\xymatrix{ 0 \ar@<0.5ex>[r]& X \ar@<0.5ex>[l]  }),\qquad \quad C_X^*:=(\xymatrix{ X\ar@<0.5ex>[r]& 0 \ar@<0.5ex>[l]  })
		\end{split}
	\end{align}
	in $\cc_{\Z_2}(\ca)$.  Note that $K_X,K_X^*$ are acyclic complexes.
	

	In the following, we always assume that $\ca$ is a hereditary abelian $\bfk$-linear category which is essentially small with finite-dimensional homomorphism and extension spaces.

	By \cite[Proposition 2.3]{LP16}, for any $K,M\in \cc_{\Z_2}(\ca)$ with $K$ acyclic, we define 
	\begin{align*}
		\langle K,M\rangle=\dim\Hom_{\cc_{\Z_2}({\ca})}(K,M)-\dim\Ext^1_{\cc_{\Z_2}({\ca})}(K,M)
	\end{align*}
	and
	\begin{align*}
		\langle M,K\rangle =\dim\Hom_{\cc_{\Z_2}({\ca})}(M,K)-\dim\Ext^1_{\cc_{\Z_2}({\ca})}(M,K).
	\end{align*}
	We call them the Euler forms of $\cc_{\Z_2}(\ca)$. They descend to bilinear forms on the Grothendieck groups $K_0(\cc_{\Z_2,ac}(\ca))$
	and $K_0(\cc_{\Z_2}(\ca))$, denoted by the same symbol:
	$$\langle\cdot,\cdot\rangle:K_0(\cc_{\Z_2,ac}({\ca}))\times K_0(\cc_{\Z_2}(\ca))\longrightarrow \Z,$$
	and
	$$\langle\cdot,\cdot\rangle:K_0(\cc_{\Z_2}(\ca))\times K_0(\cc_{\Z_2,ac}({\ca}))\longrightarrow \Z.$$
	We can use the same symbol, since these two forms coincide on $K_0(\cc_{\Z_2,ac}({\ca}))\times K_0(\cc_{\Z_2,ac}({\ca}))$.

	We also use $\langle \cdot,\cdot\rangle$ to denote the Euler form of $\ca$ or $K_0(\ca)$, i.e.,
	\begin{align*}
		\langle \widehat{A}, \widehat{B}\rangle=\dim\Hom_\ca(A,B)-\dim\Ext^1_\ca(A,B), \text{ for any }A,B\in\ca.
	\end{align*}
	Let $(\cdot,\cdot)$ be the symmetrized Euler form of $\ca$, i.e., $(\widehat{A},\widehat{B})= \langle \widehat{A}, \widehat{B}\rangle+\langle \widehat{B}, \widehat{A}\rangle$.
	
	


	
	
	\subsection{Ringel-Hall algebras}
	
	Denote by $|\mathcal{J}|$ for the  cardinality of a finite set $\mathcal{J}$.	
	Given objects $A,B,C\in\ca$, define $\Ext^1_\ca(A,C)_B\subseteq \Ext^1_\ca(A,C)$ to be the subset parameterising extensions with the middle term  isomorphic to $B$. We define the Hall algebra (also called Ringel-Hall algebra) $\ch(\ca)$ of the abelian category $\ca$ to be the $\Q$-vector space whose basis is formed by the isoclasses $[A]$ of objects $A\in\ca$, with the multiplication
	defined by
	\begin{align}
		\label{eq:mult}
		[A]\diamond [C]=\sum_{[B]\in \Iso(\ca)}\frac{|\Ext_\ca^1(A,C)_B|}{|\Hom_\ca(A,C)|}[B].
	\end{align}
	It is well known that
	the algebra $\ch(\ca)$ is associative and unital. The unit is given by $[0]$, where $0$ is the zero object of $\ca$; see \cite{Rin90,Br13}. 

	For any three objects $A,B,C$, let
	\begin{align}
		\label{eq:Fxyz}
		F_{AC}^B:= \big |\{L\subseteq B \mid L \cong C,  B/L\cong A\} \big |.
	\end{align}
	The Riedtmann-Peng formula states that
	\[
	F_{AC}^B= \frac{|\Ext^1(A,C)_B|}{|\Hom(A,C)|} \cdot \frac{|\Aut(B)|}{|\Aut(A)| |\Aut(C)|}.
	\]
	
	For any object $A$, 
	let
	\begin{align*}
		[\![A]\!]:=\frac{[A]}{|\Aut(A)|}.
	\end{align*}
	Then the Hall multiplication \eqref{eq:mult} can be reformulated to be
	\begin{align}
		[\![A]\!]\diamond [\![C]\!]=\sum_{[\![B]\!]}F_{AC}^B[\![B]\!],
	\end{align}
	which is the version of Hall multiplication used in \cite{Rin90}.

	\subsection{Semi-derived Ringel-Hall algebras}

	Let $\bfk=\F_q$ and $\sqq=\sqrt{q}$. 
	Let $\widetilde{\ch}(\cc_{\Z_2}(\ca))$ be the twisted Ringel-Hall algebra of $\cc_{\Z_2}(\ca)$ over $\Q(\sqq)$, that is, $\widetilde{\ch}(\cc_{\Z_2}(\ca))$ has a basis formed by the isoclasses $[M]$ of objects $M\in\cc_{\Z_2}(\ca)$, with the multiplication given by
	\begin{align}\label{multiplication formula}
		[L]* [M]=&\sqq^{\langle \res L,\res M\rangle}[L]\diamond [M]
		\\\notag
		=&\sqq^{\langle \res L,\res M\rangle}\sum_{[X]\in \Iso(\cc_{\Z_2}(\ca))}\frac{|\Ext^1_{\cc_{\Z_2}(\ca)}(L,M)_X|}{|\Hom_{\cc_{\Z_2}(\ca)}(L,M)|}[X].
	\end{align}
	
	Let $I_{\Z_2}$ be the two-sided ideal of $\widetilde{\ch}(\cc_{\Z_2}(\ca))$ generated by all differences $[L]-[K\oplus M]$ if there is a short exact sequence in $\cc_{\Z_2}(\ca)$ as follows with $K$ acyclic:
	\begin{equation*}
		\label{eq:ideal}
		0 \longrightarrow K \longrightarrow L \longrightarrow M \longrightarrow 0.
	\end{equation*}
	Let $\widetilde{\ch}(\cc_{\Z_2}(\ca))/I_{\Z_2}$ be the quotient algebra. 
	We also denote by $*$ the induced multiplication in $\widetilde{\ch}(\cc_{\Z_2}(\ca))/I_{\Z_2}$. In the following, we shall use the same symbols both in $\widetilde{\ch}(\cc_{\Z_2}(\ca))$ and $\widetilde{\ch}(\cc_{\Z_2}(\ca))/I_{\Z_2}$.
	Let
	\begin{align}
		\label{multiset}
		S_{\Z_2}:=\{a[K] \in \widetilde{\ch}(\cc_{\Z_2}(\ca))/I_{\Z_2} \mid  a \in \Q(\sqq)^\times, K\in\cc_{\Z_2,ac}(\ca)\},
	\end{align}
	which is a multiplicatively closed subset with the identity $[0]\in S_{\Z_2}$.

	\begin{proposition}
		[\cite{LP16}]
		\label{proposition localizaition of Ringel-Hall algebra}
		The multiplicatively closed subset $S_{\Z_2}$ is a right Ore, right reversible subset of $\widetilde{\ch}(\cc_{\Z_2}(\ca))/I_{\Z_2}$. Equivalently, the right localization of $\widetilde{\ch}(\cc_{\Z_2}(\ca))/I_{\Z_2}$ with respect to $S_{\Z_2}$ exists, which is called the twisted semi-derived Ringel-Hall algebra of $\ca$, and denoted by $\tMH(\ca)$.
	\end{proposition}

	Let $\ct_{\Z_2}(\ca)$ be the subalgebra of $\tMH(\ca)$ generated by $[K]^{\pm1}$ for all acyclic complexes $K$. Then $\ct_{\Z_2}(\ca)$ is a commutative algebra; see \cite[\S4.3]{LP16}.
	Moreover, $\ct_{\Z_2}(\ca)$ is isomorphic to the group algebra of the Grothendieck group $K_0(\cc_{\Z_2,ac}(\ca))$.
	
	For any $\alpha\in K_0(\ca)$, there exist $A,B\in\ca$ such that $\alpha=\widehat{A}-\widehat{B}$, and we set
	\begin{align}
		\label{eq:Kalpha}
		[K_\alpha]:=[K_A]* [K_B]^{-1},\qquad [K_\alpha^*]:=[K_A^*]*[K_B^*]^{-1}
	\end{align}
	Then $[K_\alpha],[K_\alpha^*]$ are well defined in $\tMH(\ca)$ (and then in $\ct_{\Z_2}(\ca)$); see \cite[\S3.4]{LP16}.

	Let $\widetilde{\ch}(\ca)$ be the twisted Ringel-Hall algebra, that is, the same vector space as $\ch(\ca)$ equipped with the twisted multiplication
	$$[A]* [B]=\sqq^{\langle A,B\rangle}[A]\diamond [B]$$
	for $[A],[B]\in\Iso(\ca)$.
	Then the maps
	\begin{align}
		\label{eq:Rpm}
		R^+:\widetilde{\ch}(\ca)&\longrightarrow \tMH(\ca),\qquad\qquad R^-:\widetilde{\ch}(\ca)\longrightarrow \tMH(\ca),
		\\\notag
		[A]&\mapsto [C_A],\qquad\qquad\qquad\qquad\qquad [A]\mapsto [C_A^*],
	\end{align}
	are embeddings of algebras.
	Moreover, we have the following triangular decomposition
	\begin{align}
		\tMH(\ca)\cong \widetilde{\ch}(\ca)\otimes \ct_{\Z_2}(\ca)\otimes \widetilde{\ch}(\ca).
	\end{align}

	\begin{lemma}
		[\text{\cite[Lemma 4.3]{LP16}}]
		\label{lem:KC}
		For any $\alpha,\beta\in K_0(\ca)$, $M\in \ca$, we have
		\begin{align}
			\big[[K_\alpha],[K_\beta]\big]=&\big[[K_\alpha],[K_\beta^*]\big]=\big[[K_\alpha^*],[K_\beta^*]\big]=0,
			\\
			[K_\alpha] *[C_M]=&\sqq^{(\alpha,\widehat{M})}[C_M]* [K_\alpha], \qquad \qquad [K_\alpha] *[C_M^*]= \sqq^{-(\alpha,\widehat{M})} [C_M^*]*[K_\alpha],
			\\
			[K_\alpha^*]*[C_M]=&\sqq^{-(\alpha,\widehat{M})} [C_M]*[K_\alpha^*],\qquad\quad [K_\alpha^*]*[C_M^*]=\sqq^{(\alpha,\widehat{M})}[C_M^*]*[K_\alpha^*].
		\end{align}
	\end{lemma}
	
	Let $\widetilde{\ch}^e(\ca)$ be the extended Hall algebra, which is defined as an extension of $\widetilde{\ch}(\ca)$ by adjoining symbols $k_\alpha$ for classes $\alpha\in K_0(\ca)$, and imposing relations
	\begin{align}
		k_\alpha*k_\beta= k_{\alpha+\beta},\qquad k_\alpha*[X]= \sqq^{(\alpha,\widehat{X})}[X]*k_\alpha.  
	\end{align}
	Denote by $\cd\widetilde{\ch}(\ca)$ the Drinfeld double Hall algebra of $\ca$, which is defined on $\widetilde{\ch}^e(\ca)\otimes \widetilde{\ch}^e(\ca)$; see \cite{X97}. 
	\begin{theorem}[{\cite[Theorem 4.9]{LP16}}]
		\label{thm:SDH-DH-iso}
		The twisted semi-derived 
		Ringel-Hall algebra $\cs\cd\widetilde{\ch}(\ca)$ is ismorphic to the Drinfeld double Hall algebra $\cd\widetilde{\ch}(\ca)$ of $\ca$, by mapping 
		\begin{align}
			& [C_X]\mapsto [X]\otimes 1,\qquad [C_X^*]\mapsto 1\otimes [X],\qquad \forall X\in\ca;
			\\
			& [K_\alpha]\mapsto k_\alpha\otimes 1,\qquad [K_\alpha^*]\mapsto 1\otimes k_\alpha,\qquad \forall \alpha\in K_0(\ca).
		\end{align}
	\end{theorem}
	
	Following \cite{X97,BS12,BS13}, we define the double Hall algebra $\cd\widetilde{\ch}_{red}(\ca)$ to be the quotient	of $\cd\widetilde{\ch}(\ca)$  by the ideal generated by $k_\alpha\otimes 1-1\otimes k_\alpha^{-1}$ ($\alpha\in K_0(\ca)$). Correspondingly, the reduced semi-derived Ringel-Hall algebra $\cs\cd\widetilde{\ch}_{red}(\ca)$ is the quotient of $\cs\cd\widetilde{\ch}(\ca)$ by the ideal generated by $[K_\alpha]*[K_\alpha^{*}]-1$ ($\alpha\in K_0(\ca)$). Then we have 
	\begin{align}
		\label{eqn:isored}
		\cs\cd\widetilde{\ch}_{red}(\ca)\cong \cd\widetilde{\ch}_{red}(\ca).
	\end{align}

	\section{Semi-derived Ringel-Hall algebra of the projective line}
	\label{sec:P1}

	In this section, we shall use the semi-derived Ringel-Hall algebra of the projective line $\P_\bfk^1$ to realize the quantum affine $\mathfrak{sl}_2$. 
	\subsection{Drinfeld new presentation of $\tU_v(\widehat{\sll}_2)$}
	\label{subsec:Drsl2}
	
	We write the Drinfeld new presentation of $\tU_v(\widehat{\sll}_2)$ explicitly. From \S\ref{sec:QLA}, $\tUD_v(\widehat{\sll}_2)$ is generated by
	$X^{\pm}_{1,m}$, $H_{1,l}$ for $m\in\Z,n\in\Z^\times$, and the invertible elements $K_{1}$, $K_1'$, $C$ and $C'$, subject to the following relations:
	\begin{align}
		\label{Drsl2-1}
		&C,C' \text{ are central},
		\\
		\label{Drsl2-2}
		&[K_1,K_1']  =  [K_1,H_{1, n}] =[K_1',H_{1, n}]=0,
		\\
		\label{Drsl2-3}
		&[H_{1,\pm m},H_{1,\pm n}]=0,\qquad
		[H_{1,m},H_{1,-n}] = \delta_{m, n} \frac{[2m]}{m} \frac{C^m -C'^{m}}{v -v^{-1}}, \text{ for }m,n>0,
		\\
		\label{Drsl2-4}
		&K_1X_{1,m}^{\pm} =v^{\pm 2} X_{1,m}^{\pm} K_1,\qquad K'_1X_{1,m}^{\pm} =v^{\mp 2} X_{1,m}^{\pm} K'_1,
		\\
		\label{Drsl2-5}
		&[H_{1 ,m},X_{1, n}^{\pm}] =\begin{cases}
			\pm\frac{[2m]}{m} C^{ \frac{m\mp m}2} X_{1,m+n}^{\pm},&\text{ if }m>0,
			\\
			\pm\frac{[2m]}{m}C'^{-\frac{m\pm m}{2}} X_{1,m+n}^{\pm},&\text{ if }m<0,
		\end{cases}
		\\
		\label{Drsl2-6}
		&[X_{1, m}^+,X_{1, n}^-] = \frac{C^{-n} K_1\psi_{1,m+n} - C'^{m} K'_1 \varphi_{1,m+n}}{v-v^{-1}}, 
		\\
		\label{Drsl2-7}
		&X_{1,m+1}^{\pm} X_{1,n}^{\pm}-v^{\pm 2} X_{1,n}^{\pm} X_{1,m+1}^{\pm} =v^{\pm 2} X_{1,m}^{\pm} X_{1,n+1}^{\pm}- X_{1,n+1}^{\pm} X_{1,m}^{\pm}.
	\end{align}
	Here
	$\psi_{1,m}$ and $\varphi_{1,n}$ are defined via \eqref{exp h+}--\eqref{exp h-}.

	\subsection{Realizing  $\tUD_v(\widehat{\sll}_2)$}

	For any closed point $x\in\P_\bfk^1$, denote by $d_x$ the degree of the defining irreducible polynomial associated to $x$.   
	For a partition $\lambda=(\lambda_1,\lambda_2,\cdots)$, let $\ell(\lambda)$ be the number of nonzero parts in $\lambda$, and $|\lambda|=\sum_{j\geq1}\lambda_j$.	We introduce the following elements for $m\ge 1$ in $\tMH(\P^1_\bfk):=\tMH(\coh(\P^1_\bfk))$:
	\begin{align}
		\label{eq:hTm}
		\widehat{\psi}_{1,m}= &\frac{1}{(q-1)\sqq^{m}}\sum_{0\neq f:\co\rightarrow \co(m) } [C_{\coker(f)}],
		\\
		\widehat{\varphi}_{1,-m}=& \frac{1}{(q-1)\sqq^{m}}\sum_{0\neq f:\co\rightarrow \co(m) } [C^*_{\coker(f)}],
		\\
		\label{def:Theta2}
		\widehat{H}_{1,m}=& \sum_{x,d_x|m} \frac{[m]_\sqq}{m} d_x \sum_{|\lambda|=\frac{m}{{d_x}}} \bn_{d_x}(\ell(\lambda)-1)[\![C_{S_x^{(\lambda)}}]\!],
		\\
		\label{def:Theta3}
		\widehat{H}_{1,-m}=& - \sum_{x,d_x|m} \frac{[m]_\sqq}{m} d_x \sum_{|\lambda|=\frac{m}{{d_x}}} \bn_{d_x}(\ell(\lambda)-1)[\![C^*_{S_x^{(\lambda)}}]\!],
	\end{align}
	where 	\[\bn_{d_x}(l)=\prod_{j=1}^l(1-\sqq^{2d_xj}),\quad \forall l\geq0.\]
	We also set
	\[\widehat{\psi}_{1,0}= \widehat{\varphi}_{1,0}=1,\qquad \widehat{\psi}_{1,-m}=\widehat{\varphi}_{1,m}=0, \,\forall m>0.\]

	\begin{proposition}[\text{cf. \cite{Ka97,BKa01,BS12}}]
		\label{prop:Realizationsl2}
		There is an injective algebra homomorphism $\Omega: \tUD_\sqq(\widehat{\mathfrak{sl}}_2) \rightarrow \tMH(\P^1_\bfk)$, which sends
		\begin{align}
			K_1\mapsto [K_\co],\qquad K_1'\mapsto [K_\co^*],
			\qquad	C\mapsto [K_\de],\qquad C'\mapsto [K_\de^*],
			\\\label{real root via line bundles}
			X^+_{1,k}\mapsto \frac{-1}{q-1}[C_{\co(k)}],\qquad X^-_{1,k}\mapsto \frac{\sqq}{q-1}[C^*_{\co(-k)}],
			\\
			\psi_{1,m}\mapsto \widehat{\psi}_{1,m}   \qquad \varphi_{1,-m}\mapsto \widehat{\varphi}_{1,-m},
			\qquad
			H_{1,l}\mapsto \widehat{H}_{1,l},
		\end{align}
		for any $k\in\Z,m>0,l\in\Z^\times$.
	\end{proposition}
	\begin{proof}
		In \cite{BS12}, the homomorphism from $\UD_\sqq(\widehat{\mathfrak{sl}}_2)$ to the double Hall algebra $\cd\widetilde{\ch}_{red}(\P^1_\bfk)$ (equivalently $\cs\cd\widetilde{\ch}_{red}(\P^1_\bfk)$ by \eqref{eqn:isored}) is established. Based on Theorem \ref{thm:SDH-DH-iso}, although we consider Drinfeld double $\tUD_\bfk(\widehat{\sll}_2)$, the proof therein can be easily generalized to the current setting, and is therefore omitted here.  
	\end{proof}

	\section{Semi-derived Ringel-Hall algebras of cyclic quivers}
	\label{sec:cyclic}

	Recall $d_i$ for $1\leq i\leq \bt$. Let $\bfk_i=\F_{q^{d_i}}$. 
	Let $C_{p_i}$ be the cyclic quiver with the vertex set $\Z_{p_i}=\{0,1,2,\dots,p_i-2,p_i-1\}$; see \eqref{fig:Cn}.
	Recall that $\rep^{\rm nil}_{\bfk_i}(C_{p_i})$ is the category of finite-dimensional nilpotent representations of $C_{p_i}$ over $\bfk_i$. For the Grothendieck group $K_0(\rep^{\rm nil}_{\bfk_i}(C_{p_i}))$, we denote $\alpha_j=\widehat{S_j}$ and $d_i\de=\sum_{j=1}^{p_i}\widehat{S_j}$ by abusing notations. For any positive real root $\beta$ of $\widehat{\mathfrak{sl}}_{p_i}$, by Gabriel-Kac Theorem, we denote by $M(\beta)$ the unique 
	(up to isomorphism) indecomposable object in $\rep^{\rm nil}_{\bfk_i}(C_{p_i})$ with its class $\beta$ in $K_0(\rep^{\rm nil}_{\bfk_i}(C_{p_i}))$.

	\subsection{Root vectors}
	Let $\tMH(\bfk_i C_{p_i})$ be the twisted semi-derived Ringel-Hall algebra of $\rep^{\rm nil}_{\bfk_i}(C_{p_i})$. The following proposition is a generalization of Bridgeland's result \cite{Br13} (without assuming that the quiver is acyclic). Set
	\begin{align*}
		\sqq_i=\sqq^{d_i},\quad \forall 1\leq i\leq \bt.
	\end{align*}
	\begin{proposition}[cf. \text{\cite[Theorem 4.9]{LP16},\cite[Theorem 5.8]{X97}}]
		\label{lem:UslnSDH}
		For the cyclic quiver $C_{p_i}$, there exists an algebra embedding
		\begin{align}
			\label{eq:semiXi}
			&\widetilde{\psi}_{C_{p_i}}:\tU_{\sqq_i}(\widehat{\mathfrak{sl}}_{p_i})\longrightarrow \cs\cd\widetilde{\ch}(\bfk_i C_{p_i})
			\\
			E_j\mapsto&\frac{-1}{q^{d_i}-1}[C_{S_j}], \qquad F_j\mapsto  \frac{\sqq_i}{q^{d_i}-1}[C_{S_{j}}^*],
			\\
			\tK_j\mapsto&[K_{S_j}], \qquad \qquad \tK_j'\mapsto[K^*_{S_j}],
			\qquad\forall 0\leq j\leq p_i-1.
		\end{align}
	\end{proposition}


	Let $\tU:=\tU_{v}(\widehat{\mathfrak{sl}_{p_i}})$ be the universal affine quantum group of type $A_{p_i-1}$, and $\tUD$ be its Drinfeld new presentation throughout this section; see \cite[\S5.2]{LR24b}, cf. \cite{Be94,Da12}.

	Recall that $\Phi:\tUD\rightarrow\tU$ is the Drinfeld-Beck isomorphism. We choose the sign function $o(\cdot)$ such that $o(j)=(-1)^j$ for any $1\leq j\leq p_i-1$ throughout this paper.
	
	Define
	\begin{align}\label{the map Psi A}
		\Omega_{C_{p_i}}:=\widetilde{\psi}_{C_{p_i}}\circ \Phi:\tUD_{\sqq_i}\longrightarrow \tMH(\bfk_i C_{p_i}).
	\end{align}
	Then $\Omega_{C_{p_i}}$ sends:
	\begin{align}
		\label{eq:HaDrA1}K_j\mapsto [K_{S_j}],\qquad K_j' \mapsto [K_{S_j}^*],\qquad
		C\mapsto [K_{{d_i}\de}],\qquad C'\mapsto [K_{{d_i}\de}^*],
		\\
		X^+_{j,0}\mapsto \frac{-1}{q^{d_i}-1}[C_{S_j}], \qquad X^-_{j,0}\mapsto \frac{\sqq_i}{q^{d_i}-1}[C_{S_j}^*],\qquad \forall \,\,1\leq j\leq p_i-1.
	\end{align}
	For any $1\leq j\leq p_i-1$, $l\in\Z$ and $r\geq 1$, we define
	\begin{align}
		\label{def:haB}
		\haX^+_{j,l}:=(1-q^{d_i})\Omega_{C_{p_i}}(X^+_{j,l}),&\qquad \haX^-_{j,l}:=(\sqq_i-\sqq_i^{-1})\Omega_{C_{p_i}}(X^-_{j,l}), \qquad \widehat{\bh}_{j,l}:= \Omega_{C_{p_i}}(H_{j,l}),
		\\
		\label{def:haphi}
		\widehat{\psi}_{j,r}:=& \Omega_{C_{p_i}}(\psi_{j,r}),\qquad\widehat{\varphi}_{j,-r}:= \Omega_{C_{p_i}}(\varphi_{j,-r}).
	\end{align}
	In particular, $\haX^+_{j,0}=[C_{S_j}]$, $\haX^-_{j,0}=[C^*_{S_j}]$ for any $1\leq j\leq p_i-1$.

	It is interesting to give explicit formulas for all root vectors in $\tMH(\bfk_i C_{p_i})$, which is difficult in general. We describe some of them below.

	For $1\leq j\leq p_i$ and any $\alpha\in K_0(\rep^{\rm nil}_{\bfk_i} (C_{p_i}))$, set
	\begin{align}
		\label{def:Mjalpha}
		\cm_{j,\alpha}:=\{[M]\mid \widehat{M}=\alpha,{\rm soc}(M)\subseteq S_1\oplus  \cdots \oplus S_j\}.
	\end{align}
	
	\begin{proposition}[\cite{Hu}]
		\label{prop:DrGenA}
		For any $1\leq j\leq p_i-1$, we have
		\begin{align}
			&\widehat{X}^+_{j,-1}=\sqq_i^{2-j}\sum\limits_{\cm_{j+1,d_i\delta-\alpha_{j}}}(-1)^{\dim_{\bfk_i}\End(M)}[K^*_{d_i\de-\alpha_j}]^{-1}*[C_M^*],
			\\
			\label{eqn:Xj1-}
			&\widehat{X}^-_{j,1}=\sqq_i^{2-j}\sum\limits_{\cm_{j+1,d_i\delta-\alpha_{j}}}(-1)^{\dim_{\bfk_i}\End(M)}[K_{d_i\de-\alpha_j}]^{-1}*[C_M],
		\end{align}
	\end{proposition}

	For any $1\leq j\leq p_i$ and $r>0$, we define
	\begin{align}
		c_{j,r}^+:=& (-1)^r \sqq_i^{-2jr}\sum_{[M]\in \cm_{j,rd_i\de}} (-1)^{\dim_{\bfk_i} \End(M)}[C_M],
		\\
		c_{j,r}^-:= &(-1)^r \sqq_i^{-2jr}\sum_{[M]\in \cm_{j,rd_i\de}} (-1)^{\dim_{\bfk_i} \End(M)}[C^*_M].
	\end{align}
	
	For convenience we set $c^\pm_{j,0}:=1$.
	The elements $\pi_{j,r}^{\pm}$ can be defined recursively in terms of $c_{j,r}^{\pm}$ via
	$$rc^\pm_{j,r}=(\sqq_i-\sqq_i^{-1}) \sum_{s=1}^r s\sqq_i^{-js} \pi^\pm_{j,s}*c^\pm_{j,r-s}.$$
	In particular, we have
	\begin{align}
		\label{eqn:pici1}
		\pi_{j,1}^\pm=\frac{\sqq_i^{j}c_{j,1}^{\pm}}{\sqq_i-\sqq_i^{-1}}.
	\end{align}
	Moreover, by \cite[(4.1)]{Sch04} (see also \cite{Hu}), for any $r>0$ we have
	\begin{align}
		\label{eqn:pi1r}\pi_{1,r}^+
		=\frac{[r]_{\sqq_i}}{r}\sum\limits_{|\lambda|=r} \bn_{d_i}(\ell(\lambda)-1)[\![C_{S_{0}^{(\lambda p_i)}}]\!],\qquad \pi_{1,r}^-
		=\frac{[r]_{\sqq_i}}{r}\sum\limits_{|\lambda|=r} \bn_{d_i}(\ell(\lambda)-1)[\![C^*_{S_{0}^{(\lambda p_i)}}]\!].\end{align}
	Here,  $S_{0}^{(\lambda p_i)}:=\bigoplus_{j}S_{0}^{(\lambda_jp_i)}$ for any partition $\lambda=(\lambda_1, \lambda_2,\cdots)$.
	
	For convenience, set $\pi^{\pm}_{0,r}:=0$ for $r>0$.
	
	\begin{lemma}
		[\cite{Hu}]
		\label{lem:Hjr-pi}
		For any $r>0$, we have
		\begin{align}
			\label{eqn:Hjr-pi}\widehat{H}_{j,\pm r}=&\pm\big(\pi_{j+1,r}^\pm-(\sqq_i^{r}+\sqq_i^{-r}) \pi_{j,r}^\pm+\pi_{j-1,r}^\pm\big).
		\end{align}
	\end{lemma}
	
	\subsection{Embedding from $C_n$ to $C_{n+1}$}
	\label{subsec:embedding-cyclic}
	
	We consider the oriented cyclic quiver $C_n$ for $n\geq2$ with the vertex set $\Z_n=\{0,1,2,\dots,n-2,n-1\}$; see \eqref{fig:Cn}.
	Let $C_1$ be the Jordan quiver, i.e., the quiver with only one vertex and one loop arrow.

	There is a natural embedding $$\F_{n+1,n}: \rep^{\rm nil}_{\bfk_i}(C_n)\longrightarrow \rep^{\rm nil}_{\bfk_i}(C_{n+1}),$$ preserving the simples $S_j$ for $1\leq j\leq n-1$, and sending $S_0$ to $S_{0}^{(2)}$.
	Note that $\langle S_{n}\rangle=\add S_n$ is a Serre subcategory of $\rep^{\rm nil}_{\bfk_i}(C_{n+1})$. 
	We have an equivalence $$\rep^{\rm nil}_{\bfk_i}(C_{n+1})/\langle S_{n}\rangle\cong{}^{\perp}S_{n}= 
	S_{0}^{\perp} \cong
	\rep^{\rm nil}_\bfk(C_n).$$ 
	Here $S_0^\perp=\{M\in \rep^{\rm nil}_{\bfk_i}(C_{n+1})\mid \Hom(S_0,M)=0=\Ext^1(S_0,M)\}$ is the perpendicular category.
	Moreover, the natural quotient functor $$\rep^{\rm nil}_{\bfk_i}(C_{n+1})\longrightarrow \rep^{\rm nil}_{\bfk_i}(C_{n+1})/\langle S_{n}\rangle\cong \rep^{\rm nil}_{\bfk_i}(C_n)$$ yields the right adjoint functor of $\F_{n+1,n}$.

	Note that $\F_{n+1,n}$ is an exact functor, hence it induces an algebraic embedding $$
	F_{n+1,n}:\tMH(\bfk_i C_{n})\longrightarrow \tMH(\bfk_i C_{n+1})
	.$$  
	
	Observe that the functor $\F_{n+1,n}$ maps 
	$$S_j^{(an)}\mapsto S_j^{(a(n+1))}, \quad\forall\;\; 0\leq j\leq n-1, a\geq 1.$$ Then by Lemma \ref{lem:Hjr-pi}, $F_{n+1,n}$ preserves $\widehat{H}_{1,\pm1}$.
	Obviously, $F_{n+1,n}$ preserves  $\widehat{X}^{\pm}_{j,0}=$ for $1\leq j\leq n-1$. It follows that $F_{n+1,n}$ preserves all root vectors; see Remark \ref{rem:gen-reduce}.



	
	

	\section{Semi-derived Ringel-Hall algebra and quantum loop algebras}
	\label{sec:hom}
	
	In this section, we shall formulate the main result of this paper.

	\subsection{Embedding from a tube to the category of coherent sheaves}
	\label{subsec:embeddingtube}
	
	\label{subsec:embeddingtube}
	Let $\X=\X_{\bfk}$ be a weighted projective line of weight type $(\bp,\bd,\bla)$ over $\bfk$.
	Recall that $\bla_i$ is the exceptional closed point of $\X$ of degree $d_i$ and weight $p_i$ for any $1\leq i\leq \bt$, and $\scrt_{\bla_i}$ is the Serre subcategory of $\coh(\X)$ consisting of torsion sheaves supported at $\bla_i$. Note that there is an equivalence $\scrt_{\bla_i}\cong\rep^{\rm nil}_{\bfk_i}(C_{p_i})$, which induces a canonical embedding of Hall algebras
	$\widetilde{\ch}(\cc_{\Z_2}(\rep_{\bfk_i}^{\rm nil} (C_{p_i})))\rightarrow \widetilde{\ch}(\cc_{\Z_2}(\coh(\X_\bfk)))$, and then an embedding of semi-derived Ringel-Hall algebras:
	\begin{equation}
		\label{eq:embeddingx}
		\iota_i: \tMH(\bfk_i C_{p_i})\longrightarrow\tMH(\X_\bfk).
	\end{equation}
	
	Inspired by \eqref{def:haB}, we define
	\begin{align}
		\label{def:haBThH}
		\widehat{X}^+_{[i,j],d_il}:= \iota_i (\widehat{X}^+_{j,l}),&\qquad \widehat{X}^-_{[i,j],d_il}:= \iota_i (\widehat{X}^-_{j,l}),\qquad \widehat{\bh}_{[i,j],d_il}:=\iota_i(\widehat{\bh}_{j,l}),\\
		\label{def:haBThHpsi}
		\widehat{\psi}_{[i,j],d_ir}:=&\iota_{i}(\widehat{\psi}_{j,r}), \qquad \widehat{\varphi}_{[i,j],-r}:=\iota_{i}(\widehat{\varphi}_{j,-r})
	\end{align}
	for any $1\leq j\leq p_i-1$, $l\in\Z$ and $r>0$.

	%
	%
	\subsection{Embedding from projective line to weighted projective line}\label{Embedding from projective line to weighted projective line}
	
	Let $\mathcal{C}$ be the Serre subcategory of $\coh(\X_\bfk)$ generated by those simple sheaves $S$ satisfying $\Hom(\co, S)=0$. Recall from \cite{BS13} that the Serre quotient $\coh(\X_\bfk)/\mathcal{C}$ is equivalent to $\coh(\P^1_\bfk)$ and the canonical functor $\coh(\X_\bfk)\rightarrow\coh(\X_\bfk)/\mathcal{C}$ has an exact fully faithful right adjoint functor
	\begin{equation}
		\label{the embedding functor F}
		\mathbb{F}_{\X,\P^1}: \coh(\P^1_\bfk)\rightarrow\coh(\X_\bfk),
	\end{equation}
	which sends $$\co_{\P^1}(l)\mapsto\co(l\vec{c}),\quad S_{\bla_i}^{(r)}\mapsto S_{i,0}^{(rp_i)},\quad S_{x}^{(r)}\mapsto S_{x}^{(r)}$$ 
	for any $l\in\Z, r\geq 1$, $1\leq i\leq \bt$ and $x\in\PL\setminus\{\bla_1,\cdots, \bla_\bt\}$. Then $\mathbb{F}_{\X,\P^1}$ induces an exact fully faithful functor $\cc_{\Z_2}(\coh(\P^1_\bfk))\rightarrow \cc_{\Z_2}(\coh(\X_\bfk))$, which is also denoted by $\mathbb{F}_{\X,\P^1}$. This functor induces a canonical embedding of Hall algebras $\ch\big(\cc_{\Z_2}(\coh(\P^1_\bfk))\big)\rightarrow \ch\big(\cc_{\Z_2}(\coh(\X_\bfk))\big)$, and then an embedding
	\begin{equation}
		\label{the embedding functor F on algebra}
		F_{\X,\P^1}:\tMH(\P^1_\bfk) \longrightarrow \tMH(\X_\bfk).
	\end{equation}
	
	We define for $r>0$: 
	\begin{align}
		\label{eqn:hTm}    \widehat{\psi}_{\star,r}:=&F_{\X,\P^1}(\widehat{\psi}_{1,r}),\qquad \widehat{\varphi}_{\star,-r}:=F_{\X,\P^1}(\widehat{\varphi}_{1,-r}),\qquad \widehat{\bh}_{\star,\pm r}:=F_{\X,\P^1}(\widehat{\bh}_{1,\pm r}).
	\end{align}
	In particular, 
	\begin{align}
		\label{formula for Hm}
		\widehat{\bh}_{\star,m}
		=\begin{cases}
			\sum_{x,d_x|m} \frac{[m]_\sqq}{m} d_x \sum_{|\lambda|=\frac{m}{{d_x}}} \bn_{d_x}(\ell(\lambda)-1)
			F_{\X,\P^1}([\![C_{S_x^{(\lambda)}}]\!] )&\text{ if }m>0,\\
			-\sum_{x,d_x|m} \frac{[m]_\sqq}{m} d_x \sum_{|\lambda|=\frac{m}{{d_x}}} \bn_{d_x}(\ell(\lambda)-1)
			F_{\X,\P^1}([\![C^*_{S_x^{(\lambda)}}]\!] )
			&\text{ if }m<0. \end{cases}
	\end{align}

	

	%
	%
	\subsection{The homomorphism $\Omega$}
	\label{subsec:homo}
	
	Recall the star-shaped graph $\Gamma=\Gamma_{p_1,\dots,p_\bt}$ in \eqref{star-shaped}.
	The main result of this paper states that:
	
	\begin{theorem}
		\label{thm:morphi}
		For any valued star-shaped graph $(\Gamma,\bd)$, let $\fg$ be the Kac-Moody algebra and $\X$ be the weighted projective line associated to $\Gamma$. Then there exists a $\Q(\sqq)$-algebra homomorphism
		\begin{align}
			\Omega: \tU_\sqq(\Lg)\longrightarrow \tMH(\X_\bfk),
		\end{align}
		which sends
		\begin{align}
			\label{eq:mor1}
			&K_{\star}\mapsto [K_{\co}], \qquad K_{[i,j]}\mapsto [K_{S_{ij}}], \qquad C\mapsto [K_\de],
			&
			\\
			&K'_{\star}\mapsto [K_{\co}^*], \qquad K'_{[i,j]}\mapsto [K_{S_{ij}}^*], \qquad C'\mapsto [K_\de^*],
			\\
			\label{eq:mor2}
			&{X^+_{\star,l}\mapsto \frac{-1}{q-1}[C_{\co(l\vec{c})}]},\qquad{X^-_{\star,l}\mapsto \frac{\sqq}{q-1}[C^*_{\co(-l\vec{c})}]},
			\\
			&\psi_{\star,r} \mapsto {\widehat{\psi}_{\star,r}}, \qquad \varphi_{\star,-r} \mapsto {\widehat{\varphi}_{\star,-r}}, \qquad
			\bh_{\star,m} \mapsto {\widehat{\bh}_{\star,m}},\\
			\label{eq:mor3}
			&X^+_{[i,j],d_il}\mapsto {\frac{-1}{q^{d_i}-1}}\haX^+_{[i,j],d_il}, \qquad X^-_{[i,j],d_il}\mapsto {\frac{\sqq_i}{q^{d_i}-1}}\haX^-_{[i,j],d_il},
			\\
			\label{eq:mor4}
			&\psi_{[i,j],d_ir} \mapsto {\widehat{\psi}_{[i,j],d_ir}}, \qquad \varphi_{[i,j],-d_ir} \mapsto {\widehat{\varphi}_{[i,j],-d_ir}}, \qquad
			H_{[i,j],d_im} \mapsto {\widehat{\bh}_{[i,j],d_im}}.
		\end{align}
		for any $[i,j]\in\II-\{\star\}$, $l\in\Z$, $r>0$, $m\in\Z^\times$.
	\end{theorem}
	
	\begin{proof}
		It is enough to verify that all relations~\eqref{DR1}--\eqref{DR8} are preserved by $\Omega$. Similarly to \cite[\S8]{LR21} (or \cite{DJX12}), most of 
		these relations can be easily checked (hence omitted here)
		except those between $\star$ and $[i,1]$ in $\cs\cd\widetilde{\ch}(\X_\bfk)$, whose proofs are given in the following two sections; see \eqref{eq:hhcomm}--\eqref{DR72-Hall} and \eqref{DR7-Hall}.
	\end{proof}
	
	For convenience, we sometimes denote for any $l\in\Z$:
	\begin{align*}
		\widehat{X}^+_{\star,l}:=[C_{\co(l\vec{c})}],\qquad \widehat{X}^-_{\star,l}:=[C^*_{\co(-l\vec{c})}].
	\end{align*}
	
	
	
	\subsection{Embedding of Hall algebras of weighted projective lines}
	\label{sub:embedding-XY}
	This subsection is inspired by \cite[\S4]{BS13}.
	
	Let $\X$ be the weighted projective line of type $(\bp,\bd,\bla)$, where $\bp=(p_1,p_2,\cdots,p_\bt)$. Let $\Y$ be  the weighted projective line of type $(\bp',\bd,\bla)$, where $\bp'=(p_1,\cdots, p_{i-1},p_i+1,p_{i+1},\cdots,p_\bt)$.

	Note that $\langle S_{i,p_i}\rangle =\add S_{i,p_i}$ is a Serre subcategory of $\coh(\mathbb{Y})$, 
	and the quotient category $\coh(\mathbb{Y})/\langle S_{i,p_i}\rangle\cong {}^{\perp}S_{i,p_i}=S_{i0}^{\perp}$ is equivalent to the category $\coh(\X)$. 
	Therefore, the natural quotient functor $\pi: \coh(\mathbb{Y})\to \coh(\X)$ admits a fully faithful right adjoint functor $\iota: \coh(\X)\to \coh(\mathbb{Y})$, which sends $S_{ij}^{(ap_i)}\mapsto S_{ij}^{a(p_i+1)}$ for $0\leq j\leq p_i-1,a\geq1$, and preserves $S_{kj}^{(ap_i)}\; (k\neq i)$ and $\co(n\vc)\; (n\in\Z)$.
	In other words, the image of $\iota$ identifies $\coh(\X)$ with the subcategory $S_{i0}^{\perp}$ of $\coh(\mathbb{Y})$. So $\iota$ induces an embedding of $\cc_{\Z_2}(\coh(\X))$ in $\cc_{\Z_2}(\coh(\Y))$, which gives an injective homomorphism $\widetilde{\ch}(\coh(\X))\to \widetilde{\ch}(\coh(\Y))$, and then an embedding $\iota:\cs\cd\widetilde{\ch}(\X)\to \cs\cd\widetilde{\ch}(\Y)$.
	
	By definition of root vectors and \S\ref{subsec:embedding-cyclic}, we can see that $\iota:\cs\cd\widetilde{\ch}(\X)\to \cs\cd\widetilde{\ch}(\Y)$ preserves all real and imaginary roots.

	\section{Relations between root vectors of $\star$ and $[i,1]$}
	\label{sec:relationstari1}
	
	In this section, we shall verify the relations \eqref{DR2}--\eqref{DR6} between $\star$ and $[i,1]$ in $\cs\cd\widetilde{\ch}(\X_\bfk)$, for $l,m>0$, $r,t,m_1,m_2,n\in\Z$: 
	\begin{align}
		\label{eq:hhcomm}
		&[\widehat{H}_{\star,l},\widehat{H}_{[i,1],d_im}]=0,\quad [\widehat{H}_{\star,-l},\widehat{H}_{[i,1],-d_im}]=0,
		\\
		&
		\label{eq:hmhi1-l}
		\big[\widehat{H}_{\star,m},\widehat{H}_{[i,1],-d_il}\big]=\delta_{m,d_il}\frac{[-m]_\sqq}{m/d_i}\frac{[K_\de]^m-[K_\de^*]^m}{\sqq_i-\sqq_i^{-1}},
		\\
		\label{eq:hi1mh-l}
		&\big[\widehat{H}_{[i,1],d_im},\widehat{H}_{\star,-l}\big]=\delta_{d_im,l}\frac{[-m]_{\sqq_i}}{m}\frac{[K_{d_i\de}]^m-[K_{d_i\de}^*]^m}{\sqq-\sqq^{-1}},
		\\
		& \label{eq:hrxm+}
		\big[\widehat{H}_{\star, m}, \widehat{X}^+_{[i,1],d_it}\big]=
		- \frac{[m]_{\sqq}}{m/d_i} \widehat{X}^+_{[i,1],d_it+ m},\quad \big[\widehat{H}_{\star,- m}, \widehat{X}^-_{[i,1],d_it}\big]=
		\frac{[m]_{\sqq}}{m/d_i} \widehat{X}^-_{[i,1],d_it- m},
		\\
		\label{eq:h-rxm+}
		&\big[\widehat{H}_{\star,-m}, \widehat{X}^+_{[i,1],d_it}\big]=-\frac{[m]_{\sqq}}{m/d_i} [K_{\de}^*]^m*\widehat{X}^+_{[i,1],d_it-m},
		\\
		\label{eq:h-rx-m+}
		&\big[\widehat{H}_{\star,m}, \widehat{X}^-_{[i,1],d_it}\big]=\frac{[m]_{\sqq}}{m/d_i} [K_{\de}]^m*\widehat{X}^-_{[i,1],d_it+m},
		\\
		&
		\label{eq:hrxk+}
		\big[\widehat{H}_{[i,1], d_im}, \widehat{X}^+ _{\star,t}\big]=-\frac{[m]_{\sqq_i}}{m} \widehat{X}^+_{\star,t+ d_im}, \quad
		\big[\widehat{H}_{[i,1],-d_im}, \widehat{X}^-_{\star,t}\big]=\frac{[m]_{\sqq_i}}{m}\widehat{X}^-_{\star,t-d_im},
		\\
		\label{Hiqxstar1}
		&\big[ \widehat{\bh}_{[i,1],-d_im}, \widehat{X}^+_{\star,t}\big]=-\frac{[m]_{\sqq_i}}{m} [K_{d_i\de}^*]^{m}* \widehat{X}^+_{\star,t-d_im},
		\\
		\label{Hiqxstar2}
		& \big[ \widehat{\bh}_{[i,1],d_im}, \widehat{X}^-_{\star,t}\big]=\frac{[m]_{\sqq_i}}{m}[K_{d_i\de}]^m* \widehat{X}^-_{\star,t+d_im},
		\\
		&		\big[\widehat{X}^+_{\star,t}, \widehat{X}^-_{[i,1],d_ir}\big]=0,
		\qquad 
		\label{x * and x_i1^+}
		\big[\widehat{X}^-_{\star,t}, \widehat{X}^+_{[i,1],d_ir}\big]=0,
		\\
		\label{eq:qcommxstarxi1+1}
		&\big[\widehat{X}^\pm_{\star,t+d_i}, \widehat{X}^\pm_{[i,1],d_ir} \big]_{\sqq_i^{\mp 1}}=\sqq_i^{\mp 1}\big[\widehat{X}^\pm_{\star,t}, \widehat{X}^\pm_{[i,1],d_i(r+1)}\big]_{\sqq_i^{\pm 1}},
		\\
		\label{DR72-Hall}
		&\Sym_{m_1,m_2}\big\{ \widehat{X}^\pm_{[i,1], m_1}*\widehat{X}^\pm_{[i,1], m_2} *\widehat{X}^\pm_{\star,n} -[2]_{\sqq_i}  \widehat{X}^\pm_{[i,1], m_1}*\widehat{X}^\pm_{\star,n}*\widehat{X}^\pm_{[i,1], m_2}
		\\
		\notag&\hspace{7cm}+\widehat{X}^\pm_{\star,n}*\widehat{X}^\pm_{[i,1], m_1}*\widehat{X}^\pm_{[i,1], m_2}  \big\}=0.
	\end{align}
	Here we set $\widehat{X}^\pm_{[i,1],t}=0$ if $d_i\nmid t$.
	
	%
	\subsection{The relation \eqref{eq:hhcomm}}
	
	For any closed point $x\in\X$, we define
	\begin{align}
		\label{def:Hxm}
		\widehat{H}_{x,m}:= \begin{cases}
			\frac{[m]_\sqq}{m} d_x \sum_{|\lambda|=\frac{m}{{d_x}}} \bn_x(\ell(\lambda)-1) F_{\X,\P^1}([\![C_{S_x^{(\lambda)}}]\!]) & \text{ if }m>0,
			\\
			\frac{[m]_\sqq}{m} d_x \sum_{|\lambda|=\frac{m}{{d_x}}} \bn_x(\ell(\lambda)-1)F_{\X,\P^1}([\![C_{S_x^{(\lambda)}}^*]\!]) & \text{ if }m<0,
		\end{cases}
	\end{align}
	if $d_x\mid m$; and $\widehat{H}_{x,m}=0$ otherwise.
	Then we have 
	\begin{align}
		\label{def:Hxm2}
		\widehat{H}_{\star,m}=\sum_{x\in\X} \widehat{H}_{x,m}
		\text{\quad for\; any\;} m\in\Z.
	\end{align}
	
	\begin{lemma}
		For $1\leq i\leq \bt$, $m>0$, we have 
		
		\begin{align}
			\label{eq:Hstar-pi}
			\widehat{H}_{\bla_i,\pm d_im}={[d_i]_\sqq}\cdot \pi_{[i,1],m}^\pm.
		\end{align}
		
	\end{lemma}
	
	\begin{proof}
		For any partition $\lambda=(\lambda_1, \lambda_2,\cdots)$, $\F_{\X,\P^1}(S_{\bla_i}^{(\lambda)})=\bigoplus_{j}S_{i,0}^{(\lambda_jp_i)}=:S_{i,0}^{(\lambda p_i)}$. Then by \eqref{eqn:pi1r}, $$\pi_{[i,1],m}^+
		=\frac{[m]_{\sqq_i}}{m}\sum\limits_{|\lambda|=m} \bn_{d_i}(\ell(\lambda)-1)
		[\![C_{S_{i,0}^{(\lambda p_i)}}]\!].$$ 
		Note that
		$$[m]_{\sqq_i}\cdot[d_i]_\sqq=\frac{\sqq_i^{m}-\sqq_i^{-m}}{\sqq_i-\sqq_i^{-1}}\cdot \frac{\sqq_i-\sqq_i^{-1}}{\sqq-\sqq^{-1}}=[d_im]_\sqq.$$
		Hence $\widehat{H}_{\bla_i, d_im}={[d_i]_\sqq}\cdot \pi_{[i,1],m}^+$. The other case is similar.
	\end{proof}

	\begin{lemma}
		We have \eqref{eq:hhcomm} holds 
		for $1\leq i\leq \bt$, $l,m>0$.
	\end{lemma}

	\begin{proof}

		For any torsion sheaves $X,Y$ supported at distinct points, we have $\Ext^1_\X(X,Y)=0=\Hom_\X(X,Y)$.
		So in order to prove $[\widehat{H}_{\star,\pm l},\widehat{H}_{[i,1],\pm d_im}]=0$, we only need to check that
		\begin{align}
			[\widehat{H}_{\bla_i,\pm d_il},\widehat{H}_{[i,1],\pm d_im} ]=0.
		\end{align}
		By \cite[Lemma 7.1]{DJX12} and its proof, we know $$[\pi^{\pm}_{[i,j],l}, \pi^{\pm}_{[i,k],m}]=0$$
		for any $1\leq j,k\leq p_i-1$.
		So the desired formula follows by 
		\eqref{eqn:Hjr-pi} and \eqref{eq:Hstar-pi}. 
	\end{proof}

	\subsection{The relations \eqref{eq:hrxm+}--\eqref{eq:h-rx-m+}}
	
	\begin{lemma}
		\label{lem:hrxm+--h-rx-m+}
		We have \eqref{eq:hrxm+}--\eqref{eq:h-rx-m+} hold for any $m>0,t\in\Z$.
	\end{lemma}
	
	\begin{proof}
		We only prove the first formula of \eqref{eq:hrxm+} and \eqref{eq:h-rxm+}.
		
		
		
		By \eqref{eq:Hstar-pi} we obtain
		\begin{align*}
			\big[\widehat{H}_{\star,d_im}, \widehat{X}^+_{[i,1],0}\big]
			=\big[\widehat{H}_{\bla_i,d_im}, \widehat{X}^+_{[i,1],0}\big]=\big[[d_i]_{\sqq}\cdot\pi^+_{[i,1],m},\widehat{X}^+_{[i,1],0}\big].
		\end{align*}
		Note that
		$\widehat{H}_{[i,1],d_im}=\pi^+_{[i,2],m}-(\sqq_i^{m}+\sqq_i^{-m})\pi^+_{[i,1],m}$ and  $\big[\pi^+_{[i,2],m},[C_{S_{i,1}}]\big]=0$. 
		Hence
		\begin{align*}
			\big[\widehat{H}_{\star,d_im}, \widehat{X}^+_{[i,1],0}\big]
			&=\frac{-[d_i]_{\sqq}}{\sqq_i^{m}+\sqq_i^{-m}}\big[\widehat{\bh}_{[i,1],d_im},\widehat{X}^+_{[i,1],0} \big]\\
			&=\frac{-[d_i]_{\sqq}}{\sqq_i^{m}+\sqq_i^{-m}} \frac{[2m]_{\sqq_i}}{m}\widehat{X}^+_{[i,1],d_im}
			\\
			&=\frac{-[d_im]_\sqq}{m} \widehat{X}^+_{[i,1],d_im},
		\end{align*}
		where the second equality holds since  $\Omega_{C_{p_i}}$ preserves \eqref{DR4} of $\tUD_v(\widehat{\mathfrak{sl}}_{p_i})$. 
		
		Using embedding as in \S\ref{sub:embedding-XY}, we can assume $p_i\geq 3$.
		Additionally, $[\widehat{H}_{\star,d_im}, \widehat{H}_{[i,2],d_it}]=0$ for any $t\neq 0$; see \cite[Lemma 7.1 (2) and Section 8.2]{DJX12}.
		Then we have 
		\begin{align*}
			\big[\widehat{H}_{\star,d_im}, \widehat{X}^+_{[i,1],d_it}\big]=&
			-\frac{t}{[t]_{\sqq_i}}\big[\widehat{H}_{\star,d_im}, [\widehat{H}_{[i,2],d_it},\widehat{X}^+_{[i,1],0}]\big]
			\\
			=& -\frac{t}{[t]_{\sqq_i}}\big[\widehat{H}_{[i,2],d_it}, [\widehat{H}_{\star,d_im},\widehat{X}^+_{[i,1],0}]\big]
			\\
			=&\frac{t}{[t]_{\sqq_i}} \frac{[d_im]_{\sqq}}{m}\big[\widehat{H}_{[i,2],d_it}, \widehat{X}^+_{[i,1],d_im}\big]
			\\
			=&- \frac{[d_im]_{\sqq}}{m}\widehat{X}^+_{[i,1],d_i(m+t)}.
		\end{align*}

		Similarly,
		since
		$\widehat{H}_{[i,1],-d_im}=\pi^-_{[i,2],m}-(\sqq_i^{m}+\sqq_i^{-m})\pi^-_{[i,1],m}$ and  $\big[\pi^-_{[i,2],m}, \widehat{X}^+_{[i,1],-d_i}\big]=0$, cf. \cite{Hu}.
		Hence,
		\begin{align*}
			\big[\widehat{H}_{\star, -d_im}, \widehat{X}^+_{[i,1],-d_i}\big]
			&=\big[\widehat{H}_{\bla_i, -d_im}, \widehat{X}^+_{[i,1], -d_i}\big]\\
			&=\big[[d_i]_{\sqq}\cdot\pi^-_{[i,1],m},\widehat{X}^+_{[i,1], -d_i}\big]\\
			&=\frac{-[d_i]_{\sqq}}{\sqq_i^{m}+\sqq_i^{-m}}\big[\widehat{\bh}_{[i,1],-d_im},\widehat{X}^+_{[i,1],-d_i} \big]\\
			&=\frac{-[d_i]_{\sqq}}{\sqq_i^{m}+\sqq_i^{-m}} \frac{[2m]_{\sqq_i}}{m}[K^*_{\de}]^{md_i}*\widehat{X}^+_{[i,1],-d_i(m+1)}
			\\
			&=\frac{-[d_im]_\sqq}{m} [K^*_{\de}]^{md_i}*\widehat{X}^+_{[i,1],-d_i(m+1)}.
		\end{align*}
		By similar arguments as above, and using the fact $[\widehat{H}_{\star,-d_im}, \widehat{H}_{[i,2],d_it}]=0$ for any $t\neq 0$, one can prove that
		\begin{align*}
			\big[\widehat{H}_{\star, -d_im}, \widehat{X}^+_{[i,1],-d_it}\big]&=\frac{-[d_im]_\sqq}{m} [K^*_{\de}]^{md_i}*\widehat{X}^+_{[i,1],-d_i(t+m)}.
		\end{align*}
		This finishes the proof.
	\end{proof}

	%
	\subsection{The relation \eqref{eq:qcommxstarxi1+1}}
	
	\begin{lemma}
		We have \eqref{eq:qcommxstarxi1+1} holds for $1\leq i\leq \bt$, $r,t\in\Z$.
	\end{lemma}
	
	
	\begin{proof}

		
		First we assume $r=-1$. Recall 
		that $\widehat{X}^+_{\star,t}=[C_{\co(t\vc)}]$ for any $t\in\Z$,  $\widehat{X}^+_{[i,1],0}=[C_{S_{i,1}}]$ and  $\widehat{X}^+_{[i,1],-d_i}=-\sqq_i[K^*_{S_{i,0}^{(p_i-1)}}]^{-1}*[C^*_{S_{i,0}^{(p_i-1)}}]$ for any  $1\leq i\leq \bt$.
		A direct computation in $\tMH(\X_\bfk)$ shows that:
		
		\begin{align} \begin{split}\label{Si1 and co double direction1}
				[C_{\co(t\vec{c})}]*[C_{S_{i,1}}]=& [C_{S_{i,1}\oplus\co(t\vec{c})}],\\
				[C_{S_{i,1}]}* C_{[\co(t\vec{c})}]=& \sqq_i^{-1}[ C_{S_{i,1}\oplus\co(t\vec{c})}]+(\sqq_i-\sqq_i^{-1})[C_{\co(t\vec{c}+\vec{x}_i)}];
		\end{split}\end{align}
		and
		\begin{align} \begin{split}\label{Si0 and co double direction2}
				[C^*_{S_{i,0}^{(p_i-1)}}]* [C_{\co((t+d_i))\vec{c})}]=& [C^*_{S_{i,0}^{(p_i-1)}}\oplus C_{\co((t+d_i)\vec{c})}],\\
				[C_{\co((t+d_i)\vec{c})}]*[C^*_{S_{i,0}^{(p_i-1)}}]=&[C^*_{S_{i,0}^{(p_i-1)}}\oplus C_{\co((t+d_i)\vec{c})}]+(q^{d_i}-1)[C_{\co(t\vec{c}+\vec{x}_i)}]*[K^*_{ d_i\de-\alpha_{i}}].
		\end{split}\end{align}
		It follows that \eqref{eq:qcommxstarxi1+1} holds  when $r=-1$ since both sides  equal to $(\sqq_i^{-1}-\sqq_i)[C_{\co(t\vec{c}+\vec{x}_i)}]$.

		We also assume $p_i\geq3$. By applying $[\widehat{H}_{[i,2],d_ir},-]$, we know that \eqref{eq:qcommxstarxi1+1} holds for any $r$, similar to the proof of Lemma \ref{lem:hrxm+--h-rx-m+}. 
	\end{proof}

	\subsection{The relation \eqref{eq:hrxk+}}
	
	\begin{lemma}
		\label{lem:eq:hrxk+}
		We have \eqref{eq:hrxk+} holds for $m>0$, $t\in\Z$. 
	\end{lemma}
	
	\begin{proof}
		We only prove the first formula. 
		The proof method is similar as \cite[Lemma 7.6]{DJX12}.
		By \cite[Lemma 4.1]{Be94}, it is equivalent to prove
		\begin{align}
			\label{formula 8.19}
			\big[\widehat{\psi}_{[i,1],d_im}, [C_{\co(t\vec{c})}]\big]=\sqq_i^{-1}\big[\widehat{\psi}_{[i,1],d_i(m-1)},[C_{\co((t+d_i)\vec{c})}]  \big]_{\sqq_i^2}.
		\end{align}
		
		Since $\Omega_{C_{p_i}}$ preserves \eqref{DR5} of $\tUD_v(\mathfrak{sl}_{p_i})$, we know   that $$(1-q^{d_i})[K_{d_i\delta}]^{-1}*[K_{S_{i,1}}]*\widehat{\psi}_{[i,1],d_im}=[\widehat{X}^+_{[i,1],d_i(m-1)}, \widehat{X}^-_{[i,1],d_i}].$$
		It follows from \eqref{eqn:Xj1-} that 
		$$\widehat{X}^-_{[i,1],d_i}=-\sqq_i^{-1}[C_{S_{i,0}^{(p_i-1)}}]*[K_{S_{i,1}}]*[K_{d_i\delta}]^{-1}.$$
		Set $\xi^i_0=(q^{d_i}-1)\sqq_i$ and for $m>0$,
		\begin{align}
			\label{eq:xim}
			\xi^i_m:=(q^{d_i}-1)\sqq_i\cdot \widehat{\psi}_{[i,1],d_im}=\big[\widehat{X}^+_{[i,1],d_i(m-1)},[C_{S_{i,0}^{(p_i-1)}}]\big]_{\sqq_i^{2}}.
		\end{align}
		Then \eqref{formula 8.19}
		is equivalent to
		\begin{align}\label{xi and co}
			\big[\xi_m^i, [C_{\co(t\vec{c})}]\big]=\sqq_i^{-1}\big[\xi^i_{m-1},[C_{\co((t+d_i)\vec{c})}]  \big]_{\sqq_i^2}.
		\end{align}

		For $m=1$, we have
		\begin{align*}
			\xi_1^i&=\widehat{X}^+_{[i,1],0}*[C_{S_{i,0}^{(p_i-1)}}]-\sqq_i^{2}[C_{S_{i,0}^{(p_i-1)}}]*\widehat{X}^+_{[i,1],0}= 
			\big[  [C_{S_{i,1}}], [C_{S_{i,0}^{(p_i-1)}}]\big]_{\sqq_i^{2}}.
		\end{align*}
		By \eqref{Si1 and co double direction1} we have 
		$$\big[[C_{\co(t\vec{c})}], [C_{S_{i,1}}]\big]_{\sqq_i}=(1-q^{d_i})[C_{\co(t\vec{c}+\vec{x}_i)}];$$
		similarly,
		$$\big[[C_{\co(t\vec{c}+\vec{x}_i)}], [C_{S_{i,0}^{(p_i-1)}}]\big]_{\sqq_i}=(1-q^{d_i})[C_{\co((t+d_i)\vc)}].$$
		Hence,
		\begin{align*}
			\big[\xi_1^i, [C_{\co(t\vec{c})}]\big]&=-\Big[[C_{\co(t\vec{c})}], \big[  [C_{S_{i,1}}], [C_{S_{i,0}^{(p_i-1)}}]\big]_{\sqq_i^{2}}\Big]\\
			&=-\Big[\big[[C_{\co(t\vec{c})}],  [C_{S_{i,1}}]\big]_{\sqq_i},  [C_{S_{i,0}^{(p_i-1)}}]\Big]_{\sqq_i} \\
			&=-\Big[(1-q^{d_i})[C_{\co(t\vec{c}+\vec{x}_i)}],  [C_{S_{i,0}^{(p_i-1)}}]\Big]_{\sqq_i} \\
			&=-(1-q^{d_i})^2[C_{\co((t+d_i)\vc)}]. 
		\end{align*}
		Here, in the second equality we use the following identity
		$$\big[a,[b,c]_{\sqq_i^{2}}\big]=\big[[a,b]_{\sqq_i},c\big]_{\sqq_i}-\big[b,[c,a]_{\sqq_i}\big]_{\sqq_i},$$ 
		and the fact 
		$\big[[C_{S_{i,0}^{(p_i-1)}}],[C_{\co(t\vec{c})}]\big]_{\sqq_i}=0.$
		This proves \eqref{xi and co} when $m=1$.

		For $m\geq 2$,
		by \eqref{eq:qcommxstarxi1+1} and $\big[[C_{S_{i,0}^{(p_i-1)}}],[C_{\co(t\vec{c})}]\big]_{\sqq_i}=0$, we have
		\begin{align*}
			\big[\xi_m^i,C_{\co(t\vec{c})}\big]=&\Big[\big[\widehat{X}^+_{[i,1],d_i(m-1)},[C_{S_{i,0}^{(p_i-1)}}]\big]_{\sqq_i^{2}},[C_{\co(t\vec{c})}]\Big]
			\\
			=&\sqq_i\Big[\big[\widehat{X}^+_{[i,1],d_i(m-1)},[C_{\co(t\vec{c})}]\big]_{\sqq_i^{-1}}, [C_{S_{i,0}^{(p_i-1)}}]\Big]_{\sqq_i}
			\\
			=&\Big[\big[\widehat{X}^+_{[i,1],d_i(m-2)},[C_{\co((t+d_i)\vec{c})}]\big]_{\sqq_i}, [C_{S_{i,0}^{(p_i-1)}}]\Big]_{\sqq_i}
			\\
			=&\sqq_i^{-1}\Big[\big[\widehat{X}^+_{[i,1],d_i(m-2)}, [C_{S_{i,0}^{(p_i-1)}}]\big]_{\sqq_i^2}, [C_{\co((t+d_i)\vec{c})}] \Big]_{\sqq_i^2}
			\\
			=&\sqq_i^{-1}\big[\xi^i_{m-1},[C_{\co((t+d_i)\vec{c})}]  \big]_{\sqq_i^2}.
		\end{align*}
		The proof is completed. 
	\end{proof}

	\subsection{The relation  \eqref{x * and x_i1^+}}
	
	\begin{lemma}
		We have \eqref{x * and x_i1^+} holds for any $r,t\in\Z$.
	\end{lemma}
	
	\begin{proof}
		We only prove the first formula.
		
		Since there are no nonzero homomorphisms between $\co(t\vec{c})$ and $S_{i,1}$ for any $t\in\bbZ$ and $1\leq i\leq \bt$, we know that
		\begin{align}\label{equation 9.22}
			\big[[C_{\co(t\vec{c})}], [C_{S_{i,1}}^*]\big]=0,
		\end{align}
		which proves the desired formula for $r=0$.
		
		
		Using \S\ref{sub:embedding-XY}, we assume $p_i>2$. 
		Since $\Omega_{C_{p_i}}$ preserves \eqref{DR4} of $\tUD_v(\widehat{\mathfrak{sl}}_{p_i})$, we have
		\begin{align}
			\label{eq:hi1xi10-}
			[\widehat{\bh}_{[i,2],d_ir},\widehat{X}_{[i,1],0}^{-}] =
			\frac{[r]_{\sqq_i}}{r} [K_{d_i\de}]^{r\cdot\delta\{r>0\}}* \widehat{X}_{[i,1],d_ir}^{-}.
		\end{align}
		We know $[\widehat{X}^+_{\star,t}, \widehat{\bh}_{[i,2],d_ir}]=0$ for any $r\neq0$; see \cite[Lemmas 7.7 and 8.6]{DJX12}.
		Then    
		\begin{align*}
			\big[\widehat{X}^+_{\star,t}, [\widehat{\bh}_{[i,2],d_ir},\widehat{X}_{[i,1],0}^{-}]\big]&=\big[[\widehat{X}^+_{\star,t}, \widehat{\bh}_{[i,2],d_ir}],\widehat{X}_{[i,1],0}^{-}\big]=0.
		\end{align*}
		Together with \eqref{eq:hi1xi10-}, we are done.
	\end{proof}
	
	\subsection{The relations \eqref{Hiqxstar1} and \eqref{Hiqxstar2}}
	
	\begin{lemma}
		We have \eqref{Hiqxstar1} and \eqref{Hiqxstar2} hold for any $m>0$, $t\in\Z$.
	\end{lemma}
	
	\begin{proof}

		We only need to prove \eqref{Hiqxstar1} since the other one is similar. The proof is similar to \cite[Lemma 8.5]{DJX12}.
		By \cite[Lemma 4.1]{Be94}, \eqref{Hiqxstar1} is equivalent to
		\begin{align*}
			\big[\widehat{\varphi}_{[i,1],-d_im}, [C_{\co(t\vec{c})}]\big]=&\sqq_i\big[\widehat{\varphi}_{[i,1],d_i(1-m)},[C_{\co((t-d_i)\vec{c})}]\big]_{\sqq_i^{-2}}*[K_{d_i\de}^*].
		\end{align*}
		
		Since $\Omega_{C_{p_i}}$ preserves \eqref{DR5} of $\tUD_v(\mathfrak{sl}_{p_i})$, we know that $$(q^{d_i}-1)[K^*_{d_i\delta}]^{m}*[K^*_{S_{i,1}}]*\widehat{\varphi}_{[i,1],-d_im}=[\widehat{X}^+_{[i,1],d_im}, \widehat{X}^-_{[i,1],-2d_im}].$$
		Set $\zeta_0^i=q^{d_i}-1$ and $$\zeta_m^i:=(q^{d_i}-1)\widehat{\varphi}_{[i,1],-d_im}=[\widehat{X}^+_{[i,1],d_im}, \widehat{X}^-_{[i,1],-2d_im}]*[K^*_{S_{i,1}}]^{-1}*[K^*_{d_i\delta}]^{-m},\quad \forall m>0.$$
		Then \eqref{Hiqxstar1} is equivalent to
		\begin{align}\label{zeta and co}
			\big[\zeta_m^i, [C_{\co(t\vec{c})}]\big]=&\sqq_i\big[\zeta_{m-1}^i,[C_{\co((t-d_i)\vec{c})}]\big]_{\sqq_i^{-2}}*[K_{d_i\de}^*].
		\end{align}

		For $m=1$, we have
		\begin{align*}
			\zeta_1^i&=(q^{d_i}-1)\widehat{\varphi}_{[i,1],-d_i}\\
			&=[K^*_{S_{i,1}}]^{-1}*[K^*_{d_i\delta}]*\big[\widehat{X}^+_{[i,1],-d_i}, \widehat{X}^-_{[i,1],0}\big]\\
			&=[K^*_{S_{i,1}}]^{-1}*[K^*_{d_i\delta}]*\big[-\sqq_i[K^*_{S_{i,0}^{(p_i-1)}}]^{-1}*[C_{S_{i,0}^{(p_i-1)}}^*], [C^*_{S_{i,1}}]\big]\\
			&=-\sqq_i\big[ [C_{S_{i,0}^{(p_i-1)}}^*], [C^*_{S_{i,1}}] \big]_{\sqq_i^{-2}}.
		\end{align*}
		Here in the last equation, we use $$[C^*_{S_{i,1}}]*[K^*_{S_{i,0}^{(p_i-1)}}]^{-1}=\sqq_i^{-2}[K^*_{S_{i,0}^{(p_i-1)}}]^{-1}*[C^*_{S_{i,1}}].$$
		Hence,
		\begin{align*}
			\big[\zeta_1^i, [C_{\co(t\vec{c})}]\big]
			&=-\Big[ [C_{\co(t\vec{c})}], -\sqq_i\big[ [C_{S_{i,0}^{(p_i-1)}}^*], [C^*_{S_{i,1}}] \big]_{\sqq_i^{-2}}\Big]\\
			&=\sqq_i\Big[\big[ [C_{\co(t\vec{c})}],  [C_{S_{i,0}^{(p_i-1)}}^*]\big], [C^*_{S_{i,1}}] \Big]_{\sqq_i^{-2}}\\
			&=\sqq_i\cdot\Big[ \sqq_i^{-1}(q^{d_i}-1)[K^*_{S_{i,0}^{(p_i-1)}}]*[C_{\co((t-d_i)\vec{c}+\vx_i)}], [C^*_{S_{i,1}}] \Big]_{\sqq_i^{-2}}\\
			&=(q^{d_i}-1) [K^*_{S_{i,0}^{(p_i-1)}}]*\big[[C_{\co((t-d_i)\vec{c}+\vx_i)}], [C^*_{S_{i,1}}] \big]\\
			&=(q^{d_i}-1) [K^*_{S_{i,0}^{(p_i-1)}}]*\big(\sqq_i^{-1}(q^{d_i}-1)[K^*_{S_{i,1}}]*[C_{\co((t-d_i)\vec{c})}]\big)\\
			&=\sqq_i^{-1}(q^{d_i}-1)^2  [K^*_{d_i\delta}]*[C_{\co((t-d_i)\vec{c})}].
		\end{align*}
		This proves \eqref{zeta and co} when $m=1$.
		
		By  using \eqref{x * and x_i1^+} and \eqref{eq:qcommxstarxi1+1}, we obtain
		\begin{align*}
			\big[\zeta_m^i, [C_{\co(t\vec{c})}]\big]=&\sqq_i^{-1}\big[[\widehat{X}^+_{[i,1],d_im}, \widehat{X}^-_{[i,1],-2d_im}], [C_{\co(t\vec{c})}]\big]_{\sqq_i}*[K^*_{S_{i,1}}]^{-1}*[K^*_{d_i\delta}]^{-m}
			\\
			=&\sqq_i^{-1}\big[[\widehat{X}^+_{[i,1],d_im}, [C_{\co(t\vec{c})}]]_{\sqq_i}, \widehat{X}^-_{[i,1],-2d_im}\big]*[K^*_{S_{i,1}}]^{-1}*[K^*_{d_i\delta}]^{-m}
			\\
			=&\big[[\widehat{X}^+_{[i,1],d_i(m+1)}, [C_{\co((t-d_i)\vec{c})}]]_{\sqq_i^{-1}}, \widehat{X}^-_{[i,1],-2d_im}\big]*[K^*_{S_{i,1}}]^{-1}*[K^*_{d_i\delta}]^{-m}
			\\
			=&\big[[\widehat{X}^+_{[i,1],d_i(m+1)}, \widehat{X}^-_{[i,1],-2d_im}], [C_{\co((t-d_i)\vec{c})}] \big]_{\sqq_i^{-1}}*[K^*_{S_{i,1}}]^{-1}*[K^*_{d_i\delta}]^{-m}
			\\
			=&\sqq_i\big[[\widehat{X}^+_{[i,1],d_i(m+1)}, \widehat{X}^-_{[i,1],-2d_im}]*[K^*_{S_{i,1}}]^{-1}*[K^*_{d_i\delta}]^{-m-1}, [C_{\co((t-d_i)\vec{c})}] \big]_{\sqq_i^{-2}}*[K^*_{d_i\delta}]
			\\
			=&\sqq_i\big[\zeta_{m-1}^i, [C_{\co((t-d_i)\vec{c})}] \big]_{\sqq_i^{-2}}*[K^*_{d_i\delta}],
		\end{align*}
		where the last equality follows 
		from the following equation:
		\begin{align*}&(q^{d_i}-1)[K^*_{d_i\delta}]^{m+1}*[K^*_{S_{i,1}}]*\widehat{\varphi}_{[i,1],-d_i(m-1)}=[\widehat{X}^+_{[i,1],d_i(m+1)}, \widehat{X}^-_{[i,1],-2d_im}].\qedhere
		\end{align*}
	\end{proof}

	\subsection{The relation \eqref{eq:hmhi1-l}--\eqref{eq:hi1mh-l}}
	
	\begin{lemma}
		\label{lem:imagecomm}
		We have \eqref{eq:hmhi1-l}--\eqref{eq:hi1mh-l} hold for any $l,m>0$.
	\end{lemma}

	\begin{proof}
		Let us prove \eqref{eq:hmhi1-l}. 
		Similar to \cite[Lemmas 4.4--4.5]{Be94}, \eqref{eq:hmhi1-l} is equivalent to
		\begin{align}
			\label{eq:hmhi1-lreform}
			[\widehat{H}_{[i,1],-d_il},\widehat{\psi}_{\star,m}]=\begin{cases} 
				\frac{[l]_{\sqq_i}}{l} ([K_{d_i\de}]^l-[K_{d_i\de}^*]^l)*\widehat{\psi}_{\star,m-d_il},&\text{ if }m\geq d_il,
				\\
				0&\text{ if }m<d_il,
			\end{cases}
		\end{align}

		Since $\Omega:\tUD_\sqq(\widehat{\mathfrak{sl}_2})\rightarrow \cs\cd\widetilde{\ch}(\P^1_\bfk)$ preserves \eqref{Drsl2-6}, we know that
		\begin{align*}[\widehat{X}_{\star, m+1}^+,\widehat{X}_{\star, -1}^-] = (1-q)[K_\de] *[K_\co]*\widehat{\psi}_{\star,m}.
		\end{align*}
		It follows from \eqref{eq:hrxk+} and \eqref{Hiqxstar1}  that
		\begin{align}
			\label{eq:hi1-lpsimkde}
			&\big[\widehat{\bh}_{[i,1],-d_il}, (1-q)[K_\de] *[K_\co]*\widehat{\psi}_{\star,m}\big]\\\notag
			&=\big[\widehat{\bh}_{[i,1],-d_il}, [\widehat{X}_{\star, m+1}^+,\widehat{X}_{\star, -1}^-]\big]\\\notag
			&=\big[[\widehat{\bh}_{[i,1],-d_il}, \widehat{X}_{\star, m+1}^+],\widehat{X}_{\star, -1}^-\big]+\big[\widehat{X}_{\star, m+1}^+,[\widehat{\bh}_{[i,1],-d_il},\widehat{X}_{\star, -1}^-]\big]\\\notag
			&=\frac{[-l]_{\sqq_i}}{l} [K_{d_i\de}^*]^{l}* \big[\widehat{X}^+_{\star,m-d_il+1},\widehat{X}_{\star, -1}^-\big]+
			\frac{[l]_{\sqq_i}}{l}\big[\widehat{X}_{\star, m+1}^+,  \widehat{X}^-_{\star,-d_il-1} \big].
		\end{align}
		Then the desired formula follows by noting that  \eqref{Drsl2-6} is preserved by $\Omega:\tUD_\sqq(\widehat{\mathfrak{sl}_2})\rightarrow \cs\cd\widetilde{\ch}(\P^1_\bfk)$.
	\end{proof}

	\subsection{The relation \eqref{DR72-Hall}}

	\begin{lemma}
		\label{lem:Serre1}
		We have \eqref{DR72-Hall} hold for any $m_1,m_2,n\in\Z$.
	\end{lemma}
	
	\begin{proof}
		First, we prove it for $m_1=m_2=0$. It is enough to check the following equality in $\cs\cd\widetilde{\ch}(\X_\bfk)$:
		\begin{align}
			\label{serreSSO}
			[C_{S_{i,1}}]*[C_{S_{i,1}}]*[C_{\co(n\vc)}]-[2]_{\sqq_i}[C_{S_{i,1}}]*[C_{\co(n\vc)}]*[C_{S_{i,1}}]+[C_{\co(n\vc)}]*[C_{S_{i,1}}]*[C_{S_{i,1}}]=0.
		\end{align}
		It is equivalent to 
		\begin{align*}
			[C_{S_{i,1}}]*\big[[C_{S_{i,1}}], [C_{\co(n\vc)}]\big]_{\red{\sqq_i^{-1}}}-\sqq_i\big[[C_{S_{i,1}}], [C_{\co(n\vc)}]\big]_{\red{\sqq_i^{-1}}}* [C_{S_{i,1}}]=0,
		\end{align*}
		which follows from the following equalities:
		$$\big[[C_{S_{i,1}}], [C_{\co(n\vc)}]\big]_{\red{\sqq_i^{-1}}}=(\sqq_i-\sqq_i^{-1})[C_{\co(n\vc+\vx_i)}] 
		\text{\; and \;}[C_{\co(n\vc+\vx_i)}]*[C_{S_{i,1}}]=\sqq_i^{-1}[C_{S_{i,1}}]*[C_{\co(n\vc+\vx_i)}].$$

		Note that \eqref{DR4} of $\tUD_v(\widehat{\mathfrak{sl}}_{p_i})$ is preserved by $\Omega_{C_{p_i}}$.  Using \S\ref{sub:embedding-XY}, we can assume $p_i \geq 3$. Then $[\widehat{H}_{[i,2],m},\widehat{X}^\pm_{\star,n}]=0$. By applying $[\widehat{H}_{[i,2],m},-]$ to \eqref{serreSSO}, we obtain 
		\begin{align*}
			\Sym_{m,0}\big\{ \widehat{X}^\pm_{[i,1], m}*\widehat{X}^\pm_{[i,1], 0} *\widehat{X}^\pm_{\star,n}
			-[2]_{\sqq_i}  \widehat{X}^\pm_{[i,1], m}*\widehat{X}^\pm_{\star,n}*\widehat{X}^\pm_{[i,1], 0}
			+\widehat{X}^\pm_{\star,n}*\widehat{X}^\pm_{[i,1], m}*\widehat{X}^\pm_{[i,1], 0}  \big\}=0.
		\end{align*}
		Then by applying $[\widehat{H}_{[i,1],m_2},-]$ to the above equality, the desired formula follows.    
	\end{proof}

	\section{The Serre relation}
	\label{sec:serre}
	
	In this section, we shall verify the Serre relation \eqref{DR7} in $\cs\cd\widetilde{\ch}(\X_\bfk)$, for $m_1,m_2,n\in\Z$:
	\begin{align}\label{DR7-Hall}
		&\Sym_{m_1,m_2}\sum_{t=0}^{d_i-1}\sqq^{d_i-1-2t}\big( \widehat{X}_{[i,1],n}^{\pm}*\widehat{X}_{\star,m_1\pm (d_i-1-t)}*\widehat{X}_{\star,m_2\pm t}^{\pm}-[2]_{\sqq_i}\widehat{X}_{\star,m_1\pm (d_i-1-t)}^{\pm}\\\notag
		&\hspace{3cm}*\widehat{X}_{[i,1],n}^{\pm}*\widehat{X}_{\star,m_2\pm t}^{\pm}+\widehat{X}_{\star,m_1\pm (d_i-1-t)}^{\pm}*\widehat{X}_{\star,m_2\pm t}^{\pm} *\widehat{X}_{[i,1],n}^{\pm} \big)=0.
	\end{align}
	
	Denote the left hand side of \eqref{DR7-Hall} by 
	$\SS(\star,[i,1]\mid m_1,m_2,n)$ for simplicity. Using the same proof of Lemma \ref{lem:Serre1}, it is reduced to 
	\begin{align}
		\label{S000}
		\SS(\star,[i,1]\mid 0,0,0)=0,
	\end{align}
	whose proof occupies the remainder of this section.
	
	We assume $i=1$, denote by $d:=d_1$ in the following.
	Let us prove \eqref{S000}. 
	By using the algebra embedding $R^\pm:\widetilde{\ch}(\X_\bfk)\rightarrow \cs\cd\widetilde{\ch}(\X_\bfk)$, it is enough to prove 
	\begin{align}
		\notag
		\sum_{r=0}^{d-1}\sqq^{d-1-2r}&\Big([S_{11}]*[\co((d-1-r)\vec{c})]*[\co(r\vec{c})]-[2]_{\sqq^d}[\co((d-1-r)\vec{c})]*[S_{11}]*[\co(r\vec{c})] \\
		&+ [\co((d-1-r)\vec{c})]*[\co(r\vec{c})]*[S_{11}]\Big)=0,
	\end{align}
	in $\widetilde{\ch}(\X_\bfk)$, 
	which is equivalent to
	\begin{align}
		\sum_{r=0}^{d-1}\sqq^{d-1-2r}&\Big(\big[[S_{11}], [\co((d-1-r)\vec{c})]\big]_{\sqq^{-d}} *[\co(r\vec{c})]
		\\\notag
		&-\sqq^d[\co((d-1-r)\vec{c})]*\big[[S_{11}],[\co(r\vec{c})]\big]_{\sqq^{-d}}\Big)=0.
	\end{align}
	Since 
	\begin{align*}
		\big[[S_{11}],[\co(m\vec{c})]\big]_{\sqq^{-d}}=\sqq^{-d}(q^d-1) \cdot [\co(m\vec{c}+\vec{x}_1)], \quad \forall m\in\Z,
	\end{align*}
	it reduces to prove that
	\begin{proposition}
		\label{Serre relation in specail case}
		We have in $\widetilde{\ch}(\X_\bfk)$: 
		\begin{align}
			\label{eq:Serre2-reduced}
			\sum_{r=0}^{d-1}\sqq^{d-1-2r}&\Big([\![\co((d-1-r)\vec{c}+\vec{x}_1)]\!] *[\![\co(r\vec{c})]\!]-\sqq^d[\![\co((d-1-r)\vec{c})]\!]*[\![\co(r\vec{c}+\vec{x}_1)]\!]\Big)=0.
		\end{align}   
	\end{proposition}
	
	The proof of Proposition \ref{Serre relation in specail case} occupies the remainder of this section and
	Appendix \ref{Appendix C}.

	\subsection{$\M_k$ and $\cn_k$}\label{Mk and Nk} 
	
	Note that the formula \eqref{eq:Serre2-reduced}  only involves line bundles of the forms $\co(k\vec{c})$ and $\co(k\vec{c}+\vec{x}_1)$ for certain $k$.
	Using \S\ref{sub:embedding-XY}, we only need to consider the weighted projective line $\X$ of weight $p=2$. Therefore, from now on, we always assume $\X$ has weight type $$\left(\begin{array}{ccc}
		2 \\ d 
	\end{array}\right).$$ 

	In order to prove Proposition \ref{Serre relation in specail case}, 
	we need to   
	stratify, in a suitable manner, the middle terms of the extensions between line bundles $\co(\vx_1+(k-1)\vc)$ and $\co((d-k)\vc)$  for any $1\leq k\leq d$.
	For this, we introduce the following notations. 
	
	Denote by $\text{ind}\coh(\X)$ the set of indecomposable objects in $\coh(\X)$. 
	For any $1\leq k\leq d$, set
	\begin{align*}
		\M_k:=\{E\in \text{ind}\coh(\X)\mid  0\to \co(\vx_1+(k-1)\vc)\to E\to \co((d-k)\vc)\to 0 \text{ is exact}\};
	\end{align*} and 
	\begin{align*}
		\cn_k:=\{E\in \text{ind}\coh(\X)\mid 0\to\co((d-k)\vc) \to E\to \co(\vx_1+(k-1)\vc)\to 0 \text{ is exact}\}.
	\end{align*}
	
	
	We have the following key observation. \begin{proposition}
		\label{Mi and Nj properties} Keep notations as above. Then the following statements hold:
		\begin{itemize}
			\item[(1)]  $\M_k\neq\emptyset$ if and only if $1\leq k<d/2$;
			\item[(2)]  $\cn_{k}\subseteq \bigcup_{1\leq j<k} \M_{j}$ for any $1\leq k\leq d$;
			\item[(3)]  Assume $1\leq k, j\leq d/2$, then
			\begin{itemize}
				\item[(i)] $\M_{k}\cap \M_{j}=\emptyset$ for any $k\neq j$;
				\item[(ii)] $\cn_{k}\cap \cn_{j}=\emptyset$ for any $k\neq j$;
				\item[(iii)] $\M_{k}\cap \cn_{j}=\emptyset$ if and only if $k\geq j$.
			\end{itemize}
		\end{itemize}
	\end{proposition}

	
	
	\begin{proof}
		
		(1) By Serre duality we have $\Ext^1(\co((d-k)\vc), \co(\vx_1+(k-1)\vc))\cong \bS_{(d-1-2k)\vc}$. It is nonzero if and only if $d-1-2k\geq 0$, i.e., $1\leq k< d/2$. Moreover,  
		there are no element $\vec{y}$ satisfying $\vx_1+(k-1)\vc\leq \vec{y}\leq (d-k)\vc$ since $(d-k)\vc-(\vx_1+(k-1)\vc)=\vx_1-(2k-1)\vc\not\geq0$. That is, there are no line bundles $\co(\vec{y})$ satisfying $\Hom(\co(\vx_1+(k-1)\vc), \co(\vec{y}))\neq 0$ and $\Hom(\co(\vec{y}), \co((d-k)\vc))\neq 0$.
		Hence the middle term of any non-split extension in $\Ext^1(\co((d-k)\vc), \co(\vx_1+(k-1)\vc))$ must be indecomposable. This proves (1).
		
		(2) For any $E\in\cn_{k}$, there exists an exact sequence 
		$$ 0\longrightarrow \co((d-k)\vc)\longrightarrow E\longrightarrow \co(\vx_1+(k-1)\vc)\longrightarrow 0.$$
		Applying $\Hom(-,\co((d-1)\vc))$ we obtain an exact sequence
		\begin{align*}
			0&\longrightarrow \Hom(E, \co((d-1)\vc)) \longrightarrow \Hom(\co((d-k)\vc), \co((d-1)\vc))\\
			&\longrightarrow \Ext^1( \co(\vx_1+(k-1)\vc), \co((d-1)\vc)).
		\end{align*}
		Since $\dim \Hom(\co((d-k)\vc), \co((d-1)\vc))=k$, and $\dim\Ext^1( \co(\vx_1+(k-1)\vc), \co((d-1)\vc))=k-1$ by using Serre duality, we have  $\Hom(E, \co((d-1)\vc)) \neq 0$. Then for any nonzero $f:E\to \co((d-1)\vc))$, $\Im(f)$ is a nonzero subsheaf of $\co((d-1)\vc)$, which must be a line bundle $\co((d-j)\vc)$ for some $1\leq j<k$, that is, $E\in \M_{j}$.  
		Then we are done.
		
		
		(3) Fix $1\leq k< j\leq d/2$. For any $E\in\M_k$, there exists an exact sequence
		$$0\longrightarrow \co(\vx_1+(k-1)\vc)\longrightarrow E\longrightarrow \co((d-k)\vc)\longrightarrow 0.$$ 
		By applying $\Hom(-, \co((d-j)\vc))$ we obtain $\Hom(E, \co((d-j)\vc))=0$, hence $E\not\in \M_j$. This implies $ \M_{k}\cap \M_{j}=\emptyset$.
		
		Similarly, for any $E\in\cn_j$, there exists an exact sequence $$0\longrightarrow \co((d-j)\vc)\longrightarrow E\longrightarrow \co(\vx_1+(j-1)\vc)\longrightarrow 0.$$ By applying $\Hom(-, \co(\vx_1+(k-1)\vc))$ we obtain $\Hom(E, \co(\vx_1+(k-1)\vc))=0$, hence $E\not\in \cn_k$. This implies $ \cn_{k}\cap \cn_{j}=\emptyset$.
		Then (i) and (ii) hold.
		
		Now we prove  (iii). On the one hand, if $k<j$, then $\M_{k}\cap \cn_{j}\neq \emptyset$ by Lemma \ref{Mi intersect with Nj for i<j}.
		On the other hand, if $ \M_{k}\cap \cn_{j}\neq \emptyset$, then we have the following graph given by two non-split short exact sequences:
		$$\xymatrix{
			&\co(\vec{x}_1+(k-1)\vc) \ar[d]_f &\\
			\co((d-j)\vc)\ar[r]&E \ar[r]\ar[d]&\co(\vx_1+(j-1)\vc)\\
			&\co((d-k))\vc).}
		$$
		Clearly, $k\neq j$ since $\co(\vec{x}_1+(k-1)\vc)$ is a direct summand of $E$. We claim that $\Hom(\co(\vec{x}_1+(k-1)\vc), \co(\vec{x}_1+(j-1)\vc))\neq 0$. Otherwise, the nonzero map $f$ in the above graph factors through $\co((d-j)\vc)$,  which contradicts to $\Hom(\co(\vec{x}_1+(k-1)\vc), \co((d-j)\vc))=0$. Hence $k< j$. 
		Then the proof is finished.
	\end{proof}
	
	\begin{remark}
		According to Proposition \ref{Mi and Nj properties} (2), we obtain that $\cn_1=\emptyset$.
		Moreover, we will prove that  $\cn_k=\emptyset$ if and only if $k=1$ in Corollary \ref{Nk nonempty}.
		For further properties of $\M_k$ and $\cn_k$, we refer to Appendix \ref{Cardinalities}.
	\end{remark}
	
	
	Based on Proposition \ref{Mi and Nj properties}, we define \begin{align}
		\label{def:coN}
		\overline{\cn}=(\bigcup\limits_{1\leq k<d/2} \M_k )\backslash (\bigcup\limits_{2\leq k\leq d/2} \cn_k).
	\end{align}
	Then we have 
	\begin{align}
		\label{small pieces}
		\bigcup\limits_{1\leq k<d/2} \M_k =\overline{\cn}\cup \bigcup\limits_{2\leq k\leq d/2} \cn_k= \bigcup\limits_{1\leq k< d/2} (\M_k\cap \overline{\cn})\cup\bigcup\limits_{1\leq k<j\leq d/2} (\M_k\cap \cn_j).
	\end{align}

	\subsection{Hall numbers}

	In this section we will calculate the Hall numnber 
	$F_{\co(\vec{x}_1+(k-1)\vc),\co((d-k)\vc)}^{E}$ for any $E\in\M_i$ with $1\leq i< d/2$. For this we introduce the following notations.

	For  $1\leq i< j\leq d/2$, $1\le k\leq d$, and $E\in\M_i\cap\cn_j$,  $E'\in\M_i\cap\overline{\cn}$, denote by
	$$\varphi_{i,j}^k:=|\{f:E\to \co(\vx_1+(k-1)\vc)\mid f \text{ is surjective}\}|;$$ and 
	$$\varphi_{i,1}^k:=|\{f: E'\to \co(\vx_1+(k-1)\vc)\mid f \text{ is surjective} \}|.$$
	The notation $\varphi_{i,j}^k$ (resp. $\varphi_{i,1}^k$) makes sense since it does not rely on the choice of $E$ (resp. $E'$); see 
	Remark \ref{independent}. 
	
	Note that the Hall number $F_{\co(\vec{x}_1+(k-1)\vc),\co((d-k)\vc)}^{E}$ is equal to $\frac{1}{q-1}\varphi_{i,j}^k$ or $\frac{1}{q-1}\varphi_{i,1}^k$ for $E\in\M_i\cap\cn_j$ or $E\in\M_i\cap\overline{\cn}$, respectively. 
	Moreover, the explicit formulas of $\varphi_{i,j}^k$'s have been calculated in Lemmas \ref{lem:varphi ijk} and \ref{lem:varphi i1k}.
	Consequently, we have the following result.

	\begin{lemma}\label{Hall number for j=1} 
		Let $1\leq i<j\leq d/2$ and $1\leq k\leq d$. Then
		\begin{itemize}
			\item[(1)] 
			for any $E\in\M_i\cap\cn_j$,
			\begin{align*}
				F_{\co(\vec{x}_1+(k-1)\vc),\co((d-k)\vc)}^{E} =&\begin{cases}
					1 & \text{ if } k=j,
					\\
					q^2 & \text{ if } d=2\mathfrak{t}, k=\mathfrak{t}+1, j=\mathfrak{t},
					\\
					\delta\{k+j\geq d+2\}\cdot q^{2k-d-2}(q^2-1)\\
					\qquad\quad+ \delta_{k+j, d+1}\cdot q^{2k-d}-\delta_{k+i,d+1} & \text{ if } [d/2]+2\leq k\leq d,
					\\
					0 & \text{ else.}
				\end{cases}
			\end{align*} 
			\item[(2)] 
			for any $E\in\M_i\cap\overline{\cn}$, 
			\begin{align*}
				F_{\co(\vec{x}_1+(k-1)\vc),\co((d-k)\vc)}^{E} 
				=&\begin{cases}
					q+1 & \text{ if } d=2\mathfrak{t}+1, k=\mathfrak{t}+1, 
					\\
					q^{2k-d-2}(q^2-1)-\delta_{k+i,d+1} & \text{ if } d=2\mathfrak{t}+1, k\geq \mathfrak{t}+2,
					\\
					0 & \text{ else.}
				\end{cases}
			\end{align*}
		\end{itemize}
	\end{lemma}
	

	\subsection{The product $[\co((d-1-r)\vec{c})]*[\co(r\vec{c}+\vec{x}_1)]$
	}
	
	In this subsection, we compute $[\co((d-1-r)\vec{c})]*[\co(r\vec{c}+\vec{x}_1)]$ for $0\leq r\leq d-1$.
	
	
	\begin{lemma}
		\label{lem:OOx1}
		The middle term $E$ of any exact sequence in $\Ext^1(\co((d-1-r)\vec{c}), \co(r\vec{c}+\vec{x}_1))$ has one of the following forms:
		\begin{itemize}
			\item[(1)] $E=\co((d-1-r)\vec{c})\oplus\co(r\vec{c}+\vec{x}_1)$;
			\item[(2)] $E\in \M_{r+1}$.
		\end{itemize}
		Moreover, in both cases we have
		$$F_{\co((d-1-r)\vec{c}), \co(r\vec{c}+\vec{x}_1)}^{E}=1.$$   
	\end{lemma}
	
	\begin{proof}
		Note that $(d-1-r)\vc-(r\vc+\vec{x}_1)=\vec{x}_1-(2r+1)\vc\not\geq0$.
		So there are no element $\vec{y}$
		satisfying $r\vc+\vec{x}_1\leq \vec{y}\leq (d-1-r)\vc$. Hence, 
		the middle term of any non-split extension in $\Ext^1(\co((d-1-r)\vec{c}), \co(r\vec{c}+\vec{x}_1))$ must be
		indecomposable. Then the first statement follows from the definition of $\M_{r+1}$, and the second one follows from Proposition \ref{ext between O(x) and O(n)}.
	\end{proof}
	
	As an immediate consequence of Lemma \ref{lem:OOx1}, we have
	\begin{align}\notag
		&[\![\co((d-1-r)\vec{c})]\!] *[\![\co(r\vec{c}+\vec{x}_1)]\!]\\
		\label{product one}
		=&\sqq^{2r+2-d}\Big([\![\co((d-1-r)\vec{c})\oplus \co(r\vec{c}+\vec{x}_1)]\!]
		+\sum_{E\in\M_{r+1}}[\![E]\!]\Big).
	\end{align}

	\subsection{The product $[\co((d-1-r)\vec{c}+\vec{x}_1)] *[\co(r\vec{c})]$}
	
	In this subsection, we compute the product $[\co((d-1-r)\vec{c}+\vec{x}_1)] *[\co(r\vec{c})]$ for $0\leq r\leq d-1$.

	\begin{lemma}
		\label{lem:Ox1O}
		The middle term $E$ of any exact sequence in $\Ext^1(\co((d-1-r)\vec{c}+\vec{x}_1), \co(r\vec{c}))$ has one of the following forms:
		\begin{itemize}
			\item[(1)] $E=\co((d-j)\vec{c})\oplus\co((j-1)\vec{c}+\vec{x}_1)$ for $r+1\leq j\leq d-r$;
			\item[(2)] $E\in \M_i\cap \cn_j$ for $1\leq i<j\leq d/2$;
			\item[(3)] $E\in\M_i\cap \overline{\cn}$ for $1\leq i<d/2$.
		\end{itemize}
		Moreover, the Hall number $F_{\co((d-1-r)\vec{c}+\vec{x}_1), \co(r\vec{c})}^{E}$ equals to $\frac{\varphi_{i,j}^{d-r}}{q-1}$ and $\frac{\varphi_{i,1}^{d-r}}{q-1}$ for the cases (2), (3) respectively, and for the case (1),
		\begin{align*}
			F_{\co((d-1-r)\vec{c}+\vec{x}_1), \co(r\vec{c})}^{\co((d-j)\vc)\oplus \co(\vec{x}_1+(j-1)\vc)}  
			=&\begin{cases}q^{\{d-2r\}_+} & \text{ if } j=d-r,\\
				q^{d-2r}-q^{d-2r-2}& \text{ if } r+1<j<d-r,\\
				q^{d-2r}-1& \text{ if } r+1=j<d-r,\\
				0& \text{ else. } \\
			\end{cases}
		\end{align*}
		Here $\{a\}_+:=\max\{a,0\}$ for any $a\in\Z$.
	\end{lemma}
	
	\begin{proof}
		Thanks to Proposition \ref{Mi and Nj properties}, Lemma \ref{Hall number for j=1} and \eqref{small pieces}, we only need to calculate the Hall number 
		$F_{\co((d-1-r)\vec{c}+\vec{x}_1), \co(r\vec{c})}^{\co((d-j)\vc)\oplus \co(\vec{x}_1+(j-1)\vc)}$.
		If $j=d-r$, then it equals to $q^{\{d-2r\}_+}$ 
		since $\Hom(\co(r\vc), \co(\vec{x}_1+(d-1-r)\vc))\cong {\bfk}^{\{d-2r\}_+}$.

		Now assume $j\neq d-r$. Then $F_{\co(\vec{x}_1+(d-1-r)\vc),\co(r\vc)}^{\co((d-j)\vc)\oplus \co(\vec{x}_1+(j-1)\vc)} \neq 0$ if and only if there exists a pushout commutative diagram as follows:
		$$\xymatrix{\co(r\vc)\ar[r]\ar[d]& \co((d-j)\vc)\ar[d]^{f}\\
			\co(\vec{x}_1+(j-1)\vc)\ar[r]_{g}&\co(\vec{x}_1+(d-1-r)\vc).}
		$$
		So we have $r+1\leq j<d-r$. In this case, $f=X_1\cdot f_{j-1-r}$ and $g=g_{d-r-j}$ for some homogeneous polynomials $f_{j-1-r}$ and $g_{d-r-j}$ in $\bfk[T_0,T_1]$ of degree $j-1-r$ and $d-r-j$ respectively.
		Since $d-r-j<d$, we know that $X_1$ and $g=g_{d-r-j}$ are coprime. 
		If $(f,g): \co((d-j)\vc)\oplus \co(\vec{x}_1+(j-1)\vc)\to \co(\vec{x}_1+(d-1-r)\vc$ is surjective, then the rank of $\ker((f,g))$ is $1$, which is a line bundle, and then must be $\co(r\vc)$. 
		Hence the Hall number is given by $$\frac{1}{q-1}|\{(f,g): \co((d-j)\vc)\oplus \co(\vec{x}_1+(j-1)\vc)\to \co(\vec{x}_1+(d-1-r)\vc)\mid (f,g) \text{\,is\, surjective}\}|,$$
		which equals 
		\begin{align*}\frac{\mathfrak{s}(j-1-r, d-r-j)}{q-1}=&\begin{cases}
				q^{d-2r}-q^{d-2r-2}& \text{ if } j-1-r\neq 0,\\
				q^{d-2r}-1& \text{ if } j-1-r=0. 
		\end{cases}\end{align*} 
		Here, $\mathfrak{s}(a,b)$ denotes the cardinality of the set of pairs $(f,g)\in\mathbf{k}[T_0, T_1]$ consisting of non-zero coprime homogeneous polynomials of degree $a$ and $b$ respectively, and the last equality follows from \cite[Lemma 9]{BKa01}.
	\end{proof}
	
	Consequently, we obtain the following formula:
	\begin{align}
		[\![\co((d-1-r)&\vec{c}+\vec{x}_1)]\!] *[\![\co(r\vec{c})]\!]
		=q^{r+1-d}\Big(q^{\{d-2r\}_+}[\![\co((d-1-r)\vec{c}+\vec{x}_1)\oplus \co(r\vec{c})]\!]\notag\\
		&+ \sum_{r+1<j<d-r}(q^{d-2r}-q^{d-2r-2})[\![ \co((d-j)\vec{c})\oplus\co((j-1)\vec{c}+\vec{x}_1)]\!]\notag \\
		&+\delta\{2r+1<d\}\cdot (q^{d-2r}-1)[\![\co((d-1-r)\vec{c})\oplus \co(r\vec{c}+\vec{x}_1)]\!]
		\notag\\ \label{product two}
		&+\sum_{1\leq i<j\leq d/2}\sum_{E\in\M_i\cap \cn_j}\frac{\varphi_{i,j}^{d-r}}{q-1}[\![E]\!]+\sum_{1\leq i<d/2}\sum_{E\in\M_i\cap \overline{\cn}}\frac{\varphi_{i,1}^{d-r}}{q-1}[\![E]\!]\Big).
	\end{align}
	\subsection{Proof of Proposition \ref{Serre relation in specail case}}
	
	
	\begin{proof}[Proof of Proposition \ref{Serre relation in specail case}]
		
		Let us plug the formulas \eqref{product one} and \eqref{product two} into the left hand side of \eqref{eq:Serre2-reduced}. Then \eqref{eq:Serre2-reduced} holds by analysing the coefficients of all terms, see the following.
		
		For any term $[\![\co((d-1-s)\vec{c}+\vec{x}_1)\oplus \co(s\vec{c})]\!]$ with $0\leq s\leq d-1$, its coefficient is 
		\begin{align*}
			&\sqq^{d-1-2s+2s+2-2d+2\{d-2s\}_+} + \sum_{r:r<s,r<d-s-1} \sqq^{d-1-2r+2r+2-2d}(\sqq^{2d-4r}-\sqq^{2d-4r-4})
			\\
			&+\delta(2s+1>d) \sqq^{d-1-2r+2r+2-2d}(\sqq^{4s+4-2d}-1) 
			-\sqq^{d+1}
			\\
			=&\sqq^{1-d+2\{d-2s\}_+} + \sum_{r:r<s,r<d-s-1}\sqq^{1-d} (\sqq^{2d-4r}-\sqq^{2d-4r-4})+\delta(2s+1>d)\sqq^{1-d} (\sqq^{4s+4-2d}-1) \\&
			-\sqq^{d+1}
			=0.
		\end{align*}
		
		For any term $\llbracket E\rrbracket$ with $E\in\cm_a\cap \ov{\cn}$, by Proposition \ref{Mi and Nj properties} and  Lemma 
		\ref{card of Mi and Nj}, we know $1\leq a<d/2$ and $d$ is odd. Assume $d=2t+1$.  
		Then by Lemma \ref{lem:varphi i1k}, its coefficient is 
		\begin{align*}
			& \sum_{r=0}^{2t} \sqq^{d-1-2r} q^{r+1-d} \frac{\varphi_{a1}^{d-r}}{q-1}-\sqq^{d+1}
			=q^{-t}(q+1+\sum_{r=0}^{t-1} q^{d-2r-2}(q^2-1)-1)-q^{t+1}
			=0.
		\end{align*}
		

		For any term $\llbracket E\rrbracket$ with  $E\in\cm_a\cap\cn_b$, where $1\leq a<b\leq d/2$, its coefficient 
		is
		\begin{align*}
			&\sum_{r=0}^{d-1} \sqq^{d-1-2r} q^{r+1-d}\cdot\delta\{b\leq d/2\}\frac{\varphi_{a,b}^{d-r}}{q-1}-\sqq^{d+1-2a}\sqq^d\sqq^{2a-d}\cdot \delta\{a\leq  (d-1)/2\}
			\\
			&=\sum_{r=0}^{d-1} \sqq^{1-d}\frac{\varphi_{a,b}^{d-r}}{q-1}-\sqq^{d+1}.
		\end{align*}
		If $d$ is odd, then by Lemma \ref{lem:varphi ijk}, the coefficient equals to
		\begin{align*}
			& \sqq^{1-d}+\sum_{r=0,r\neq d-b}^{d-1} \sqq^{1-d}\big( q^{d-2r-2}(q^2-1)\delta\{r\leq b-2\} +q^{d-2r}\delta_{r,b-1}-\delta_{r,a-1} \big)-\sqq^{d+1}
			\\
			&= \sqq^{1-d}+\sum_{r=0}^{b-2} \sqq^{1-d}\big( q^{d-2r-2}(q^2-1)\big) +\sqq^{1-d}q^{d-2(b-1)}-\sqq^{1-d}-\sqq^{d+1}
			\\
			&=0.
		\end{align*}
		If $d$ is even, the argument is  similar except $b=d/2$, in which case the coefficient equals 
		\begin{align*}
			&\sqq^{1-d}+q^2\sqq^{1-d}+\sum_{r=0,r\neq b,b-1}^{d-1} \sqq^{1-d}\big( q^{d-2r-2}(q^2-1)\delta\{r\leq b-2\} +q^{d-2r}\delta_{r,b-1}-\delta_{r,a-1} \big)-\sqq^{d+1}
			\\
			&=\sqq^{1-d}+\sqq^{3-d}+\sum_{r=0}^{b-2} \sqq^{1-d}\big( q^{d-2r-2}(q^2-1)\big)  -\sqq^{1-d}-\sqq^{d+1}
			\\
			&=0.
		\end{align*}
		The proof is completed.		
	\end{proof}

	
	
	

	\appendix
	
	\section{Weighted projective lines of weight type 
		$\left(\begin{array}{ccc}
			2 \\ d 
		\end{array}\right)$}\label{Appendix C}

	In this section, we consider the weighted projective line $\bX$ obtained from the projective line by inserting weight 2 to a closed point of degree $d$.
	More precisely, let $\bla_1=(f)$ be a homogeneous prime ideal, where $f$ is an irreducible polynomial in $\bfk[T_0,T_1]$ with $\deg(f)=d$.
	Then $\bX$ is obtained from $\PL$ by attaching the weight $p_1=2$ to the point $\bla_1$.
	In this case, $\bL(\bp,\bd)=\bL(2,d)=\mathbb{Z} \vx_1\oplus \mathbb{Z}\vc/(2\vx_1-d\vc)$, and the dualizing element $$\vw=(d-2)\vc-\vx_1=\vx_1-2\vc.$$

	
	\subsection{Extensions between $\co(n\vc)$ by $\co(\vec{x}_1)$}
	
	In this subsection, we shall study extensions between $\co(n\vc)$ by $\co(\vec{x}_1)$ for $0\leq n\leq d-1$.

	Observe that $\Ext^1(\co(n\vc),\co(\vec{x}_1))\cong \bS_{n\vc-\vx_1+\vw}=\bS_{(n-2)\vc}\neq 0$ if and only if $n\geq 2$. 
	
	
	\begin{proposition}
		\label{ext between O(x) and O(n)}
		For any fixed $n$ with $2\leq n\leq d-1$ and any non-split exact sequence \begin{equation}\label{middle term E}
			\xymatrix{
				0\ar[r]&\co(\vec{x}_1)\ar[r]& E \ar[r]& \co(n\vc)\ar[r]&0,
			}
		\end{equation}
		we have $$ \Hom(E, \co(n\vc))\cong \End(E)\cong {\bf{k}}.$$ Consequently, the middle term $E$ is indecomposable and $F_{\co(n\vc),\co(\vec{x}_1)}^{E}=1$. Moreover, there are exactly $\frac{q^{n-1}-1}{q-1}$ isomorphism classes of such objects $E$.
	\end{proposition}

	
	\begin{proof}
		Firstly, applying $\Hom(\co(\vec{x}_1),-)$ to \eqref{middle term E}, we obtain an exact sequence 
		\begin{equation*}
			\xymatrix{
				0\ar[r]&\Hom(\co(\vec{x}_1),\co(\vec{x}_1)) \ar[r]& \Hom(\co(\vec{x}_1), E) \ar[r]& \Hom(\co(\vec{x}_1),\co(n\vc)). \\
			}
		\end{equation*}
		Since $n\leq d-1$, we have $\Hom(\co(\vec{x}_1),\co(n\vc))=0$, and then 
		$$\Hom(\co(\vec{x}_1), E) \cong \Hom(\co(\vec{x}_1),\co(\vec{x}_1))\cong \bf{k}.$$
		
		Applying $\Hom(\co(n\vec{c}),-)$ to  \eqref{middle term E}, we obtain an exact sequence 
		\begin{align*}				
			\Hom(\co(n\vc),\co(\vec{x}_1))\longrightarrow&\Hom(\co(n\vc), E) \longrightarrow \Hom(\co(n\vc),\co(n\vc)) 
			\stackrel{f}{\longrightarrow}\Ext^1(\co(n\vc),\co(\vec{x}_1)).
		\end{align*}
		Note that $\Hom(\co(n\vc),\co(\vec{x}_1)) =0$ and $\Hom(\co(n\vc),\co(n\vc))\cong\bfk$. The connection map $f$ is nonzero since \eqref{middle term E} is non-split, which implies 
		$\Hom(\co(n\vc), E)=0.$
		
		Now applying $\Hom(-,E)$ to \eqref{middle term E}, we obtain an exact sequence 
		\begin{align*}
			0\longrightarrow\Hom(\co(n\vc),E) \longrightarrow \Hom(E, E) \longrightarrow \Hom(\co(\vec{x}_1),E). 
		\end{align*}
		It follows that $\End(E)\cong \bfk$, and hence $E$ is indecomposable.
		
		Applying $\Hom(-,\co(n\vc))$ to \eqref{middle term E}, we obtain
		$\Hom(E, \co(n\vc))\cong\Hom(\co(n\vc),\co(n\vc))\cong \bfk$.
		Therefore, $$F_{\co(n\vc),\co(\vec{x}_1)}^{E}=\frac{|\{f:E\to \co(n\vc)\,|\, f {\text{\,is\, surjective}}\}|}{|\Aut(\co(n\vc))|}=\frac{q-1}{q-1}=1.$$
		By Riedtmann-Peng formula, we obtain 
		\begin{align*}
			|\Ext^1(\co(n\vc),\co(\vec{x}_1))_E|=F_{\co(n\vc),\co(\vec{x}_1)}^{E} 
			\cdot\frac{|\Aut(\co(n\vc))|\cdot|\Aut(\co(\vec{x}_1))|}{|\Aut(E)|}
			=\frac{(q-1)(q-1)}{q-1}=q-1.
		\end{align*}
		Then the last statement follows from \begin{align*}\dim\Ext^1(\co(n\vc),\co(\vx_1))&=\dim \bS_{n\vc-\vx_1+\vw}=\dim \bS_{(n-2)\vc}=n-1. \qedhere
		\end{align*}
	\end{proof}
	
	Recall that $\cn_1=\emptyset$.
	As a consequence of Proposition \ref{ext between O(x) and O(n)}, we have

	\begin{corollary}
		\label{End of E} 
		For any $E\in\mathcal{M}_k$ with $1\leq k<\frac{d}{2}$, or $E\in\mathcal{N}_k$ with $2\leq k\leq d$,
		we have 
		$\End(E)\cong {\bf{k}}$.
	\end{corollary}
	
	\begin{proof}
		For any $E\in\mathcal{M}_k$, up to degree shift, we obtain that 
		$E((1-k)\vc)$ fits into the exact sequence \eqref{middle term E} with $n=d+1-2k\leq d-1$. Then $\End(E)\cong {\bf{k}}$ by Proposition \ref{ext between O(x) and O(n)}. 
		
		If $E\in\mathcal{N}_k$, then the result follows from Proposition \ref{Mi and Nj properties}.
	\end{proof}

	
	

	\subsection{Cardinalities for $\M_i, \cn_j$ and $\overline{\cn}$}\label{Cardinalities}
	
	Recall $\M_k,\cn_k$  for $1\leq k\leq d$ defined in Section \ref{Mk and Nk},    
	and $\ov{\cn}$ defined in \eqref{def:coN}. 
	
	
	
	\begin{lemma} 
		\label{card of Mi and Nj} Let $1\leq i<j\leq d/2$, and $E\in\M_i, E'\in\cn_j$. Then 
		\begin{itemize}
			\item[(1)]  $\Hom(E, \co((d-i)\vc))\cong\End(E)\cong \bfk$, and  $|\M_i|=\frac{q^{d-2i}-1}{q-1}$;
			\item[(2)] $\Hom(E', \co(\vx_1+(j-1)\vc))\cong \End(E')\cong \bfk$, and $|\cn_j|=\frac{q^{2j-2}-1}{q-1}$.
		\end{itemize}
		Consequently, $\overline{\cn}=\emptyset$ if $d$ is even, and $|\overline{\cn}|=\frac{q^{d-1}-1}{q^2-1}$ if $d$ is odd.
	\end{lemma}
	
	\begin{proof}
		The first two statements follow from Proposition \ref{ext between O(x) and O(n)} and Corollary \ref{End of E}.
		
		If $d$ is even, then 
		\begin{align}
			\sum\limits_{i=2}^{d/2}|\cn_i|=&\sum\limits_{i=2}^{d/2}\frac{q^{2i-2}-1}{q-1}
			=\sum\limits_{i=1}^{d/2-1}\frac{q^{2i}-1}{q-1}
			=\sum\limits_{k=1}^{d/2-1}\frac{q^{d-2k}-1}{q-1}=\sum\limits_{k=1}^{d/2-1}|\M_k|,
		\end{align}
		where we set $k=\frac{d}{2}-i$ in the third equality.
		Hence $\overline{\cn}=\emptyset$ by using \eqref{def:coN}.
		
		If $d$ is odd, then 
		\begin{align}
			\sum\limits_{i=2}^{\frac{d-1}{2}}|\cn_i|=&\sum\limits_{i=2}^{\frac{d-1}{2}}\frac{q^{2i-2}-1}{q-1}
			=\sum\limits_{i=1}^{\frac{d-3}{2}}\frac{q^{2i}-1}{q-1}
			=\sum\limits_{k=1}^{\frac{d-3}{2}}\frac{q^{d-2k-1}-1}{q-1},
		\end{align}
		where we set $k=\frac{d-1}{2}-i$ in the third equality.
		
		By using \eqref{def:coN}, we have \begin{align*}|\overline{\cn}|
			&=\sum\limits_{k=1}^{\frac{d-1}{2}}|\M_k|- \sum\limits_{i=2}^{\frac{d-1}{2}}|\cn_i|
			=1+\sum\limits_{k=1}^{\frac{d-3}{2}}\frac{q^{d-2k}-q^{d-2k-1}}{q-1}
			=1+\sum\limits_{k=1}^{\frac{d-3}{2}}q^{d-2k-1}
			=\frac{q^{d-1}-1}{q^2-1}. \qedhere
		\end{align*}
	\end{proof}

	
	We denote by $\mathbb N (\PL)$ the set of all functions ${\bf m}: \PL \rightarrow \N$ such that ${\bf m}_x \neq 0$ for only finitely many $x\in \PL$. We sometimes write $\bm \in \mathbb N (\PL)$ as ${\bf m} =({\bf m}_x)_{x\in \PL}$, and define \begin{align*}
		|| {\bf m} || :=\sum_{x\in \P^1_{\mathbf{k}}} d_x {\bf m}_x.
	\end{align*}
	For any $k\geq 1$, denote  
	$$\mathcal{S}_{k}:=\{S_{\bf{n}} \,|\,{\bf{n}}\in \mathbb N (\PL), {\bf n}_{\bla_1}=0, || {\bf n} ||=k\}.$$
	Then we can view $\mathcal{S}_{k}$  in $\coh(\X_\bfk)$ via the embedding $\mathbb{F}_{\X,\P^1}: \coh(\P^1_\bfk)\rightarrow\coh(\X_\bfk)$.
	
	

	\begin{lemma}
		\label{Mi intersect with Nj for i<j} 
		For $1\leq i<j\leq d/2$, we have $|\M_{i}\cap \cn_{j}|=\frac{1}{q-1}\sum\limits_{[A]\in\Iso(\mathcal{S}_{j-i})} |\Aut(A)|$.
	\end{lemma}
	
	\begin{proof}
		Consider the following commutative diagram, where all rows and columns are short exact sequences in $\coh(\X)$:		\begin{align}\label{Green commutative diagram}\xymatrix{D\ar[r]\ar[d]	&\co(\vec{x}_1+(i-1)\vc)\ar[r] \ar[d] &B\ar[d]\\
				\co((d-j)\vc)\ar[r]\ar[d]&E \ar[r]\ar[d]&\co(\vx_1+(j-1)\vc)\ar[d]\\
				C\ar[r]&\co((d-i)\vc)\ar[r]&A.}
		\end{align}
		Note that $\rank(B)+\rank(D)=1$ from the exactness of the first row, and both $B$ and $D$ are subsheaves of line bundles. So either $B=0$ or $D=0$.
		Since $1\leq i<j\leq d/2$,  we have $\Hom( \co((d-i)\vc), \co(\vec{x}_1+(j-1)\vc))=0$. Hence $B\neq 0$, and then
		$D=0$, $B=\co(\vec{x}_1+(i-1)\vc)$, $C= \co((d-j)\vc)$ and $A\in\mathcal{S}_{j-i}$.

		By Green's formula \cite{Gr95}, we obtain
		\begin{align}
			\label{Green1}
			&\sum_{[E]\in\Iso(\coh(\X))} F_{ \co((d-i)\vc), \co(\vec{x}_1+(i-1)\vc)}^{E} F_{\co(\vec{x}_1+(j-1)\vc), \co((d-j)\vc)}^{E} \frac{1}{|\Aut(E)|}
			\\\notag
			&=\sum_{[A]\in\Iso(\mathcal{S}_{j-i})} F_{A, \co((d-j)\vc)}^{\co((d-i)\vc)} F_{A, \co(\vx_1+(i-1)\vc)}^{\co(\vx_1+(j-1)\vc)} \frac{|\Aut(A)|}{(q-1)^2}.
		\end{align}
		We focus on nonzero terms on the left-hand side. We claim that $E$ is indecomposable. Otherwise,  the second column of \eqref{Green commutative diagram} splits by Proposition \ref{ext between O(x) and O(n)}. Then
		$\Hom( \co((d-i)\vc), \co(\vec{x}_1+(j-1)\vc))=0$ yields a contradiction to the exactness of the second row. 
		Then by definition, $E\in \M_{i}\cap \cn_{j}$. It follows that $|\Aut(E)|=q-1$ and all the Hall numbers appearing in \eqref{Green1} equal to $1$. Therefore, we have
		\begin{align*}|\M_{i}\cap \cn_{j}|&=\frac{1}{q-1}\sum_{[A]\in\Iso(\mathcal{S}_{j-i})} |\Aut(A)|. 
			\qedhere
		\end{align*}
	\end{proof}

	\subsection{Calculation of $\varphi_{i,j}^k$}\label{A3}
	
	In this subsection, we always assume $1\leq i<j\leq d/2$ and $1\leq k\leq d$. 
	

	

	
	
	
	\begin{lemma}\label{Hom from E to O(x1)}
		For any $2\leq  j\leq d/2\leq k\leq d$ and any $E\in\cn_j$,  
		we have  \begin{align*}\dim\Hom(E, \co(\vx_1+(k-1)\vc))=k-j+1+\{k+j-d\}_+ 
			=&\begin{cases}k-j+1 & \text{ if } k+j< d,\\
				2k-d+1& \text{ else. } \\
			\end{cases}
		\end{align*}
	\end{lemma}
	
	\begin{proof}
		By assumption, $E$ fits into the following exact sequence: $$0\longrightarrow \co((d-j)\vc)\longrightarrow E\longrightarrow \co(\vx_1+(j-1)\vc)\longrightarrow 0.$$ 
		Applying $\Hom(-, \co(\vx_1+(k-1)\vc))$, we obtain an exact sequence: 
		\begin{align*}
			0&\longrightarrow \Hom(\co(\vx_1+(j-1)\vc),  \co(\vx_1+(k-1)\vc))\longrightarrow\Hom(E, \co(\vx_1+(k-1)\vc))\\
			&\longrightarrow\Hom(\co((d-j)\vc), \co(\vx_1+(k-1)\vc))\longrightarrow 0,
		\end{align*}
		since $\Ext^1(\co(\vx_1+(j-1)\vc), \co(\vx_1+(k-1)\vc))=0$ with the help of Serre duality.
		Then the result follows from $$\dim\Hom(\co(\vx_1+(j-1)\vc), \co(\vx_1+(k-1)\vc))=k-j+1$$ and 
		\begin{align*}\dim\Hom(\co((d-j)\vc), \co(\vx_1+(k-1)\vc))&=\{k+j-d\}_+.\qedhere
		\end{align*}
	\end{proof}
	
	For any $x \in \R$, we denote by $\lfloor x\rfloor$ the largest integer not exceeding $x$.
	
	\begin{lemma} 
		\label{lem:varphi ijk} 
		For $1\leq i<j\leq d/2$ and $1\leq k\leq d$, we have
		\begin{align}
			\varphi_{i,j}^k=&\begin{cases} 
				\delta_{j,k}\cdot (q-1), &\text{ if }1\leq k\leq d/2,\\
				\delta_{j,d/2}\cdot q^2(q-1), &\text{ if }k=\lfloor d/2\rfloor+1,\\ 
				\delta(k+j\geq d+2) \cdot q^{2k-d-2}(q^3-q^2-q+1) \\\qquad +\delta_{k+j,d+1}\cdot q^{2k-d}(q-1) -\delta_{k+i,d+1}\cdot (q-1)
				,&\text{ if }k\geq \lfloor d/2\rfloor+2. \\
			\end{cases}
		\end{align}
	\end{lemma}
	
	\begin{proof}
		Fix any object $E\in\M_i\cap \cn_j$. First we consider $1\leq k\leq d/2$. Then $\varphi_{i,j}^k\neq 0$ means there exists an epimorphism $f:E\rightarrow \co(\vx_1+(k-1)\vc))$, which implies $E\in \cn_{k}$ by definition of $\cn_k$. Hence $k=j$ by Proposition \ref{Mi and Nj properties}. In this case, $\varphi_{i,j}^k=q-1$ by Lemma \ref{card of Mi and Nj}.

		In the following, we assume $j\leq d/2<k\leq d$. 
		
		For any non-surjective map $0\neq f:E\rightarrow \co(\vx_1+(k-1)\vc)$, since $E\in\M_i\cap \cn_j$, by Proposition \ref{Mi and Nj properties} we know that the image $\Im(f)$ belongs to the set 
		$$\{\co((d-i)\vc)),\co(\vx_1+(j-1)\vc)), \co(\vx_1+(\ell-1)\vc))\mid  \lfloor d/2\rfloor<\ell<k\}.$$
		By Lemma \ref{card of Mi and Nj} we have $$\Hom(E, \co((d-i)\vc))\cong \bfk \quad
		\text{and}\quad \Hom(\co((d-i)\vc), \co(\vx_1+(k-1)\vc))\cong \bfk^{\{k+i-d\}_{+}}.$$ Hence 
		$$
		|\{f:E\to \co(\vx_1+(k-1)\vc)\mid \Im(f)\cong \co((d-i)\vc))\}|=\frac{(q-1)(q^{\{k+i-d\}_{+}}-1)}{q-1}=q^{\{k+i-d\}_{+}}-1.
		$$ 
		Similarly, 
		$$|\{f:E\to \co(\vx_1+(k-1)\vc))\mid \Im (f)\cong \co(\vx_1+(j-1)\vc))\}|=q^{k-j+1}-1,$$
		$$|\{f:E\to \co(\vx_1+(k-1)\vc))\mid \Im (f)\cong \co(\vx_1+(\ell-1)\vc)) \}|=\frac{q^{k-\ell+1}-1}{q-1}\varphi_{ij}^\ell.$$
		
		By Lemma \ref{Hom from E to O(x1)} we obtain a recursive formula for the number of surjective maps:
		\begin{align}\notag
			\varphi_{i,j}^k&=(q^{k-j+1+\{k+j-d\}_+}-1)-(q^{\{k+i-d\}_{+}}-1)-(q^{k-j+1}-1)-\sum_{\lfloor d/2\rfloor<\ell<k}\frac{q^{k-\ell+1}-1}{q-1}\varphi_{ij}^{\ell}\\\label{ph_iijk recursion}
			&=q^{k-j+1+\{k+j-d\}_+}-q^{\{k+i-d\}_{+}}-q^{k-j+1}+1-\sum_{\lfloor d/2\rfloor<\ell<k}\frac{q^{k-\ell+1}-1}{q-1}\varphi_{ij}^\ell.
		\end{align}
		We consider the following three cases.
		
		\underline{{\bf Case 1}: $k=\lfloor d/2\rfloor+1$}. In this case, 
		\begin{align*}
			\varphi_{i,j}^k=q^{k-j+1+\{k+j-d\}_+}-q^{\{k+i-d\}_{+}}-q^{k-j+1}+1.
		\end{align*}
		
		If $d$ is even, then 
		\begin{align*}
			\varphi_{ij}^{d/2+1}=\delta_{j,d/2}\cdot (q^3-q^{2}).
		\end{align*}

		If $d$ is odd, then 
		\begin{align*}
			\varphi_{ij}^{(d+1)/2}=0.
		\end{align*}

		\underline{{\bf Case 2}: $k=\lfloor d/2\rfloor+2$}. In this case, 
		\begin{align*}
			\varphi_{i,j}^k=
			q^{k-j+1+\{k+j-d\}_+}-q^{\{k+i-d\}_+}-q^{k-j+1}+1-(q+1)\varphi_{ij}^{\lfloor d/2\rfloor+1}.
		\end{align*}
		
		If $d$ is even, then 
		
		\begin{align*}
			\varphi_{ij}^{d/2+2}=\delta_{i+1,d/2}(1-q)+\begin{cases}
				q^5-q^4-q^3+q^2,&\text{ if }j=d/2,
				\\
				q^5-q^4,&\text{ if } j= d/2-1,
				\\
				0,&\text{ if } j\leq d/2-2.
			\end{cases}
		\end{align*}

		If $d$ is odd, then 
		\begin{align*}
			\varphi_{ij}^{\frac{d+3}{2}}=\delta_{j,\frac{d-1}{2}}(q^4-q^{3})-\delta_{i,\frac{d-1}{2}}(q-1).
		\end{align*}
		
		\underline{{\bf Case 3}: $k>\lfloor d/2\rfloor+2$}.
		In this case, \eqref{ph_iijk recursion} can be reformulated as
		\begin{align}\label{reform 1 for phi_ijk}
			\sum_{\lfloor d/2\rfloor<\ell\leq k}\frac{q^{k-\ell+1}-1}{q-1}\varphi_{ij}^\ell=q^{k-j+1+\{k+j-d\}_+}-q^{k-j+1}+1-q^{\{k+i-d\}_+}.
		\end{align}
		Replacing $k$ by $k-1$, we obtain
		\begin{align}\label{reform 2 for phi_ijk}
			\sum_{\lfloor d/2\rfloor<\ell\leq k-1}\frac{q^{k-\ell}-1}{q-1}\varphi_{ij}^\ell=q^{k-j+\{k-1+j-d\}_+}-q^{k-j}+1-q^{\{k+i-d-1\}_+}.
		\end{align}
		Then $\eqref{reform 1 for phi_ijk}-q\cdot \eqref{reform 2 for phi_ijk}$ yields
		\begin{align*}
			\sum_{\lfloor d/2\rfloor<\ell\leq k} \varphi_{ij}^\ell= q^{k-j+1+\{k+j-d\}_+}-q^{k-j+1+\{k-1+j-d\}_+}-q+1-q^{\{k+i-d\}_+}+q^{\{k+i-d-1\}_++1}.
		\end{align*}
		Therefore,
		\begin{align*}
			\varphi_{i,j}^k=&q^{k-j+1}\big(q^{\{k+j-d\}_+}-q^{\{k+j-d-1\}_+}-q^{-1+\{k+j-d-1\}_+}+q^{-1+\{k+j-d-2\}_+}\big)
			\\
			&- \big(q^{\{k+i-d\}_+}-q^{\{k+i-d-1\}_++1}-q^{\{k+i-d-1\}_+}+q^{\{k+i-d-2\}_++1}\big)
			\\
			=& \delta(k+j\geq d+2) \cdot q^{2k-d-2}(q^2-1)(q-1) +\delta_{k+j,d+1}\cdot q^{2k-d}(q-1) -\delta_{k+i,d+1}\cdot (q-1).
		\end{align*}
		
		This finishes the proof.
	\end{proof}



	
	
	
	
	\subsection{Calculation of $\varphi_{i,1}^k$}	
	\label{A4}	
	
	In this subsection, we always assume $1\leq i< d/2$ and $1\leq k\leq d$. 
	

	

	\begin{lemma}\label{Hom from  overline{N} to O(x1)}
		For any $d/2< k\leq d$ and any $E\in \overline{\cn}$, we have  $$\dim\Hom(E, \co(\vx_1+(k-1)\vc))=2k-d+1.$$
	\end{lemma}
	
	\begin{proof}
		By assumption, $E$ fits into the following exact sequence for some $1\leq i<d/2$: 
		$$0\longrightarrow \co(\vx_1+(i-1)\vc)\longrightarrow E\stackrel{g}{\longrightarrow} \co((d-i)\vc)\longrightarrow0.$$ 
		Applying $\Hom(-, \co(\vx_1+(k-1)\vc))$,  we obtain a long exac sequence: 
		\begin{align}
			\notag
			&0\longrightarrow\Hom(\co((d-i)\vc), \co(\vx_1+(k-1)\vc))\stackrel{\theta_1}{\longrightarrow}  \Hom(E, \co(\vx_1+(k-1)\vc)) \longrightarrow \\\notag
			&\Hom(\co(\vx_1+(i-1)\vc), \co(\vx_1+(k-1)\vc))\stackrel{\theta_2}{\longrightarrow}  \Ext^1(\co((d-i)\vc), \co(\vx_1+(k-1)\vc)).
		\end{align}
		Note that 
		\begin{align}&\dim\Hom(\co((d-i)\vc), \co(\vx_1+(k-1)\vc))=\{k+i-d\}_+, \label{dim line 1}
			\\
			&\dim\Hom(\co(\vx_1+(i-1)\vc), \co(\vx_1+(k-1)\vc))=k-i+1>0,\label{dim line 2}
			\\
			&\dim\Ext^1(\co((d-i)\vc), \co(\vx_1+(k-1)\vc))=\{d-k-i\}_+.\label{dim line 3}
		\end{align}
		Then $(k-i+1)-(d-k-i)=2k-d+1\geq 1$ implies $\Hom(E, \co(\vx_1+(k-1)\vc))\neq 0$.
		
		We claim that there exists  $ f:E\rightarrow \co(\vx_1+(k-1)\vc)$, such that $\Im (f)\cong \co(\vx_1+(j-1)\vc)$ for some $d/2<j\leq k$.
		Otherwise, for any 
		$f:E\rightarrow \co(\vx_1+(k-1)\vc)$, 
		$\Im (f)$ has the form  $\co((d-s)\vc)$ for some $1\leq s\leq d$. Then $E\in\cm_i\cap\cm_s$, which forces $s=i$ by Lemma 10.2. Note that $\Hom(E, \co(\vx_1+(d-i)\vc))\cong \mathbf{k}$ by Lemma \ref{card of Mi and Nj}.  Therefore, $f$ factors through $g$. This implies $\theta_1$ is surjective and hence an isomorphism. Then from \eqref{dim line 1} we get $k+i\geq d+1$, and then \eqref{dim line 3} vanishes, yielding a contradiction to 
		\eqref{dim line 2} since $\theta_2$ is injective. This proves the claim.
		
		
		Therefore, we obtain an exact sequence for some $d/2<j\leq k$:
		$$ 0\longrightarrow \co((d-j)\vc)\longrightarrow E\longrightarrow \co(\vx_1+(j-1)\vc)\longrightarrow 0.$$
		It follows that 
		\begin{align*}&\dim\Hom(E, \co(\vx_1+(k-1)\vc))\\
			=&\dim\Hom(\co((d-j)\vc), \co(\vx_1+(k-1)\vc))+\dim\Hom(\co(\vx_1+(j-1)\vc), \co(\vx_1+(k-1)\vc))\\
			=&(k+j-d)+(k-j+1)\\
			=&2k-d+1.\qedhere
		\end{align*}
	\end{proof}
	
	\begin{lemma}
		\label{lem:varphi i1k}
		If $d$ is even or $1\leq k\leq d/2$, then $\varphi_{i,1}^k=0$.  Otherwise,
		\begin{align}
			\varphi_{i,1}^k=&\begin{cases}  
				q^2-1, &\text{ if }k=\lfloor d/2\rfloor+1,\\    
				q^{2k-d-2}(q^{2}-1)(q-1)-\delta_{k+i,d+1}\cdot (q-1),&\text{ if }k\geq \lfloor d/2\rfloor+2.
			\end{cases}
		\end{align}
	\end{lemma}
	
	\begin{proof}
		If $d$ is even, then by Lemma \ref{card of Mi and Nj} we have $\overline{\cn}=\emptyset$; if $1\leq k\leq d/2$, then $N_{k}\cap \overline{\cn}=\emptyset$; in both cases we have  $\varphi_{i,1}^k=0$. 
		
		In the following we assume $d$ is odd and $d/2< k\leq d$. Write $d=2\mathfrak{t}+1$, then $\mathfrak{t}<k\leq d$.
		For any non-surjective map $0\neq f:E\rightarrow \co(\vx_1+(k-1)\vc)$, since $E\in\M_i\cap \overline{\cn}$, by Proposition \ref{Mi and Nj properties} we know that the image $\Im(f)$ can only be in  $$\{\co((d-i)\vc))\;  \;\co(\vx_1+(\ell-1)\vc)) \mid \mathfrak{t}<\ell<k\}.$$
		By Lemma \ref{Hom from  overline{N} to O(x1)} we obtain the following recursive formula for the number of surjective maps:
		\begin{align}\label{recursion for phi_i1}
			\varphi_{i,1}^k=q^{2k-d+1}-q^{\{k+i-d\}_+}-\sum_{t<l<k}\frac{q^{k-\ell+1}-1}{q-1}\varphi_{i1}^\ell.
		\end{align}

		
		We consider the following three cases.
		
		\underline{{\bf Case 1}: $k=\mathfrak{t}+1$}. In this case, $2k-d+1=2$ and $k+i\leq d$. Hence
		\begin{align*}
			\varphi_{i1}^{\mathfrak{t}+1}=q^{2k-d+1}-q^{\{k+i-d\}_+}=q^2-1.
		\end{align*}
		
		

		\underline{{\bf Case 2}: $k=\mathfrak{t}+2$}. In this case, $2k-d+1=4$, $k+i\leq d$ for $1\leq i<t$, and $k+i=d+1$ for $i=\mathfrak{t}$. Hence
		\begin{align*}
			\varphi_{i1}^{\mathfrak{t}+2}
			=&q^{2k-d+1}-q^{\{k+i-d\}_+}-(q+1)\varphi_{i1}^{\mathfrak{t}+1}
			\\
			=&\begin{cases}
				q^4-q^3-q^2+q, & \text{ if } 1\leq i<\mathfrak{t},\\
				q^4-q^3-q^2+1,& \text{ if } i=\mathfrak{t},
			\end{cases}
			\\
			=& q^4-q^3-q^2+q-\delta_{i,\mathfrak{t}}\cdot(q-1).
		\end{align*}

		\underline{{\bf Case 3}: 
			$k>\mathfrak{t}+2$}. In this case, we have \eqref{recursion for phi_i1} can be reformulated as
		\begin{align}\label{reform 1 for phi_i1}
			\sum_{\mathfrak{t}<\ell\leq k}\frac{q^{k-\ell+1}-1}{q-1}\varphi_{i1}^\ell=q^{2k-d+1}-q^{\{k+i-d\}_+}.
		\end{align}
		Replace $k$ by $k-1$, we obtain
		\begin{align}\label{reform 2 for phi_i1}
			\sum_{t<\ell\leq k-1} \frac{q^{k-\ell}-1}{q-1}\varphi_{i1}^\ell=q^{2k-d-1}-q^{\{k-1+i-d\}_+}.
		\end{align}
		Then $\eqref{reform 1 for phi_i1}-q\cdot\eqref{reform 2 for phi_i1}$ yields


		\begin{align*}
			\sum_{\mathfrak{t}<\ell\leq k}\varphi_{i1}^\ell =q^{2k-d+1}-q^{\{k+i-d\}_+}- q^{2k-d}+q^{\{k-1+i-d\}_++1}.
		\end{align*}
		Therefore, 
		\begin{align*}
			\varphi_{i,1}^k=&q^{2k-d-2}(q^3-q^2-q+1)+ \big(q^{\{k+i-d-1\}_++1}-q^{\{k+i-d\}_+}-q^{\{k+i-d-2\}_++1}+q^{\{k+i-d-1\}_+}\big)
			\\
			=&q^{2k-d-2}(q^{2}-1)(q-1)-\delta_{k+i,d+1}\cdot (q-1).
		\end{align*}
		This finishes the proof.
	\end{proof}

	As a consequence, we obtain the following corollary.
	
	\begin{corollary}
		\label{Nk nonempty} 
		For any $1\leq k\leq d$,
		$\cn_k=\emptyset$ if and only if $k=1$.
	\end{corollary}

	
	\begin{proof} By Proposition \ref{Mi and Nj properties} we have $\cn_1=\emptyset$. 
		If $2\leq k\leq d/2$, then $\cn_k\neq\emptyset$ by Lemma \ref{card of Mi and Nj}.
		
		For $\lfloor d/2\rfloor+2\leq k\leq d$, we take $1\leq i<j=d+1-k$. Then by Lemma \ref{lem:varphi ijk} we have $\varphi_{i,j}^k=q^{d+2-2j}(q-1)\neq 0$. By definition there is a surjective map 
		$f:E\to \co(\vx_1+(k-1)\vc)$ for any $E\in\M_i\cap\cn_j$. Note that $\ker (f)\cong \co((d-k)\vc)$ in this case. Hence $E\in\cn_k$ and then $\emptyset\neq\M_i\cap\cn_j\subseteq \cn_k$ by Proposition \ref{Mi and Nj properties}.
		
		Now assume $k=\lfloor d/2\rfloor+1$. If $d$ is even, we take $1\leq i<j=d/2$, then by Lemma \ref{lem:varphi ijk} we have $\varphi_{i,j}^k=q^3-q^{2}\neq 0$, hence $\emptyset\neq\M_i\cap\cn_j\subseteq \cn_k$ by similar arguments as above; if $d$ is odd, then by Lemma \ref{lem:varphi i1k} we have $\varphi_{i,1}^k=q^2-1\neq 0$ for any $1\leq i<d/2$. In this case, we have $\emptyset\neq\M_i\cap\overline{\cn} \subseteq\cn_k$ by Lemma \ref{card of Mi and Nj}.
	\end{proof} 
	
	\begin{remark}
		\label{independent}
		From Subsections \ref{A3} and \ref{A4}, we see that $\varphi_{i,j}^k$ and  $\varphi_{i,1}^k$ are independent with the choices of $E$ and $E'$ respectively, which justify the notations.
	\end{remark}


\end{document}